%% file: sugar_hal.tex
\documentclass[11pt]{article}

\usepackage{amsmath,amsthm,amssymb,dsfont}
\usepackage{cite}
\usepackage{algpseudocode}
\usepackage{graphicx}
\usepackage{subfigure}
\usepackage{array}
\usepackage{tabularx}
\usepackage{indentfirst}
\usepackage{xspace}
\usepackage{xcolor}
\usepackage{defs}
\usepackage{float}
\usepackage{bm}
\usepackage{centernot}
\usepackage{dsfont}
\usepackage{nth}
\usepackage{enumitem}
\usepackage{url}
\usepackage[left=3cm,right=3cm,top=2cm,bottom=2.5cm,]{geometry}

\DeclareMathOperator{\diag}{diag}

\graphicspath{{images/}}

\title{
  Stein Unbiased GrAdient estimator of the Risk (SUGAR)\\
  for multiple parameter selection%
  \thanks{This work has been supported by the European Research Council (ERC project SIGMA-Vision)}
  }

\author{Charles-Alban Deledalle\footnotemark[2]\;\footnotemark[5] \and
        Samuel Vaiter\footnotemark[3] \and
        Jalal Fadili\footnotemark[4] \and
        Gabriel Peyr\'e\footnotemark[3]
}

\begin{document}

\sloppy

\maketitle


\renewcommand{\thefootnote}{\fnsymbol{footnote}}

\footnotetext[2]{IMB, CNRS-Universit\'e Bordeaux, Bordeaux, France.}
\footnotetext[3]{CEREMADE, CNRS-Paris-Dauphine, Paris, France.}
\footnotetext[4]{GREYC, CNRS-ENSICAEN-Universit\'e de Caen, France.}
\footnotetext[5]{Part of this work was completed while the first author was at CEREMADE\footnotemark[3].}

{\footnotesize
\begin{abstract}
Algorithms to solve variational regularization of ill-posed inverse problems usually involve operators that depend on a collection of continuous parameters.
When these operators enjoy some (local) regularity, these parameters
can be selected using the so-called Stein Unbiased Risk Estimate (SURE).
While this selection is usually performed
by exhaustive search,
we address in this work the problem of using
the SURE to efficiently optimize for a collection of
continuous parameters of the model.
When considering non-smooth regularizers, such as the popular
$\ell_1$-norm corresponding to soft-thresholding mapping,
the SURE is a discontinuous function
of the parameters preventing the use of gradient descent
optimization techniques. Instead, we focus on an
approximation of the SURE based on finite differences
as proposed in \cite{ramani2008montecarlosure}.
Under mild assumptions on the estimation mapping, we show that this approximation is a weakly differentiable
function of the parameters and its weak gradient,
coined the Stein Unbiased GrAdient estimator of the Risk (SUGAR),
provides an asymptotically (with respect to the data dimension)
unbiased estimate of the gradient of the risk.
Moreover, in the particular case of soft-thresholding,
it is proved to be also a consistent estimator.
This gradient estimate can then be used as a basis to perform a quasi-Newton optimization.
The computation of the SUGAR
relies on the closed-form (weak) differentiation of the non-smooth function.
We provide its expression for a large class of
iterative methods including proximal splitting ones and apply our strategy to
regularizations involving non-smooth convex structured penalties.
Illustrations on various image
restoration and matrix completion problems are given.
\end{abstract}

}

\pagestyle{myheadings}
\thispagestyle{plain}
\markboth{}{}


\input{sections/intro}

\input{sections/risk_estimation}
\input{sections/risk_optimization}

\input{sections/differentiation}

\input{sections/numeric}

\section{Conclusion}~

We have proposed a methodology for optimizing
multiple continuous parameters
of a weakly differentiable estimator that attempts solving
a linear ill-posed
inverse problem contaminated by additive white Gaussian noise.
The proposed method selects the parameters minimizing an estimate of
the risk and is driven by an estimate of its gradient.
Classical unbiased estimators of the risk are generally
non-continuous functions of the parameters, so that, their local
variations cannot be used to estimate the gradient of the risk.
These estimators require estimating the degree of freedom by
evaluating the variations of the estimator with respect to the observations.
We have shown that estimating the degree of freedom by finite differences
leads to a weakly differentiable risk estimator. By carefully choosing
the finite differences step and by computing explicitly the (weak) gradient of this estimate,
an asymptotically unbiased estimator of the gradient of the risk is obtained.
This estimator is numerically smooth enough to apply a quasi-Newton method.
An explicit strategy to compute this (weak) gradient is given for
a large class of (iterative) weakly differentiable algorithms.
We exemplified our methodology on several popular proximal splitting methods.
Numerical experiments have demonstrated the wide applicability and scope of the approach.

Our choice of the finite differences step size was essentially guided by
a careful analysis of the soft-thresholding estimator.
Choosing this step size with theoretical guarantees (such as consistency or optimality) in more general cases remains an open question. 
Beyond consistency and optimality, the question of quantifying
the influence of the finite differences step on the smoothness of the risk gradient estimates and then
on the performance of quasi-Newton methods is still open.
To deal with parameter space of higher dimensions,
other accumulation Jacobian strategies
could be explored following \cite{griewank2008evaluating}.
Improvements could also be achieved on the settings of the quasi-Newton methods.
In particular, a drawback of our approach is the sensitivity to
local minima of the risk with respect to the collection
of parameters. In some settings, more elaborated
optimization strategies could be employed.
Future work could also focus on the extensions to non-weakly differentiable
estimators and/or inverse problems with non-Gaussian noises.

\begin{appendix}
\input{sections/proofs/proofs_risk_optimization}
\input{sections/proofs/proofs_extra}
\input{sections/proofs/proofs_differentiation}

\end{appendix}

\bibliographystyle{siam}
\bibliography{biblio}

\end{document}

%% file: sections/intro.tex
\section{Introduction}

In this paper, we consider the recovery problem of a signal $x_0 \in \Xx$
(where $\Xx = \RR^N$ or is a suitable finite-dimensional Hilbert space
that can be identified to $\RR^N$)
from a realization $y \in \Yy = \RR^P$ of the normal random vector
\begin{align}\label{ew:forward_model}
  Y = \mu_0 + W
  \qwithq
  \mu_0 = \Phi x_0
\end{align}
where $W \sim \Nn(0,\si^2\Id_P)$, and the linear imaging operator $\Phi: \Xx \to \Yy$ entails some loss of information. Typically, $P=\dim(\Yy)$ is smaller than $N=\dim(\Xx)$, or $\Phi$ is rank-deficient, and the recovery problem is ill-posed.

Let $(y, \theta) \mapsto x(y,\theta)$ be some recovery mapping, possibly multivalued, which attempts to approach $x_0$ from a given realization $y \in \Yy$ of $Y$ and is parametrized by a collection of continuous parameters $\theta \in \Theta$. Throughout, $\Theta$ is considered as a subset of a linear subspace of dimension $\dim(\Theta)$. We also denote $\mu(y, \theta)=\Phi x(y,\theta) \in \Yy$ and assume in the rest of the paper that it is always a single-valued mapping though $x(y,\theta)$ may not.

Depending on the smoothness of the mapping $y \mapsto \mu(y, \theta)$, the recovered 
estimate enjoys different regularity properties. For instance, $\mu(y,\theta)$ can be built by solving a variational problem with some regularizing penalty parametrized by $\theta$ (see the example in \eqref{eq:optim_energy}, as well as Section~\ref{sec:formal_diff} and \ref{sec:numeric}). This regularization is generally chosen so as to preserve/promote the interesting structure underlying $x_0$, e.g.~singularities, textures, etc. Also, depending on its choice and that of the data fidelity, the resulting mapping $y \mapsto \mu(y, \theta)$ may be smooth or not. To cover most of these situations, throughout the paper, we will assume that $(y, \theta) \mapsto \mu(y, \theta)$ is {\textit{weakly differentiable}} with respect to both the observation $y$, and the collection of parameters $\theta$.

Recall that for a locally integrable function $f: a \in \Omega \mapsto \RR$, $\Omega$ is an open subset of $\RR^N$, its weak partial derivative with respect to $a_i$ in $\Omega$ is the locally integrable function $g_i$ on $\Omega$ such that
\[
\int_{\Omega} g_i(a) \varphi(a) \d a = - \int_{\Omega} f(a) \frac{\partial \varphi(a)}{\partial a_i} \d a
\] 
holds for all functions $\varphi \in C_c^1(\Omega)$, i.e.~the space of continuously differentiable functions of compact support. The weak partial derivative, if it exists, is uniquely defined Lebesgue-almost everywhere (a.e.). Thus we write
\[
g_i = \frac{\partial f}{\partial a_i}
\]
and all such pointwise relations involving weak derivatives will be accordingly understood to hold Lebesgue-a.e. A function is said to be weakly differentiable if all its weak partial derivatives exist. Similarly, a vector-valued function $h: a \in \RR^N \mapsto h(a)=(h_1(a),\ldots,h_P(a)) \in \RR^P$ is weakly differentiable if $h_k(a)$ is weakly differentiable $\forall k \in \{1,\ldots,P\}$, and we will denote $\JACs{}{h}(a)$ its weak Jacobian, and $\GRADs{}{g}(a) = \JACs{}{h}(a)^*$ its adjoint. Remark that weak differentiation concepts boil down to the classical ones when the considered function is $\Cc^1$. A comprehensive account on weak differentiability can be found in e.g.~\cite{EvansGariepy92,GilbargTrudinger98}.

Getting back to the estimator $x(y, \theta)$, we now discuss some typical examples covered in this paper.
\begin{itemize}[leftmargin=*]
\item Given $(y, \theta)$, consider a minimizer of a convex variational problem of the form
\begin{align}\label{eq:optim_energy}
  x(y, \theta) = \uArgmin{x \in \Xx} \{E(x,y,\theta) = H(y,\Phi x) + R(x, \theta)\}
\end{align}
where $x(y, \theta)$ is the set of minimizers of $x \mapsto E(x,y,\theta)$ which is
considered nonempty (the minimizer may not be unique but is assumed to exist). 
The data fidelity term $x \mapsto H(y,\Phi x)$ is defined using a strongly convex map $\mu \mapsto H(y,\mu)$.
The regularization term $x \mapsto R(x,\theta)$ is assumed to be a closed proper and convex function, that accounts for the prior structure of $x_0$. 
Typical priors correspond to non-smooth regularizers such as sparsity in a suitable domain, e.g.~Fourier, wavelet \cite{mallat2009wavelet}, or gradient \cite{rudin1992nonlinear}. Such regularizers are usually parametrized with a collection of parameters $\theta$. 
%
A typical example is $R(x,\th) = \th R_0(x)$ where $\th \in \RR^+$ is a scaling which controls the strength of the regularization. Of course, more complicated (multi-parameters) regularizations, are often considered in the applications, and our methodology aims at dealing with these higher dimensional sets of parameters.

An important observation is that even though $x(y, \theta)$ may not be a singleton (minimizer of $E(x,y,\theta)$ may not be unique), strict convexity of $H(y,\cdot)$ implies that all minimizers share the same image under $\Phi$, see e.g.~\cite{vaiter2012local}.
Hence $(y, \theta) \mapsto \mu(y, \theta)$ is defined without ambiguity as a single-valued mapping. Moreover, strong convexity of $H(y,\cdot)$ implies that $y \mapsto \mu(y, \theta)$ is non-expansive (i.e.~uniformly 1-Lipschitz) \cite{vaiter2014dof}, hence weakly differentiable \cite[Theorem~5, Section~4.2.3]{EvansGariepy92}.

\item Consider now the $\ell$-th iterate, denoted by $\il{x}(y, \theta)$, of
an iterative algorithm converging to a fixed point of an operator acting on $\Xx$.
In this case, $\theta$ can include the parameters of
the fixed point operator, as well as other continuous parameters inherent
to the fixed point iteration (such as, e.g., step sizes). 
Section~\ref{sec:formal_diff} is completely dedicated to this setting, and appropriate sufficient conditions
will be exhibited to ensure weak differentiability of $\il{x}(y, \theta)$ with respect to both its arguments. 

This general setting encompasses the case of proximal splitting methods that have become popular to solve large-scale optimization problems of the form \eqref{eq:optim_energy}, especially with convex non-smooth terms, e.g.~those encountered in sparsity regularization. The precise splitting algorithm to be used depends on the structure of the optimization problem at hand. See for instance~\cite{CombettesPesquet09,BauschkeCombettes11} for an overview. Some of these algorithms are considered in detail in Section~\ref{sec:formal_diff}.
\end{itemize}

The choice of $\theta$ is generally a challenging task, especially as the dimension of $\Theta$ gets large. Ideally, one would like to choose the parameters $\theta^\star$ that makes $\mu(y, \theta^\star)$ (or some appropriate image of it) as faithful as possible to $\mu_0$ (or some appropriate image of it). Formally, this can be cast as selecting $\theta^\star$ that minimizes the expected reconstruction error (a.k.a., mean-squared error or quadratic risk), i.e.
\begin{align}\label{eq:risk_optim}
  \theta^\star \in \uArgmin{\theta \in \Theta}
  \{\ARISK{\mu}(\mu_0, \theta)
  =
  \EE_W \norm{A (\mu(Y, \theta) - \mu_0)}^2 \}
\end{align}
where the matrix $A  \in R^{M \times P}$ is typically chosen to counterbalance the effect of $\Phi$, see Section~\ref{sec:sure} for a precise discussion.

If $\theta \mapsto \ARISK{\mu}(\mu_0, \theta)$ were sufficiently smooth, at least locally (e.g.~Lipschitz), one could expect to solve \eqref{eq:risk_optim} using a (sub)gradient-descent scheme relying on the (weak) gradient of the risk $\GRAD{2}{\ARISK{\mu}}(\mu_0, \theta)$, where the subscript 2 specifies that the (weak) gradient is with respect to the second argument $\theta$.
However, this would only apply if $\mu_0$ were available.
In the context of our observation model \eqref{ew:forward_model},
$\mu_0$ is however considered to be unknown. Our motivation is then to build
an estimator of $\GRAD{2}{\ARISK{\mu}}(\mu_0, \theta)$
that depends solely on $y$, without prior knowledge of $\mu_0$.

Toward this goal, we adopt the framework of the (generalized) Stein Unbiased Risk Estimator (SURE) \cite{stein1981estimation,eldar-gsure,pesquet-deconv,raphan2011least,vaiter2012local}.
For a fixed $\theta$, the celebrated Stein's lemma \cite{stein1981estimation} allows to unbiasedly estimate $\ARISK{\mu}(\mu_0, \theta)$ through the weak Jacobian $\JACs{1}{\mu}(y, \theta)$, where the subscript $1$ specifies that the (weak) Jacobian is with respect to the first argument $y$.
This idea have been exploited for years in several statistical and signal processing applications,
typically for selecting thresholds in wavelet based reconstruction algorithms,
see e.g.~\cite{donoho1995adapting,benazza2005building,blu2007surelet,raphan2008optimal,chaux2008nonlinear,pesquet-deconv,ramani2012regularization,deledalle2013score}.
Given such an estimator $\AERISK{\mu}(y, \theta)$ (see Section~\ref{sec:sure}), the idea is so to replace the optimization problem \eqref{eq:risk_optim} with
\begin{align}\label{eq:_est_risk_optim}
  \theta^\star \in \uArgmin{\theta \in \Theta} \AERISK{\mu}(y, \theta)~.
\end{align}
Providing that the variance of $\frac{1}{P}\AERISK{\mu}(Y ,\theta)$
can be made arbitrarily small or even asymptotically vanishing as $P$ increases, so that
it becomes a consistent estimator of $\frac{1}{P}\ARISK{\mu}(\mu_0 ,\theta)$,
one can expect that the minimizers of \eqref{eq:_est_risk_optim}
becomes close to those of \eqref{eq:risk_optim}.

It remains to find an efficient way to solve the optimization \eqref{eq:_est_risk_optim}.
Again, a (sub)gradient-descent algorithm can qualify as a good candidate if
$\theta \mapsto \AERISK{\mu}(y, \theta)$ were sufficiently smooth.
To our knowledge, only \cite{deledalle2010pnl} have performed such an optimization
with Newton's method where $(y, \theta) \mapsto \AERISK{\mu}(y, \theta)$ was
$C^\infty$.
Unfortunately, being a function of $\JACs{1}{\mu}(y, \theta)$, 
$\theta \mapsto \AERISK{\mu}(y, \theta)$ is in general not differentiable, not even continuous (think of a simple soft-thresholding).
This then precludes the use of standard descent schemes.

The common practice has been to apply an exhaustive search
by evaluating the risk estimate $\AERISK{\mu}(y, \theta)$ at different values of $\theta$.
Even if in some particular cases this can be done efficiently
(see for instance \cite{donoho1995adapting}), the computational expense can become prohibitive in general
especially as $\dim(\Theta)$ increases.

Derivative-free optimization algorithms have also been investigated
(see for instance \cite{ramani2008montecarlosure}
for the case of 2 parameters). But such approaches
typically do not scale up to problems where
$\Theta$ has a linear vector space structure with dimension larger than 2.
Their performance are known to degrade exponentially with problem size, and they require to compute
a lower and an upper bound on the optimal value over a given region.



\subsection*{Contributions}
In this paper, we address the challenging problem of
solving efficiently \eqref{eq:_est_risk_optim}:
a main subject of interest for applications that
has been barely investigated.
Our main contribution (Section~\ref{sec:risk_optimization}) is an effective strategy
to optimize automatically a collection of parameters $\theta$ independently of their dimension.
While classical unbiased risk estimates entail optimizing a
non-continuous function of the parameters, we show that the biased risk estimator
introduced in \cite{ramani2008montecarlosure} is differentiable in the weak sense.
This allows us, whenever the derivatives exist, to perform a quasi-Newton optimization
driven by a biased estimator of the
gradient of the risk based on the evaluation
of $\JACs{2}{\mu}(y, \theta)$.
Such optimization technique can be provably faster thanks to first-order information compared
to derivative-free approaches.
We prove that, under mild assumptions, this estimator is asymptotically (with respect to $P$) unbiased,
hence the name: Stein Unbiased GrAdient estimator of the Risk (SUGAR). 
Moreover, in the particular case of soft-thresholding, we go a step further and show
that it is actually a consistent estimator of the gradient of the risk.

As a second contribution (Section~\ref{sec:formal_diff}), we propose a versatile approach to compute the derivatives $\JACs{1}{\mu}(y, \theta)$
and $\JACs{2}{\mu}(y, \theta)$, involved respectively in the computation of the SURE and SUGAR,
when $\mu(y, \theta)$ is computed through an iterative algorithm, typically proximal splitting methods.
We illustrate the versatility of our method by applying it to both primal (forward-backward~\cite{CombettesWajs05}, Douglas-Rachford~\cite{Combettes2007a} and generalized foward-backward~\cite{raguet-gfb}) and primal-dual \cite{chambolle2011first} algorithms (see \cite{komodakis2014playing} for a recent review).
The proposed methodology can however be adapted to any other proximal splitting method and more generally to any algorithm whose iteration operator is weakly differentiable.

Numerical simulations involving multi-parameter selection for image restoration and matrix completion problems are reported in Section~\ref{sec:numeric}. The proofs of our results are collected in the appendix.


%% file: sections/risk_estimation.tex
\section{Overview on Risk Estimation}
\label{sec:risk_estimation}

This section gives an overview of the literature to estimate the risk via the SURE and its variants for ill-posed inverse problems contaminated by additive white Gaussian noise.

\subsection{Stein Unbiased Risk Estimator}
\label{sec:sure}


Degrees of freedom (DOF) is often used to quantify
the complexity of a statistical modeling procedure,
see for instance, GCV (generalized cross-validation \cite{golub1979generalized}).
From \cite{efron1986biased,vaiter2012local},
the degrees of freedom of a function
$y \mapsto \mu(y, \theta)$ relatively to a matrix $A \in R^{M \times P}$ is given by
\begin{equation}\label{eq:dof}
  \ADOF{\mu}( \mu_0, \theta ) = \sum_{i=1}^P \frac{\cov(A Y_i, (A \mu(Y, \theta))_i)}{\sigma^2} ~,
\end{equation}
such that $\ADOF{\mu}( \mu_0, \theta )$ is maximal when $A \mu(Y, \theta)$
is highly correlated with the random vector $A Y$.
Taking $A = \Id$, leads to the standard
definition of the DOF defined in the seminal work of Efron \cite{efron1986biased}.
But other choices of $A$ allow to counterbalance the undesirable effect of the
linear operator $\Phi$ (recall that $\mu_0 = \Phi x_0$).
For instance, setting $A = (\Phi^* \Phi)^{-1} \Phi^*$ when $\Phi$ has full-rank,
or $A = \Phi^*(\Phi \Phi^*)^+$ when $\Phi$ is rank deficient\footnote{$(\cdot)^+$ stands for the Moore-Penrose pseudo-inverse.},
provides a measure of the DOF relatively to the least-squares estimate of $x_0$, i.e.~$x_{\mathrm{LS}}(y) = A y$ \cite{eldar-gsure,pesquet-deconv,vaiter2012local}.

With the proviso that $y \mapsto \mu(y, \theta)$
is weakly differentiable with essentially bounded weak partial derivatives,
an unbiased estimate of the DOF can be used to
unbiasedly estimate the risk in \eqref{eq:risk_optim}.
This leads to the (generalized) SURE (also known as weighted SURE \cite{ramani2012iterative}) given as
\begin{align}
  \label{eq:sure}
  & \ASURE{\mu}(y, \theta) =
  \norm{A (\mu(y, \theta) - y)}^2
  - \sigma^2 \tr( A^*\!A )
  + 2 \sigma^2 \AEDOF{\mu}( y, \theta ) \\
  \quad \text{with} \quad &
  \AEDOF{\mu}( y, \theta ) =
  \tr \pa{ A \JACs{1}{\mu}(y, \theta) A^* }
  \nonumber
\end{align}
where we recall that $\JACs{1}{\mu}(y, \theta)$ is the weak Jacobian of $\mu(y, \theta)$
with respect to the first argument $y$.
It can be shown that (see e.g.~\cite{stein1981estimation,eldar-gsure,vaiter2012local})
\[
\EE_W [ \AEDOF{\mu}( Y, \theta ) ] = \ADOF{\mu}( \mu_0, \theta ) \qandq \EE_W [ \ASURE{\mu}( Y, \theta ) ] = \ARISK{\mu}( \mu_0, \theta ) ~.
\]
Expression~\eqref{eq:sure}
is general enough to encompass unbiased
estimates of the \emph{prediction} risk $\EE_W \norm{\mu(Y, \theta) - \mu_0 }^2$
(i.e.~$A = \Id$),
the \emph{projection} risk $\EE_W \norm{\Pi (x(Y, \theta) - x_0) }^2$,
where $\Pi$ is the orthogonal projector on $\ker(\Phi)^\perp$ (i.e.~$A = \Phi^*(\Phi \Phi^*)^+$),
and the \emph{estimation} risk $\EE_W \norm{x(Y, \theta) - x_0 }^2$ 
when $\Phi$ has full rank (i.e.~$A = (\Phi^* \Phi)^{-1} \Phi^*$).
This can prove useful when $\Phi$ is rank deficient, since
in this case, the minimizers of the {prediction} risk
can be far away from the minimizers of the {estimation} risk \cite{rice1986choice}.
The projection risk restricts the estimate to the subspace where there is a signal beside noise, and in this sense, is a good
approximation of the estimation risk \cite{eldar-gsure}.

Note that generalization of the $\SUREn$ have been developed for other noise models,
typically within the multivariate canonical exponential family
(see e.g.~\cite{hudson1978nie,hudson1998maximum,eldar-gsure,raphan2011least}).

Applications of SURE emerged for choosing the smoothing
parameters in families of linear estimates \cite{li1985sure}
such as for model selection, ridge regression,
smoothing splines, etc. 
After its introduction in the wavelet community with the
SURE-Shrink algorithm \cite{donoho1995adapting},
it has been widely used for various image restoration problems, e.g.~%
with sparse regularizations
\cite{benazza2005building,blu2007surelet,vonesch2008sure,raphan2008optimal,ramani2008montecarlosure,chaux2008nonlinear,pesquet-deconv,cai2009data,luisier2010sure,ramani2012regularization,ramani2012iterative,deledalle2013score} or with non-local filters
\cite{vandeville2009sure,duval2011abv,dds2011nlmsap,vandeville2011non}.

However, a major practical difficulty when using the SURE lies in the numerical computation
of the DOF estimate, i.e.~the quantity $\AEDOF{\mu}(y, \theta)$
for a given realization $y$. We now give a brief overview of some previous works to deal
with this computation.

\subsection{Closed-form SURE}

The SURE is based on a DOF estimate $\AEDOF{\mu}(Y, \theta)$
that can be sampled from the observation $y \in \RR^P$
by evaluating the Jacobian $\JACs{1}{\mu}(y, \theta) \in \RR^{N \times P}$.
A natural way, to evaluate $\JACs{1}{\mu}(y, \theta)$ would be to
derive its closed-form expression.
This has been studied for some classes of variational problems.

In quadratic regularization (e.g.~ridge regression), where
solutions are of the form $x(y, \theta) = K(\theta) y$, where $K(\theta)$ is known as the hat or influence matrix,
the Jacobian has a closed-form $\JACs{1}{\mu}(y, \theta) = \Phi K(\theta)$.
In $\ell^1$-synthesis regularization (a.k.a.~the lasso), the
Jacobian matrix depends on the support
(set of non-zero coefficients)
of any lasso solution $x(y, \theta)$.
An estimator of the DOF
can then be retrieved from the number of non-zero entries of this solution
\cite{zou2007degrees,tibshirani2012dof,dossal-dof}.
These results have in turn been extended to more general
sparsity promoting regularizations \cite{yuan2006model,kato2009degrees,solo2010threshold,tibshirani2011solution,tibshirani2012dof,deledalle2012unbiased,vaiter2012local}, and spectral regularizations (e.g.~nuclear norm) \cite{deledalle2012risk,Candes12}. 

This approach however has three major bottlenecks.
First, deriving the closed-form expression of the Jacobian is in general challenging and has to be addressed on a case by case basis.
Second, in large-dimensional problems, evaluating numerically this Jacobian is barely possible.
Even if it were possible, it might be subject to serious numerical instabilities. 
Indeed, solutions of variational problems are achieved via iterative schemes providing iterates
$\il{x}(y, \theta)$ that eventually converge to the set of solutions as $\ell \to +\infty$. And yet, for instance,
substituting the support of the true solution by the support of $\il{x}(y, \theta)$, obtained at a prescribed convergence accuracy, 
might be imprecise (all the more since the problem is ill-conditioned).

The three next sections review previous work to address one or some of these three points.

\subsection{Monte-Carlo SURE}
\label{sec:mcsure}

To deal with the large dimension of the Jacobian,
the standard approach is to exploit the fact that the DOF
only depends on the trace of $A \JACs{1}{\mu}(y, \theta) A^*$.
In denoising applications where $A = \Id$ and $\Phi = \Id$,
this trace can generally be obtained by closed-form computations
of the $P$ diagonal elements of $\JACs{1}{\mu}(y, \theta)$
(see e.g.~\cite{donoho1995adapting,vandeville2009sure}).
This can also be done for some particular inverse problems.
For instance, the authors of \cite{pesquet-deconv} provide an
expression of this trace for the wavelet-vaguelette estimator
when $\Phi$ is a convolution matrix and $A = \Phi^+$.
However, in more general settings,
the complexity of the closed-form computation of the trace is non-linear,
typically the number of operations is in $O(P \times P)$
(think of $\Phi$ a mixing operator or $\mu$ an iterative estimator). To avoid such a costly
procedure, the authors of \cite{girard1989fast,ramani2008montecarlosure}
suggest making use of the following trace equality
\begin{align}
  \AEDOF{\mu}(y, \theta) =
  \tr \pa{ A \JACs{1}{\mu}(y, \theta) A^* }
  =
  \EE_\Delta
  \dotp{\JACs{1}{\mu}(y, \theta)[\Delta]}{ A^*\!A \Delta}
\end{align}
where $\Delta \sim \Nn(0,\Id_P)$ and
$\JACs{1}{\mu}(y, \theta)[\delta] \in \RR^P$ denotes the directional derivative of
$y \mapsto \mu(y, \theta)$ at $y$ in direction $\delta$.
Remark that $\Delta$ does not necessary have to be Gaussian and higher
precisions can be reached in some specific cases, see for instance
\cite{hutchinson1989stochastic,dong1994stochastic,avron2011randomized,roosta2013improved}.
As shown in \cite{roosta2013improved}, the performance of this trace estimator is governed by the distribution of the singular values of the operator $A \JACs{1}{\mu}(y, \theta) A^*$. More specifically, the slower the decay, the better the performance. While it is difficult to make a general claim, we observed numerically that for the recovery problems we consider, it provide a very accurate estimator of the trace.
Hence, following \cite{ramani2008montecarlosure,vonesch2008sure},
an estimate of $\ASURE{\mu}(y, \theta)$
can be obtained by Monte-Carlo simulations using
\begin{align}
  \label{eq-randomized-df}
  & \ASUREMC{\mu}(y, \theta, \delta) =
  \norm{A(\mu(y, \theta) - y)}^2
  - \sigma^2 \tr( A^*\!A )
  + 2 \sigma^2 \AEDOFMC{\mu}(y, \theta, \delta) \nonumber\\
  \text{with} \quad &
  \AEDOFMC{\mu}(y, \theta, \delta) =
  \dotp{\JACs{1}{\mu}(y, \theta)[\delta]}{ A^*\!A \delta}~.
\end{align}
The evaluation of \eqref{eq-randomized-df}
necessitates only computing the $P$ entries
of $\JACs{1}{\mu}(y, \theta)[\delta]$.


It remains to find a stable and efficient way to evaluate
for any vector $\delta \in \RR^P$ the directional derivative
$\JACs{1}{\mu}(y, \theta)[\delta] \in \RR^P$.

\subsection{Iterative Differentiation for Monte-Carlo SURE}
\label{sec:formal_mcsure}

When considering solutions $x(y, \theta)$ of a variational
problem, the DOF cannot be robustly estimated if one knows only
the iterates $\il{\mu}(y, \theta)$ that eventually converge to some $\mu(y, \theta)$ as $\ell \to +\infty$.
It appears then natural to estimate the DOF of $\il{\mu}(Y, \theta)$ directly and make the assumption that it will converge to that of $\mu(y, \theta)$.
For a realization $y \in \RR^P$, one can sample
an estimate of the DOF of the iterate $\il{\mu}(Y, \theta)$
by evaluating its directional derivative
$\JACs{1}{\il{\mu}}(y, \theta)[\delta]$.
A practical way, initiated by \cite{vonesch2008sure},
to compute this quantity,
consists in recursively differentiating the sequence of iterates.
The authors of \cite{vonesch2008sure} have derived the closed-form expression
of the directional derivative for the Forward-Backward (FB) algorithm.
The directional derivative at iteration $\ell+1$,
denoting by $\ill{\Dd_\mu} = \JACs{1}{\ill{\mu}}(y, \theta)[\delta]$,
is obtained iteratively
as a function of $\il{\mu}(y, \theta)$ and
$\il{\Dd_\mu} = \JACs{1}{\il{\mu}}(y, \theta)[\delta]$.
The Monte-Carlo DOF and the Monte-Carlo SURE can in turn be iteratively estimated by plugging
$\JACs{1}{\il{\mu}}(y, \theta)[\delta]$ in \eqref{eq-randomized-df}
leading to
\begin{align}\label{aq:mc_sure}
  & \ASUREMC{\il{\mu}}(y, \theta, \delta) =
  \norm{A(\il{\mu}(y, \theta) - y)}^2
  - \sigma^2 \tr( A^*\!A )
  + 2 \sigma^2 \AEDOFMC{\il{\mu}}(y, \theta, \delta) \nonumber\\
  \text{with} \quad &
  \AEDOFMC{\il{\mu}}(y, \theta, \delta) =
  \dotp{\il{\Dd_\mu}}{ A^*\!A \delta}
  ~.
\end{align}
A similar approach is described in \cite{giryes-proj-gsure}.
Pursuing this idea, the authors of \cite{ramani2012iterative,ramani2012regularization} have recently
provided such closed-form expressions in the case of the split Bregman
method. Concurrently, in an early short version of this paper
\cite{deledalle2012proximal},
we have also considered this approach for general proximal
splitting algorithms, an approach that we extend in Section~\ref{sec:formal_diff}.

\subsection{Finite-Difference SURE}
\label{sec:asure}

An alternative initiated in \cite{ye1998measuring,shen2002adaptive} and rediscovered in \cite{ramani2008montecarlosure}
consists in estimating  $\tr \left(A \JACs{1}{\mu}(y, \theta) A^*\right)$ via finite differences given, for $\epsilon >0$, by
\begin{align}
  \tr \left(A \JACs{1}{\mu}(y, \theta) A^*\right)
  \approx
  \sum_{i=1}^P
  \frac{[A^*A (\mu(y + \epsilon e_i, \theta) - \mu(y, \theta))]_i}{\epsilon} ~,
\end{align}
where $(e_i)_{1 \leq i \leq P}$ is the canonical basis of $\RR^P$.
Plugging this expression in~\eqref{eq-randomized-df}
yields the Finite-Difference (FD) $\ASUREn$
given by
\begin{align}\label{aq:a_sure}
  &\ASUREA{\mu}(y, \theta, \epsilon) =
  \norm{A(\mu(y, \theta) - y)}^2
  - \sigma^2 \tr( A^*\!A )
  + 2 \sigma^2 \AEDOFA{\mu}(y, \theta, \epsilon) \nonumber\\
  \quad \text{with} \quad &
  \AEDOFA{\mu}(y, \theta, \epsilon) =
  \frac{1}{\epsilon}
  \sum_{i=1}^P (A^*\!A (\mu(y + \epsilon e_i, \theta) - \mu(y, \theta)))_i
  .
\end{align}
The main advantage of this method is that $(y, \theta) \mapsto \mu(y, \theta)$
can be used as a black-box, i.e., without knowledge on
the underlying algorithm that provides $\mu(y, \theta)$,
while, for $\epsilon$ small enough,
it performs as well as the approach described in
Section \ref{sec:formal_mcsure}
that requires the knowledge of the derivatives in closed-form.
In fact, if $y \mapsto \mu(y, \theta)$ is Lipschitz-continuous, then it is differentiable Lebesgue a.e. (Rademacher's theorem), and its derivative equals its weak derivative Lebesgue a.e.~\cite[Theorem~1-2, Section~6.2]{EvansGariepy92}, which in turn implies
\begin{align}\label{eq:limsurefd}
  \lim_{\epsilon \to 0}
  \ASUREA{\mu}(y, \theta, \epsilon) =
  \ASURE{\mu}(y, \theta) \quad \text{Lebesgue a.e.}
\end{align}
The value of $\epsilon$ can so be chosen as small as possible
as soon as it does not rise to numerical instabilities due to limited machine precision.
To avoid numerical instabilities, $\epsilon$ should be chosen in a reasonable
range of values with respect to the amplitudes of the data.
In practice, we observe that precised results can be reached
compared to the closed-form derivation,
hence yielding to a quasi unbiased risk estimator
(i.e., with a negligible bias).
It remains that when the data dimension
$P$ is large, the evaluation of $P$ finite differences
along each axis might be numerically intractable.
In that case, the Monte-Carlo approach (see Section \ref{sec:mcsure})
can also be used in conjunction with finite differences leading to
the Finite-Difference Monte-Carlo (FDMC) $\ASUREn$
given by
\begin{align}\label{aq:amc_sure}
  &\ASUREAMC{\mu}(y, \theta, \delta, \epsilon) =
  \norm{A(\mu(y, \theta) - y)}^2
  - \sigma^2 \tr( A^*\!A )
  + 2 \sigma^2 \AEDOFAMC{\mu}(y, \theta, \delta, \epsilon) \nonumber\\
  \text{with} \quad &
  \AEDOFAMC{\mu}(y, \theta, \delta, \epsilon) =
  \frac{1}{\epsilon}
  \dotp{\mu(y + \epsilon \delta, \theta) - \mu(y, \theta)}{ A^*\!A \delta}
  .
\end{align}

The originality of our approach described in the next section is to devise a grounded choice
of $\epsilon > 0$. This introduces a bias in the estimation of the risk.
Nevertheless, as we will see, using $\epsilon > 0$ plays an important role
in risk optimization since, unlike $\ASURE{\mu}$,
$\ASUREA{\mu}$ is a smooth function of $\theta$ in the weak sense.
This is the key point to optimize the risk. By choosing $\epsilon > 0$ carefully,
a smoother objective function can be used as a basis to perform
a quasi-Newton-like optimization at the expense of a controlled bias.

%% file: sections/risk_optimization.tex
\section{Risk Estimate Minimization}
\label{sec:risk_optimization}

In this section, we investigate how risk estimates can be used for
optimizing a collection of continuous parameters.

\subsection{Stein's Unbiased GrAdient Risk (SUGAR) Estimator}

The difficulty is that even if $\theta \mapsto \ARISK{\mu}(\mu_0, \theta)$ is
differentiable in the weak sense, the function $\theta \mapsto \ASURE{\mu}(y, \theta)$
might contain discontinuities.
Typically, $\AEDOF{\mu}(y, \theta)$ has discontinuities where
$(y, \theta) \mapsto \mu(y, \theta)$ is not differentiable.

We start with a simple result showing that unlike $\ASURE{\mu}(y, \theta)$,
the finite-difference based mapping
$\theta \mapsto \ASUREA{\mu}(y, \theta, \epsilon)$, for $\epsilon > 0$,
is weakly differentiable.

\begin{prop}~\label{prop:sugar}
  Assume $\mu(y, \theta)$ is weakly differentiable with respect to $y$ and $\theta$.
  Given $\epsilon > 0$,
  $\AEDOFA{\mu}(y, \theta, \epsilon)$
  and
  $\ASUREA{\mu}(y, \theta, \epsilon)$
  are also weakly differentiable with respect to $y$ and $\theta$,
  and their (weak) gradients with respect to $\theta$ are given,
  for almost all $\theta \in \Theta$, as
  \begin{align*}
    \begin{array}{rlll}
    &\ASUGARA{\mu}(y, \theta, \epsilon) &=& \GRAD{2}{\ASUREA{\mu}}(y, \theta, \epsilon)\\
    &&=&
    2 \JACs{2}{\mu}(y, \theta)^* A^*\!A (\mu(y, \theta) - y)
    + 2 \sigma^2
    \GRAD{2}{\AEDOFA{\mu}}(y, \theta, \epsilon)\\
    \qwhereq &
    \GRAD{2}{\AEDOFA{\mu}}(y, \theta, \epsilon) &=&
    \displaystyle
    \frac{1}{\epsilon}
    \sum_{i=1}^P
    \left(\JACs{2}{\mu}(y + \epsilon e_i, \theta) - \JACs{2}{\mu}(y, \theta)\right)^* \!A^*\!A e_i
    ~.
    \end{array}
  \end{align*}
\end{prop}

Thanks to Proposition \ref{prop:sugar}, a quasi-Newton-like method %
can now be used to optimize $\ASUREA{\mu}(y, \theta, \epsilon)$ for the vector of continuous parameters $\theta$ by implementing the iteration
\begin{align*}
  \theta_{n+1}
  &=
  \theta_n
  -
  B_n
  \ASUGARA{\mu}(y ,\theta_n, \epsilon)
\end{align*}
where $B_n \in \RR^{\dim(\Theta) \times \dim(\Theta)}$ is a sequence of definite-positive matrices.
Typically, if $\theta \mapsto \ASUREA{\mu}(y ,\theta, \delta, \epsilon)$
behaves locally as a $C^2$ function,
$B_n$ should approach the inverse of the corresponding Hessian at $\theta_n$.
Remark that in general there is no guarantee that the risk has a unique global minimizer,
though in the 1-D case it is generally the case.
When several parameters are involved, such objective can have several local minima and
specific quasi-Newton-like methods might be developed to avoid being stuck in one of them.

In practice, the calculation of $\ASUGARAn$ depends on the computation of the
Jacobian matrices with respect to the parameters $\theta$.
We will see in Section \ref{sec:formal_diff} how this quantity
can be efficiently computed when $\mu$ results from
an iterative algorithm.
\\


\input{sections/thm1_harder}

Theorem \ref{thm:sugar} can be given the following interpretation.
As $\epsilon$ gets close to $0$, e.g.~a decreasing function of the dimension $P$\footnote{as we will see, the higher the dimension $P$, the smaller $\epsilon$ could be.},
the gradient of $\ASUREA{\mu}(y,\cdot,\epsilon)$ (normalized by $P$) can be used to estimate the gradient of
the risk (also normalized by $P$) provided that $P$ is large enough.

However, even if $\epsilon$ should decrease towards $0$, it should not
decrease too fast. In particular, for a fixed dimension $P$,
the step $\epsilon$ cannot be chosen arbitrarily small.
This would not be an issue if
$\mu(y ,\cdot)$
were differentiable,
but in general, there might be singularities.
In fact, for a finite dimension $P$,
the limit when $\epsilon \to 0$ of the sample
$\ASUGARA{\mu}(y ,\theta, \epsilon)$
may not even exist, though that of its expectation does exist Lebesgue a.e.~as shown in the proof of Theorem~\ref{thm:sugar}.
As a consequence, the quantity $\frac{1}{P}\ASUGARA{\mu}(Y ,\theta, \epsilon)$
can become very unstable when $\epsilon$ decreases too fast with the dimension $P$.
The underlying statistical question is whether one can control the variance of $\frac{1}{P}\ASUGARA{\mu}(Y ,\theta, \epsilon)$
as $P$ increases, and make arbitrarily small or even asymptotically vanishing, so that $\frac{1}{P}\ASUGARA{\mu}(Y ,\theta, \epsilon)$ becomes a consistent estimator.
Unfortunately, consistency of our gradient estimator of the risk is very intricate to get in the general case, as it is the case for the consistency of the SURE.
However, when $\mu$ specializes to soft-thresholding, such a result can be achieved.

\subsection{SUGAR for Soft-tresholding}
\label{sec:stats_fd}


In this section, we show that the proposed gradient estimator of the risk can be consistent
in the case where $\mu$ is the soft-thresholding (ST) function
and $A = \Id_P$. The ST is the proximal operator of the $\ell_1$-norm. 
Understanding the ST is of chief interest since it is at the heart of any proximal
splitting algorithm solving a regularized inverse problem involving
terms of the form $\norm{D^* x}_1$ where $D$ is a linear operator.

Let first recall the definition of soft-thresholding.
\begin{defn}[Soft-Thresholding]~
  The soft-thresholding (ST) is defined, for $\lambda > 0$, and
  for all $1 \leq i \leq P$, as
  \begin{equation}\label{eq:st}
    \ST(y, \lambda)_i = \choice{
      y_i + \lambda & \text{if}\quad y_i \leq -\lambda\\
      0 & \text{if}\quad -\lambda < y_i < \lambda\\
      y_i - \lambda & \text{otherwise}
    }~.
  \end{equation}%
\end{defn}%
Observe that as a proximity operator of a norm, soft-thresholding satisfies Assumptions~\ref{assump:L1} through \ref{assump:L2} of Theorem~\ref{thm:sugar}, see the corresponding discussion. Hence, we already anticipate from Theorem~\ref{thm:sugar} that our gradient estimator of the soft-thresholding risk is asymptotically unbiased.

We start with following lemma which collects the statistics of the gradient of the finite difference DOF estimator.
\begin{lem}[Statistics of the gradient of the finite difference DOF estimator]\label{prop:st_stats}
  Let $0 < \epsilon < 2\lambda$.
  The weak gradient of
  $\lambda \mapsto \EDOFA{\ST}(Y, \lambda, \epsilon)$
  is such that
  \begin{align*}
    \begin{array}{r@{\;}l@{\;}r@{\;}l@{\;}l}
      \displaystyle
      \EE_W \left[ \GRAD{2}{\EDOFA{\ST}}(Y, \lambda, \epsilon) \right]
      &=
      \displaystyle
      \frac{-1}{2}
      \sum_{i = 1}^P
      \frac{\varphi[(\mu_0)_i, \lambda, \epsilon]}{\epsilon},\\
      \displaystyle
      \text{and} \quad
      \VV_W \left[ \GRAD{2}{\EDOFA{\ST}}(Y, \lambda, \epsilon) \right]
      &=
      \displaystyle
      \frac{1}{2\epsilon}
      \sum_{i = 1}^P
      \frac{\phi[(\mu_0)_i, \lambda, \epsilon]}{\epsilon}
      -
      \frac{1}{4}
      \sum_{i = 1}^P
      \bigg[
      \frac{\phi[(\mu_0)_i, \lambda, \epsilon]}{\epsilon}
      \bigg]^2~.\\
    \end{array}
  \end{align*}
  where for $a \in \RR$, $\varphi[a, \lambda, \epsilon] =
        \erf\pa{\frac{a + \lambda + \epsilon}{\sqrt{2}\sigma}}
        -
        \erf\pa{\frac{a + \lambda }{\sqrt{2}\sigma}}
        +
        \erf\pa{\frac{a - \lambda  + \epsilon}{\sqrt{2}\sigma}}
        -
        \erf\pa{\frac{a - \lambda }{\sqrt{2}\sigma}}
$.
\end{lem}

We now turn to the asymptotic behavior of the proposed gradient estimator of the risk for large $P$, at
a single realization of $Y$, i.e., our observation $y$.
To this end, we first have to define how
the observation model evolves with the dimension $P$.
Given $z_0 \in \RR^N$, we consider the sequence
$\{\Psi_P\}_{P \geq 1}$ where $\Psi_P \in \RR^{P \times N}$ and such that, for all $P>1$,
$\Psi_{P}$ is the sub-matrix obtained by cutting down one line of $\Psi_{P+1}$.
We can then define a sequence of observation models as a the sequence
of random vector $\{Y_P\}_{P \geq 1}$ defined as
\begin{align}
  Y_P = \Psi_P z_0 + W_P
  \qwhereq
  W_P \sim \Nn(0, \sigma^2 \Id_P).
\end{align}
We also define the sequence $\{(\mu_0)_P\}_{P \geq 1}$ where $(\mu_0)_P = \Psi_P z_0$.
In the following, for the sake of clarity,
we omit the dependency of $Y_P$, $W_P$ and $(\mu_0)_P$
on $P$.



We can now state our consistency result for soft-thresholding.

\begin{thm}[Consistency of SUGAR]~\label{thm:cons_sugar}
  Take $\hat{\epsilon}(P)$ such that $\lim_{P \to \infty} \hat{\epsilon}(P) = 0$ and
  $\lim_{P \to \infty} P^{-1} \hat{\epsilon}(P)^{-1} = 0$. Then for any
Lebesgue point $\lambda > 0$ (i.e.~such that $\forall (i, P), \lambda \ne |Y_i| \qandq  \lambda \ne |Y_i + \hat{\epsilon}(P) e_i|$)
  \begin{align*}
    \uplim{P \to \infty} \left[ \frac{1}{P} \left(
        \SUGARA{\ST}(Y, \lambda, \hat{\epsilon}(P))
        - \GRAD{2}{\RISK{\ST}}(\mu_0, \lambda) \right) \right] &= 0\\
    \qandq
    \uplim{P \to \infty} \left[ \frac{1}{P} \left(
        \GRAD{2}{\EDOFA{\ST}}(Y, \lambda, \hat{\epsilon}(P))
        - \GRAD{2}{\DOF{\ST}}(\mu_0, \lambda) \right) \right] &= 0~.
  \end{align*}
\end{thm}

In plain words, Theorem \ref{thm:cons_sugar} asserts that for our gradient estimator of the soft-thresholding risk
to be consistent, $\hat\epsilon(P)$ should not decrease faster than the inverse of the dimension $P$.
With the proviso that $\hat\epsilon(P)$ fulfills the requirement, for $P$ large enough,
$\frac{1}{P}\SUGARA{\ST}(y, \lambda, \hat{\epsilon}(P))$
is guaranteed to come close to
$\frac{1}{P}\GRAD{2}{\RISK{\ST}}(\mu_0, \lambda)$ with high probability.

Unfortunately, Theorem \ref{thm:cons_sugar} does not dictate an explicit
choice of $\hat{\epsilon}(P)$, and the practitioner may wonder how to choose this value for a given $P$.
It turns out that studying the mean squared error (MSE) of the gradient of the finite difference DOF estimator
helps unveiling the link between $P$ and $\epsilon$ through a bias-variance trade-off.

\begin{prop}[MSE of the gradient of the finite difference DOF estimator]
  \label{prop:st_mse}
  The weak gradient of
  $\lambda \mapsto \EDOF{\ST}(Y, \lambda, \epsilon)$
  is such that
  \begin{align*}
    &\EE_{W} \left[
      \frac{1}{P}
      \left(
        \GRAD{2}{\EDOF{\ST}}(Y, \lambda, \epsilon) -
        \GRAD{2}{\DOF{\ST}}(\mu_0, \lambda)
      \right)
    \right]^2
    =\\
    &
    \quad
    \underbrace{\frac{1}{P^2}
      \left(\EE_{W}\!\left[ \GRAD{2}{\EDOF{\ST}}(Y, \lambda, \epsilon) \right] \!-\!
        \GRAD{2}{\DOF{\ST}}(\mu_0, \lambda) \right)^2
    }_{Bias^2}
    +
    \underbrace{
      \frac{1}{P^2}
      \VV_{W}\!\left[ \GRAD{2}{\EDOF{\ST}}(Y, \lambda, \epsilon) \right].
    }_{Variance}
  \end{align*}
  where the statistics of $\GRAD{2}{\EDOF{\ST}}(Y, \lambda, \epsilon)$
  are given in Lemma \ref{prop:st_stats} and
  \begin{align*}
    \GRAD{2}{\DOF{\ST}}(\mu_0, \lambda)
    =
    \frac{-1}{\sqrt{2 \pi} \sigma}
    \sum_{i = 0}^P
    \bigg[
      \exp\left( -\frac{((\mu_0)_i + \lambda)^2}{2 \sigma^2}\right)
      +
      \exp\left( -\frac{((\mu_0)_i - \lambda)^2}{2 \sigma^2}\right)
    \bigg]~.
  \end{align*}
\end{prop}

\begin{figure}[!t]
  \centering
  \subfigure[Bias-variance trade-off w.r.t. $\epsilon$]{\includegraphics[height=0.34\linewidth]{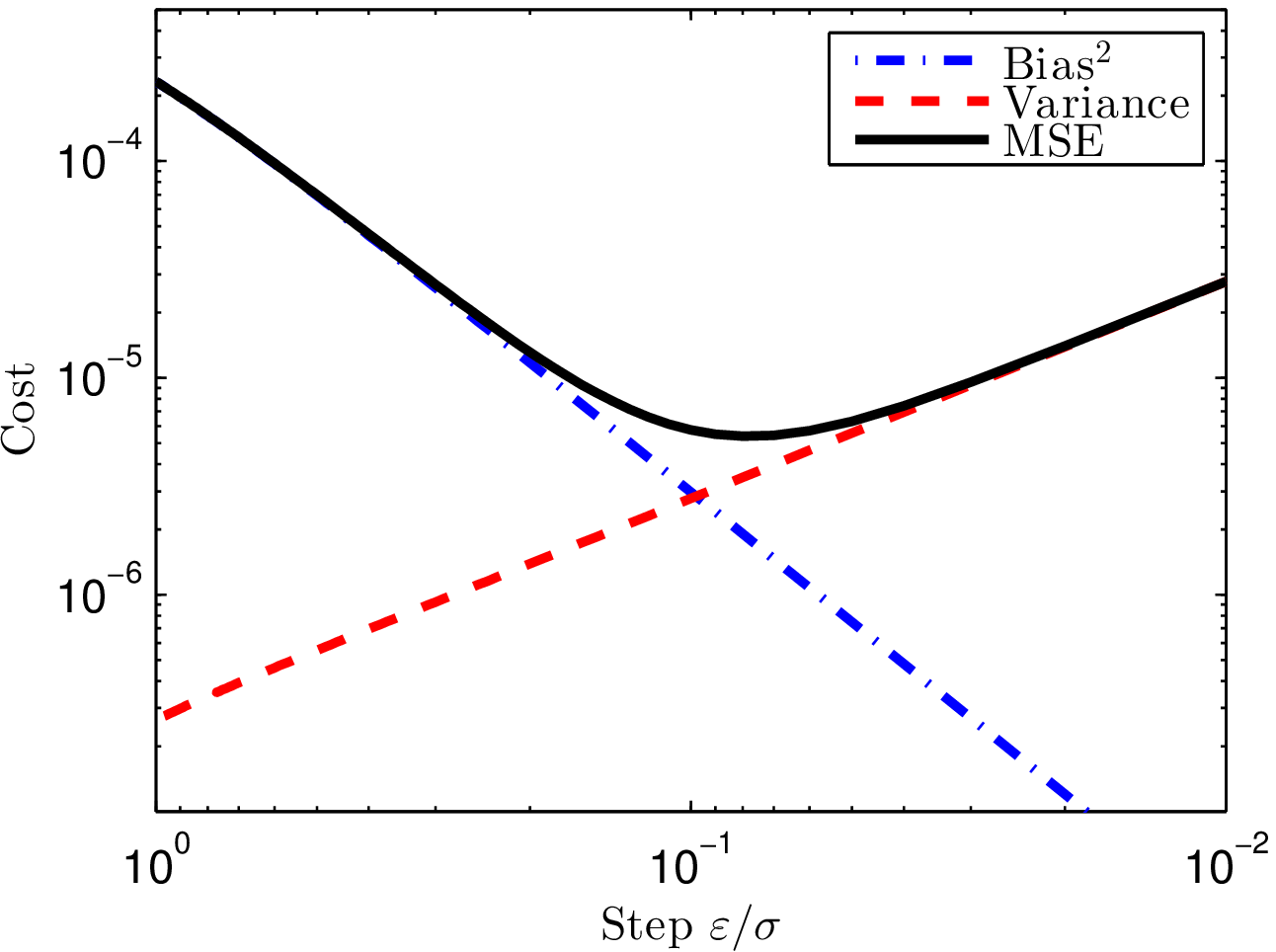}}
  \subfigure[Bias-variance w.r.t. $P$ (using $\hat{\epsilon}(P)$)]{\includegraphics[height=0.34\linewidth]{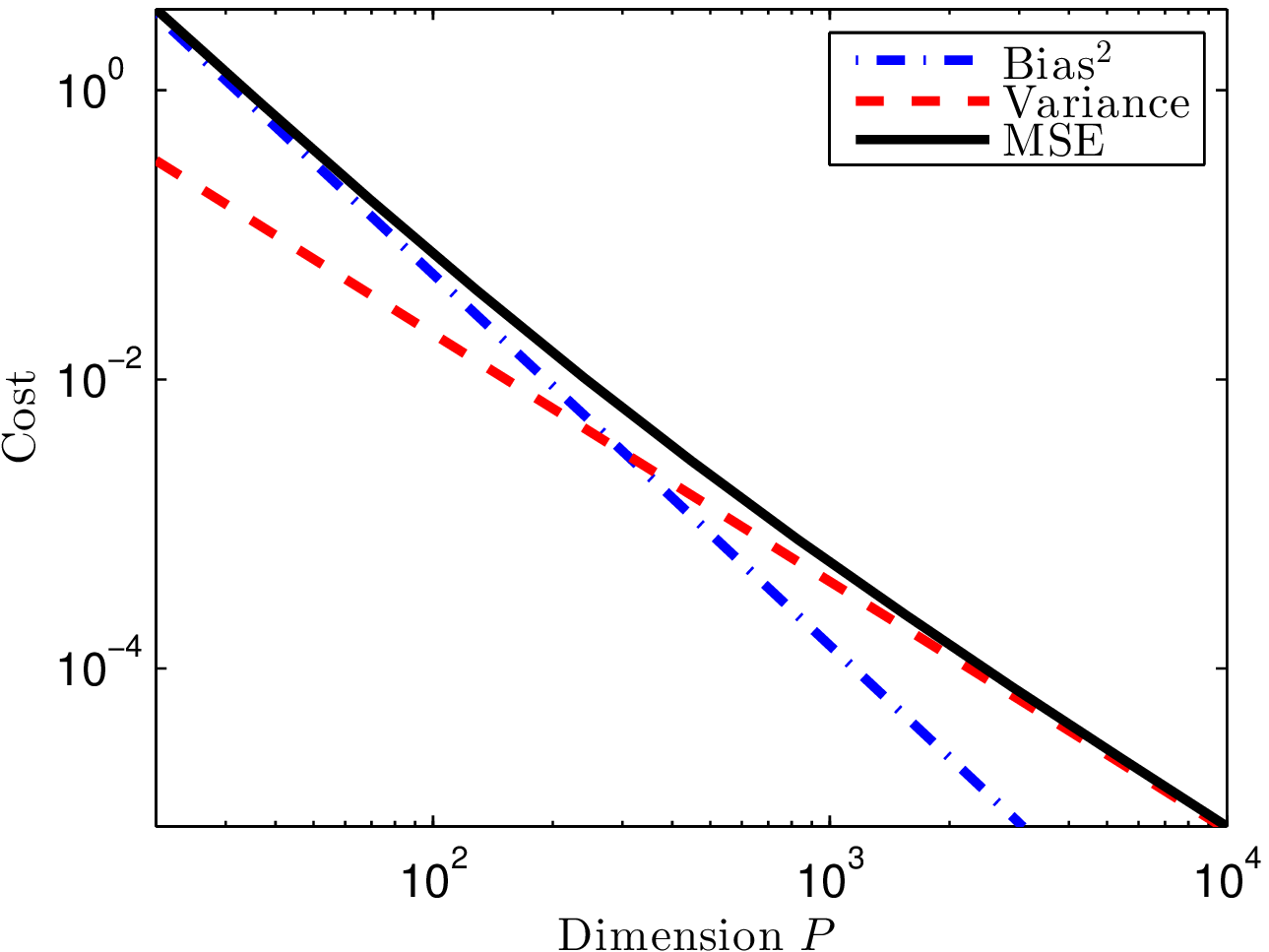}}
  \subfigure[MSE w.r.t. $(\epsilon, P)$]{\includegraphics[height=0.34\linewidth]{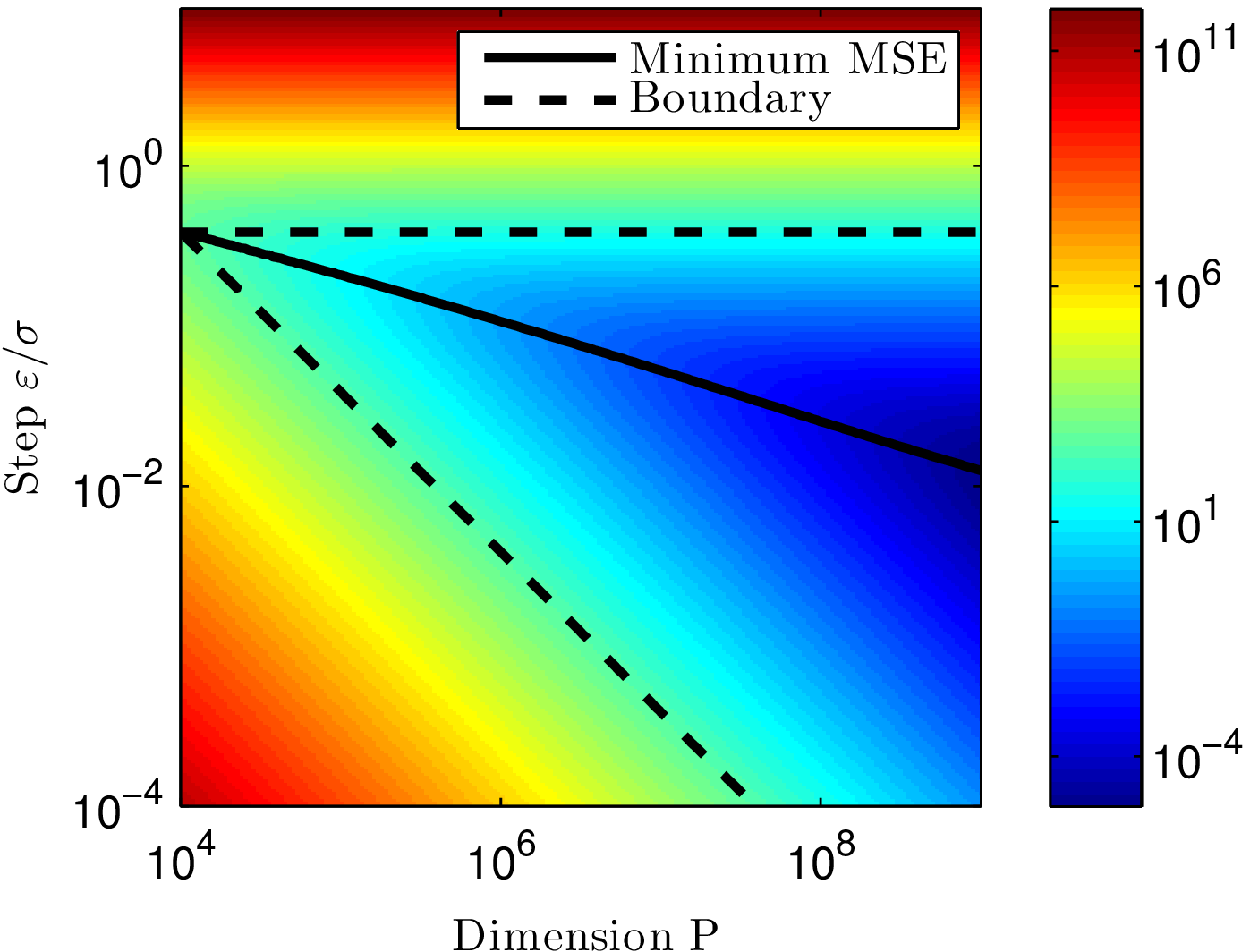}}
  \subfigure[Choice of $\sigma$, $\mu_0$ and $\lambda$]{\includegraphics[height=0.34\linewidth]{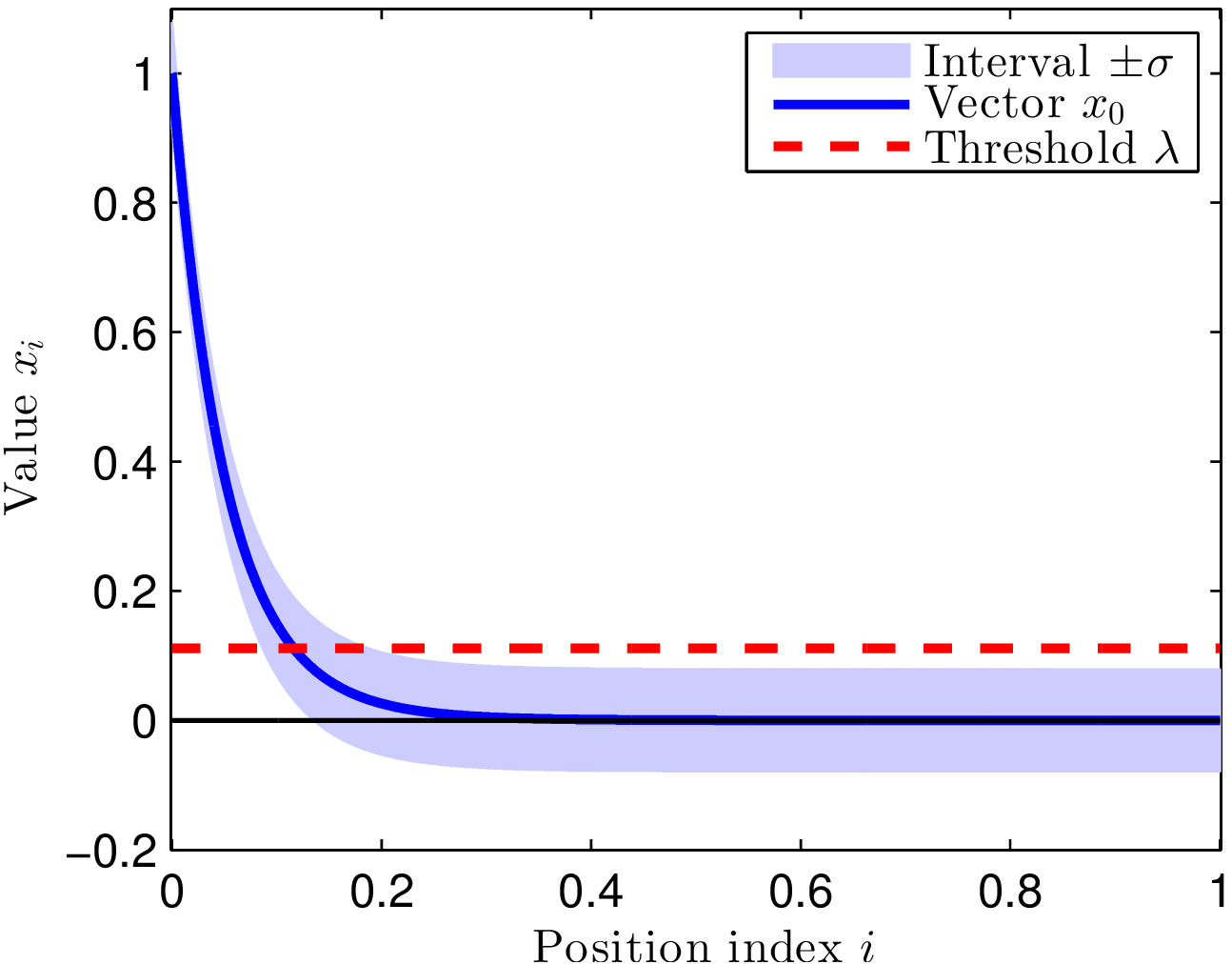}}
  \caption{Bias-variance trade-off of the
    gradient estimator of the DOF of soft-thresholding,
    (a) with respect to the step
    $\epsilon$ and (b) with respect to the dimension $P$ when using
    a power decay function $P \mapsto \hat{\epsilon}(P)$.
    (c) Its mean squared error as a function of $P$ and
    $\epsilon$ (in logarithmic scales).
    The solid line represents the pairs $(\hat{\epsilon}^\star(P), P)$
    where, for a fixed dimension $P$,
    $\hat{\epsilon}^\star(P)$ minimizes the mean squared error.
    The function $\hat{\epsilon}^\star(P)$ looks like a power function
    of the form of $C \sigma / P^\alpha$ with
    $C > 0$ and $0 < \alpha < 1$.
    The dashed lines represent respectively the power functions
    $\hat{\epsilon}^\mathrm{inf}(P) = C\sigma$ and
    $\hat{\epsilon}^\mathrm{sup}(P) = C\sigma/P$
    outside which the mean squared error diverges when $P$ increases.
    (d) Description of the settings of the experiments, i.e.,
    the choice of $\sigma$, $\mu_0$ and $\lambda$.
  }
  \label{fig:bias_var_mse_st}
\end{figure}%

Thus, if $\mu_0$ were given, the quantities in Proposition \ref{prop:st_mse} could be computed in closed-form. The MSE can then be evaluated to select the optimal value of $\epsilon$ for a fixed dimension $P$ and a given threshold $\lambda$. See the following numerical experiments which illustrate this relationship. When $\mu_0$ is unknown, an {\it a priori} model can be imposed, such as for instance belonging to some ball promoting sparsity, e.g.~a weak $\ell_{\gamma}$-ball for $\gamma > 0$. For $\gamma$ sufficiently small, this ball corresponds to compressible or nearly sparse vectors $\mu_0$ whose entries $\abs{\mu_i}$ sorted in descending order of magnitude behave as $O(i^{-1/\gamma})$. With such a model at hand, the MSE in Proposition~\ref{prop:st_mse} can be optimized for $\epsilon$ given $P$, $\sigma$, $\lambda$ and $\gamma$. This however entails a highly non-linear equation that cannot be solved in closed form. We defer such a development to a future work.

Figure~\ref{fig:bias_var_mse_st}.(a) shows the evolution of
the bias and the variance as a function of the ratio $\epsilon/\sigma$
for fixed values of $\sigma$, $\lambda$ and
a compressible vector $\mu_0$, i.e.~$\abs{(\mu_0)_i} = O(i^{-1/\gamma})$,
chosen as illustrated on Figure \ref{fig:bias_var_mse_st}.(d).
When $\epsilon \to 0$, for fixed $P$, the bias
vanishes while the variance, and in turn the MSE,
increases. However, for a step $\epsilon > 0$,
the MSE is finite
and seems to be optimal around the value $0.1 \sigma$.
Figure \ref{fig:bias_var_mse_st}.(c) shows the evolution
of the MSE as a function of the dimension
$P$ and the ratio $\epsilon/\sigma$ for the same fixed values
as before. The optimal step, minimizing the MSE,
seems to evolve as a power decay function (the scale is log-log) of the form
$\epsilon^\star(P) = C \sigma / P^ \alpha$ with
$C > 0$ and $0 < \alpha < 1$.
Of course the optimal constants $C$ and $\alpha$ depend on
the choice of $\mu_0$, $\sigma$ and $\lambda$.
However,
whatever $C > 0$ and $0 < \alpha < 1$, or more generally
for any admissible choice of $\hat{\epsilon}$ such that
$\lim_{P \to \infty} \hat{\epsilon}(P) = 0$ and
$\lim_{P \to \infty} P^{-1} \hat{\epsilon}(P)^{-1} = 0$,
the MSE vanishes with respect to $P$.
Figure \ref{fig:bias_var_mse_st}.(b) shows indeed the evolution
of the bias, the variance and the MSE as a function of the
dimension $P$ when $\hat{\epsilon}$ is chosen
as a power decay function.
For $\alpha = 0$ or $\alpha = 1$, the MSE
remains constant while, for $\alpha > 1$,
the MSE diverges which suggests the necessity of $\lim_{P \to \infty} P^{-1} \hat{\epsilon}(P)^{-1} = 0$.

%% file: sections/thm1_harder.tex

We now turn to the asymptotic unbiasedness of the proposed gradient estimator of the risk as $\epsilon$ approaches $0$. Toward this goal we need the following assumptions.
\begin{enumerate}[label={\rm\bf (A.\arabic{*})}, ref={\rm\bf (A.\arabic{*})}]
\item The mapping $y \mapsto \mu(y, \theta)$ is uniformly Lipschitz continuous with Lipschitz constant $L_1$.\label{assump:L1}
\item The mapping $y \mapsto \mu(y,\theta)$ is such that $\mu(0,\theta)=0$ for any $\theta$.\label{assump:mu}
\item The mapping $\theta \mapsto \mu(y, \theta)$ is uniformly Lipschitz continuous with Lipschitz constant $L_2$ independently of $y$.\label{assump:L2}
\end{enumerate}
\bigskip
\begin{rem}[Discussion of the assumptions]
\begin{enumerate}
\item Assumption~\ref{assump:L1} is mild and is fulfilled in many situations of interest. In particular, this is the case when $y \mapsto \mu(y, \theta)$ is the proximal operator of proper closed and convex function, as considered in Section~\ref{sec:formal_diff} (see also Section~\ref{sec:stats_fd} for soft-thresholding). Standard convex analysis arguments \cite{hiriart1996convex} show that the proximity operator is indeed a uniformly Lipschitz of its argument $y$ with constant $L_1=1$, independently of $\theta$.

\item Assumption~\ref{assump:mu} is very natural and does not entail any loss of generality. It basically states that, when the observations are zero, so is the estimator.

\item As far as Assumption~\ref{assump:L2} is concerned, it is verified under certain circumstances. This is for instance the case when $\mu(y,\theta)=\Prox_{\theta G}(y)$, $\theta > 0$, where $G$ is the
gauge (see Definition \ref{defn:convex-gauge} in Appendix \ref{sec:reggauge}) of any compact convex set containing the origin as an interior point\footnote{Another case which is trivial corresponds to $G$ being the indicator function of a non-empty closed convex set, in which case $L_2=0$.}; see Proposition~\ref{prop:lipproxgauge} in Appendix~\ref{sec:reggauge}. By induction, this also holds when $\mu(y,\theta)=\Prox_{\theta_1 G_1} \circ \cdots \circ\Prox_{\theta_m G_m}(y)$, $\theta \in ]0,+\infty[^m$, and for any $i=1,\ldots,m$, $G_i$ is the gauge of any compact convex set containing the origin as an interior point, see Corollary~\ref{cor:lipcompproxgauge}. Typical instances of these gauges are norms, e.g. $\ell_1$, $\ell_1-\ell_2$ or nuclear norms very popular now in the signal and image processing community.

\item Assumption~\ref{assump:L2} can be relaxed to cover the case where $L_2$ depends on $y$. In such a situation, additional assumptions on the function $y \mapsto L_2(y)$ are needed for steps 3) and 4) in the proof of Theorem~\ref{thm:sugar} to go through. We omit this case for the sake of clarity and to avoid further technicalities.

\end{enumerate}
\end{rem}
{~}\\
We are now ready to state our theorem. 
\begin{thm}[Asymptotic unbiasedness of SUGAR]~\label{thm:sugar}
  Assume that \ref{assump:L1}-\ref{assump:L2} hold.
  Then, $\ARISK{\mu}(\mu_0, \cdot)$ and $\ADOF{\mu}(\mu_0, \cdot)$ are weakly differentiable, and for any Lebesgue point $\theta$, 
  \begin{align*}
    \lim_{\epsilon \to 0}\EE_W \left[\ASUGARA{\mu}(Y ,\theta, \epsilon)\right]
    &=\GRAD{2}{{\ARISK{\mu}}}(\mu_0, \theta)
    \\
    \qandq
    \lim_{\epsilon \to 0}
    \EE_W \left[\GRAD{2}{\AEDOFA{\mu}}(Y, \theta, \epsilon)\right]
    &=\GRAD{2}{\ADOF{\mu}}(\mu_0, \theta)~.
  \end{align*}
\end{thm}

%% file: sections/differentiation.tex
\section{Differentiation of an Iterative Scheme}
\label{sec:formal_diff}

We now turn to iterative algorithms that involve linear and soft-thresholding operators.
We observed empirically that for all the inverse problems exposed in Section \ref{sec:numeric}, setting $\epsilon^\star(P) = C \sigma / P^ \alpha$, as suggested by our study on the soft thresholding, resulted in a reliable way to parametrize our estimator. The effectiveness of this heuristic might be explained by the fact that the singularities encountered in most imaging problems are similar to absolute values, in order to encourage some sort of sparsity in the solution.

In this section, we focus on iterates, defined unambiguously as single-valued mappings $(y, \theta) \mapsto \il{x}(y, \theta)$,
where $\ell$ is the iteration counter of the iterative algorithm.
In this context, we propose to compute in closed-form
the derivatives of $\il{x}(y, \theta)$ with respect to either $y$ (in a direction $\delta$)
or $\theta$. This proves useful to respectively
estimate the risk via $\ASUREMCn$ (see Section \ref{sec:risk_estimation})
and estimate its gradient via $\ASUGARAMCn$
(see Section \ref{sec:risk_optimization}).


The iterative schemes we consider can be cast in the same framework, which subsumes proximal splitting algorithms designed to minimize a proper, closed and convex objective function $x \mapsto E(x,y,\theta)$, whose set of minimizers is supposed non-empty. All these algorithms can be unified as an iterative scheme of the form
\begin{equation}\label{eq-iterations}
\begin{cases}
  \il{x}  &= \gamma(\il{a} )\\
  \ill{a} &= \psi( \il{a}, y, \theta),
\end{cases}
\end{equation}
where $\il{a} \in \Aa$ is a sequence of auxiliary variables.
$\psi : \Aa \times \Yy \times \Theta \to \Aa$ is a fixed point operator in such a way that $\il{a}$ converge to a fixed point $a^\star$, 
and $\gamma : \Aa \to \Xx$ is 
non-expansive (i.e, $\norm{\gamma(a_1) - \gamma(a_2)} \leq \norm{a_1 - a_2}$ for any $a_1, a_2 \in \Aa$)
entailing that $\il{x}$ will converge to $x^\star=\gamma(a^\star)$.
Note that for the sake of clarity, we have dropped the dependencies of $a^\star$ and $x^\star$ to $y$ and $\theta$.

To make our ideas clear, consider the instructive example where $x \mapsto E(x,y,\theta)$ is convex and $C^1(\Xx)$ with $L$-Lipschitz gradient, in which case $\Aa=\Xx$, $a = x$ and $\psi(x,y,\theta)=x - \tau \nabla_1 E(x,y,\theta)$ where $0 < \tau < 2/L$.

\subsection{Iterative Weak Differentiability}

A practical way to get the weak directional derivative $\JACs{1}{x}(y, \theta)[\delta]$ and the weak Jacobian $\JACs{2}{x}(y, \theta)$, is to compute them iteratively from the sequences \eqref{eq-iterations} by relying on the chain rule. However, two major issues have to be taken care of. First, one has to ensure weak differentiability of the iterates \eqref{eq-iterations} so that $\JACs{1}{\il{x}}(y, \theta)[\delta]$ (or resp.~to $\JACs{2}{\il{x}}(y, \theta)$) exist Lebesgue a.e. Second, one may legitimately ask whether the sequence of weak derivatives converges, and the properties of its cluster point, if any, with respect to the weak derivatives at a minimizer $x^\star$. 

Regarding weak differentiability of the iterates, it relies essentially on regularity conditions to apply the chain rule, e.g.~\cite[Section~4.2.2]{EvansGariepy92}, i.e.~regularity properties of the iteration mappings $\gamma$ and $\psi$ and of the initialization.
For instance, for proximal splitting algorithms, it turns out that $\gamma$ is the composition of one or several non-expansive operators, hence 1-Lipschitz, operators.
In turn, $\gamma$ is 1-Lipschitz. Furthermore, in all examples we consider, $\psi$ is also 1-Lipschitz with respect to its second and third arguments.
Therefore, if one starts at a Lipschitz continuous initialization, by induction, $y \mapsto \il{x}(y, \theta)$ and $\theta \mapsto \il{x}(y, \theta)$ are also Lipschitz. Using the chain rule for Lipschitz mappings~\cite[Theorem~4 and Remark, Section~4.2.2]{EvansGariepy92}, weak differentiability of $\il{x}$ follows with respect to both arguments.

As far as convergence of the sequence of weak Jacobians is concerned, this remains an open question in the general case, and we believe this would necessitate intricate arguments from non-smooth and variational analysis. This is left to future research.

From now on, we suppose that the Lipschitzian assumptions on $\gamma$, $\psi$ and the initial points hold. The next two sections detail the computation of $\JACs{1}{\il{x}}(y, \theta)[\delta]$ and $\JACs{2}{\il{x}}(y, \theta)$ in order to get the estimates $\ASUREMCn$ and $\ASUGARAMCn$.

\begin{figure}[!t]
\centering
\begin{minipage}{\linewidth}
\rule{\linewidth}{2px}
\begin{algorithmic}[0]
  \State \textbf{Algorithm} Risk estimation of an iterative scheme
\end{algorithmic}
\vspace{-0.2cm}
\rule{\linewidth}{1px}
\begin{algorithmic}[0]
  \State \makebox[3cm][l]{\textbf{Inputs:}} observation $y \in \Yy = \RR^P$, collection of parameters $\theta \in \Theta$
  \State \makebox[3cm][l]{\textbf{Parameters:}} noise variance $\sigma^2 > 0$, linear operator $\Phi \in \RR^{P\times N}$,
  \State \makebox[3cm][l]{}  matrix $A \in \RR^{M \times P}$, number $\mathcal{L}$ of iterations
  \State \makebox[3cm][l]{\textbf{Output:}} solution $x(y, \theta) \in \Xx$ and its risk estimate $\AERISK{x}(y, \theta)$\\
  \State Sample a vector $\de$ from $\Nn(0, \Id_P)$
  \State Initialize $a^{(0)} \leftarrow 0$ \hfill *
  \State Initialize $\Dd_a^{(0)} \leftarrow 0$
  \For {$\ell$ from $0$ to $\mathcal{L} - 1$} \hfill *
  \State $\ill{a} \leftarrow \psi( \il{a}, y, \theta)$ \hfill *
  \State $\ill{\Dd_a} \leftarrow \il{\Psi_a}(\il{\Dd_a}) + \il{\Psi_y}(\delta)$
  \EndFor \hfill *
  \State $\iL{x} \leftarrow \gamma(\iL{a})$ \hfill *
  \State $\iL{\Dd_x} \leftarrow \iL{\Gamma}_a(\iL{\Dd_a})$
  \State $\AEDOFMCn \leftarrow
  \dotp{\Phi \iL{\Dd_x}}{ A^*\!A \delta}$
  \State $\ASUREMCn \leftarrow
  \norm{A(y - \Phi \iL{x})}^2
  - \sigma^2 \tr( A^*\!A )
  + 2 \sigma^2 \AEDOFMCn$\\
  \Return {$x(y, \theta) \leftarrow \iL{x}$ and $\AERISK{x}(y, \theta) \leftarrow \ASUREMCn$}
\end{algorithmic}
\rule{\linewidth}{1px}
\end{minipage}
\caption{Pseudo-algorithm for
  risk estimation of an iterative scheme. The symbols * indicate
  the lines corresponding to the computation of $x$. The
  others are dedicated to the computation of the estimated risk
  $\AERISKn$ using Monte Carlo simulation.
  Even if computing the risk requires more operations,
  the global complexity of the algorithm is unchanged.
}
\vspace{-0.5cm}
\label{fig:pseudo-code-gsure}
\end{figure}

\begin{figure}[!t]
\centering
\begin{minipage}{\linewidth}
\rule{\linewidth}{2px}
\begin{algorithmic}[0]
  \State \textbf{Algorithm} Risk and gradient risk estimation of an iterative scheme
\end{algorithmic}
\vspace{-0.2cm}
\rule{\linewidth}{1px}
\begin{algorithmic}[0]
  \State \makebox[3cm][l]{\textbf{Inputs:}} observation $y \in \Yy = \RR^P$, collection of parameters $\theta \in \Theta$
  \State \makebox[3cm][l]{\textbf{Parameters:}} noise variance $\sigma^2 > 0$, linear operator $\Phi \in \RR^{P\times N}$,
  \State \makebox[3cm][l]{}  matrix $A \in \RR^{M \times P}$, number $\mathcal{L}$ of iterations
  \State \makebox[3cm][l]{}  decay parameters $C>0$ and $0 < \alpha < 1$
  \State \makebox[3cm][l]{\textbf{Output:}} solution $x(y, \theta) \in \Xx$, its risk estimate $\AERISK{x}(y, \theta)$,
  \State \makebox[3cm][l]{} and its gradient risk estimate $\hat{\nabla}_2{\ARISK{x}}(y, \theta)$\\
  \State Sample a vector $\de$ from $\Nn(0, \Id_P)$ \hfill *
  \State Choose $\epsilon = C \sigma / P^\alpha$ \hfill *
  \For {$y' = y$ and $y' = y + \epsilon \de$} \hfill *
  \State Initialize $a^{(0)} \leftarrow 0$ \hfill *
  \State Initialize $\Jj_a^{(0)} \leftarrow 0$
  \For {$\ell$ from $0$ to $\mathcal{L} - 1$} \hfill *
  \State $\ill{a} \leftarrow \psi( \il{a}, y', \theta)$ \hfill *
  \State $\ill{\Jj_a} \leftarrow \il{\Psi_a}( \il{\Jj_a} ) + \il{\Psi_\theta}$
  \EndFor \hfill *
  \State $\il{x}(y') \leftarrow \gamma(\il{a})$ \hfill *
  \State $\il{\Jj_x}(y') \leftarrow \il{\Gamma}_a(\il{\Jj_a})$
  \EndFor \hfill *
  \State $\EDOFAMCn \leftarrow
  \frac{1}{\epsilon} \dotp{\Phi (\iL{x}(y+\epsilon \delta) - \iL{x}(y))}{ A^*\!A \delta}$ \hfill *
  \State $\ASUREAMCn \leftarrow
  \norm{A(y - \Phi \iL{x})}^2
  - \sigma^2 \tr( A^*\!A )
  + 2 \sigma^2 \EDOFAMCn$ \hfill *
  \State $\ASUGARAMCn \leftarrow
    2 \iL{\Jj_x}(y)^* \Phi^*\!A^*\!A (\Phi \iL{x} - y)
    + \frac{2 \sigma^2}{\epsilon}\!%
    \left(\iL{\Jj_x}(y\!+\!\epsilon \delta)\!-\!\iL{\Jj_x}(y)\right)^*\!\Phi^*\!A^*\!A \delta$\\
  \Return {$x(y, \theta) \leftarrow \iL{x}(y)$, $\AERISK{x}(y, \theta) \leftarrow \ASUREAMCn$
  and $\hat{\nabla}_2 {\ARISK{x}}(y, \theta) \leftarrow \ASUGARAMCn$}
\end{algorithmic}
\rule{\linewidth}{1px}
\end{minipage}
\caption{Pseudo-algorithm for
  risk and gradient risk estimation of an iterative scheme.
  The symbols * indicate
  the lines corresponding to the computation of $x$ and its
  estimated risk $\ARISKn$ using approximated Monte Carlo simulation,
  i.e., as described in \cite{ramani2008montecarlosure}.
  The others are dedicated to the computation of the estimated gradient
  of the risk $\hat{\nabla}{\ARISKn}$.
  Even if computing the gradient of the risk requires more operations,
  the global complexity of the algorithm is unchanged.}
\vspace{-0.5cm}
\label{fig:pseudo-code-sugar}
\end{figure}

\subsection{Computation of $\bm \ASUREMCn$ for Risk Optimization}
\label{sec:formal_diff_re}

We describe here the iterative computation of
the directional derivative $\JACs{1}{\il{x}}(y, \theta)[\delta]$
following the idea introduced in \cite{vonesch2008sure}
(see Section \ref{sec:formal_mcsure}).
Note that we focus on the directional derivative since, on the one hand,
$\JACs{1}{\il{x}}(y, \theta) \in \RR^{N \times P}$ is never used explicitly but only its trace, not to mention its storage cost,
and, on the other hand, the risk can be estimated by applying only the weak directional derivatives on random directions $\delta$
(see Section \ref{sec:mcsure} for more details).

The next proposition summarizes a recursive scheme to compute
the weak derivatives $\JACs{1}{\il{x}}(y, \theta)[\delta]$.

\begin{prop}\label{prop:gen_diff_y}
For any vector $\delta \in \Xx$, the weak directional derivative $\il{\Dd}_x = \JACs{1}{\il{x}}(y, \theta)[\de]$
is given by
\begin{align*}
  \il{\Dd}_x &= \il{\Gamma}_a(\il{\Dd}_a)\\
  \qwithq
  \ill{\Dd}_a
  &= \il{\Psi_a}( \il{\Dd}_a )
  + \il{\Psi_y}( \de ),
\end{align*}
where $\il{\Dd}_a = \JACs{1}{\il{a}}(y, \theta)[\de]$ and
we have defined the following linear mappings
\begin{align*}
  \il{\Gamma}_a(\cdot) &= \JACs{1}{\gamma}( \il{a})[\cdot],\\
  \il{\Psi_a}(\cdot) &= \JACs{1}{\psi}( \il{a}, y, \theta)[\cdot],\\
  \il{\Psi_y}(\cdot) &= \JACs{2}{\psi}( \il{a}, y, \theta)[\cdot].
\end{align*}

\end{prop}

Plugging $\JACs{2}{\il{x}}(y, \theta)[\de]$ in \eqref{eq-randomized-df}, and in turn in \eqref{eq:sure}, gives iteratively an unbiased\footnote{Expectation is to be taken here with respect to both the Gaussian measure of the noise $W$ and the direction~$\Delta$.} estimate of the risk at the current iterate $\il{x}(y, \theta)$. The whole procedure is summarized in Fig.~\ref{fig:pseudo-code-gsure}. It is worth point out that although estimating the risk entails additional operations, the global complexity is the same as for the original iterative algorithm without risk estimation.

\subsection{Computation of $\bm \ASUGARAMCn$ for Risk Optimization}

We now focus on the computation of the weak Jacobian $\JACs{2}{\il{x}}(y, \theta)$.
Unlike for risk estimation that required only weak directional derivatives,
for risk optimization we need the full weak Jacobian matrix
$\JACs{2}{x}(y, \theta) \in \RR^{\dim(\Theta) \times N}$.
The proposed strategy, known as the forward accumulation, is one
of the possible strategies to iteratively evaluate the derivatives
by the use of the chain rule. The reverse accumulation is another
strategy that does not require computing the full Jacobian
matrix at the expense of a large memory load with respect to
the number of iterations. Between these two extreme approaches, they
are several hybrid strategies that can also be considered, knowing, that
finding the optimal Jacobian accumulation strategy is
an NP-complete problem.
Such strategies have been studied in the field of
``automatic differentiation'' and the reader is invited to refer to
\cite{griewank2008evaluating,naumann2008optimal} for a comprehensive account of these approaches.

In our case, we consider that, unlike for $\JACs{1}{\il{x}}(y, \theta)$,
the matrix $\JACs{2}{x}(y, \theta)$ is in practice quite small
since $\dim(\Theta) \ll P$, hence implying only a memory load overhead of small fraction of $P$. 
Hence following the forward accumulation strategy, we propose
a practical way to compute iteratively the full weak Jacobian matrix
$\JACs{2}{x}(y, \theta)$.


The next result describes an iterative scheme to compute
$\JACs{2}{\il{x}}(y, \theta)$.

\begin{prop}\label{prop:gen_diff_theta}
The weak Jacobian $\il{\Jj}_x = \JACs{2}{\il{x}}(y, \theta)$
is given by
\begin{align*}
  \il{\Jj}_x &= \il{\Gamma}_a(\il{\Jj_a})\\
  \qwithq
  \ill{\Jj_a}
  &= \il{\Psi_a}( \il{\Jj_a} )
  + \il{\Psi_\theta},
\end{align*}
where $\il{\Jj_a} = \JACs{2}{\il{a}}(y, \theta)$ and
we have defined
\begin{align*}
  \il{\Gamma}_a(\cdot) &= \JACs{1}{\gamma}( \il{a})[\cdot],\\
  \il{\Psi_a}(\cdot) &= \JACs{1}{\psi}( \il{a}, y, \theta)[\cdot],\\
  \il{\Psi_\theta} &= \JACs{3}{\psi}( \il{a}, y, \theta).
\end{align*}

\end{prop}

Plugging $\JACs{2}{\il{x}}(y, \theta)$ in the expression of $\ASUGARAMCn$ given by Proposition \ref{prop:sugar} provides iteratively an asymptotically (see Theorem~\ref{thm:sugar}) unbiased estimate of the gradient of the risk at the current iterate $\il{x}(y, \theta)$.
The main steps of the procedure are summarized in Fig.~\ref{fig:pseudo-code-sugar}.
The estimation of the gradient of the risk entails only a small computational overhead
compared to the risk estimation approach of \cite{ramani2008montecarlosure}.
Their respective complexity remains however the same.\\

Note finally that in both schemes, another initialization than $a^{(0)} = 0$ can be chosen, for instance depending on $y$ and $\theta$, in which case the respective derivatives require to be initialized accordingly.\\

The following sections are devoted to instantiate this approach to more specific iterative algorithms that are able to handle non-smooth convex objective functions $E$.

\subsection{Application to Generalized Forward Backward Splitting}

The Generalized Forward Backward (GFB) splitting~\cite{raguet-gfb}
allows one to find one element belonging to the set $x(y, \theta)$
solution of the structured convex optimization problem
\eql{\label{eq-min-gfb}
  x(y, \theta) = \uArgmin{x \in \Xx} \left\{E(x,y, \theta) =  F(x,y,\theta) + \sum_{k=1}^Q G_k(x,y, \theta)\right\}
}
under the assumptions that all functions are proper, closed and convex, $F$ is $C^1(\Xx)$ with $L$-Lipschitz continuous gradient, and the $G_k$ functions are simple, in the sense their proximity operator can be computed in closed form (e.g.~the $\ell_1$ norm is simple since its proximal operator is explicitly the soft-thresholding). Recall that the proximal mapping of a proper closed convex function $G$ is defined as
\[
\Prox_{G}: x \in \Xx \mapsto \uargmin{z \in \Xx} \frac{1}{2}\norm{z - x}^2 + G(z) ~.
\]
It is uniquely valued and non-expansive (i.e.,
$\norm{\Prox_{G}(x_1) - \Prox_{G}(x_2)} \leq \norm{x_1 - x_2}$ for any $x_1, x_2 \in \Xx$),
in fact even firmly so.

The GFB implements iteration \eqref{eq-iterations} with
$\il{a} = (\il{\xi}, \il{z_1}, \ldots, \il{z_Q}) \in \Aa = \Xx^{1+Q}$,
$\il{x} = \gamma(\il{a}) = \il{\xi}$
and $\ill{a} = \psi(\il{a}, y, \theta)$ chosen such that for all $k=1,\ldots,Q$,
\begin{align*}
  \begin{array}{rlll}
  &
  \ill{x} &=& \frac{1}{Q} \sum_{k=1}^Q \ill{z_k}\\
  \qandq &
  \ill{z_k} &=& \il{z_k}  - \il{x} + \Prox_{\nu Q G_k}( \il{\Zz}_k, y, \theta )\\
  \qwithq&
  \il{\Zz}_k &=& 2 \il{x} - \il{z_k} - \nu \nabla_1 F(\il{x}, y, \theta ) ~.
  \end{array}
\end{align*}
With the parameter $\nu \in ]0, 2/L[$,
the sequence of iterates $\il{x}$ is provably guaranteed to converge to a minimizer $x(y, \theta)$ of \eqref{eq-min-gfb}.
One recovers as a special cases the Forward-Backward splitting~\cite{CombettesWajs05} when $Q=1$ and the Douglas-Rachford splitting~\cite{Combettes2007a} when $F=0$.

\begin{cor}\label{cor:gfb_diff_y}
For any vector $\delta \in \Xx$, the GFB weak directional derivatives $\il{\Dd}_x = \il{\Gamma_a}(\il{\Dd}_a)$ and $\ill{\Dd}_a = \il{\Psi_a}( \il{\Dd}_a ) + \il{\Psi_y}( \de )$ are
computed by evaluating iteratively
\begin{align*}
  \begin{array}{rlll}
  &
  \ill{\Dd_x} &=& \frac{1}{Q} \sum_{k=1}^Q \ill{\Dd_{z_k}}\\
  \qandq &
  \ill{\Dd_{z_k}} &=& \il{\Dd_{z_k}}  - \il{\Dd_x} + \il{\Gg_{k,x}}( \il{\Dd_{\Zz_k}} ) +  \il{\Gg_{k,y}} (\de)\\
  \qwithq &
  \il{\Dd_{\Zz_k}} &=& 2 \il{\Dd_x} - \il{\Dd_{z_k}} - \nu ( \il{\Ff_x}(\il{\Dd_{x}}) + \il{\Ff_y}(\de) ) ~,
  \end{array}
\end{align*}
where we have defined the following linear mappings
\begin{align*}
  \begin{array}{rlll}
  & \il{\Gg_{k,x}}(\cdot) &=& \JAC{1}{\Prox_{\nu Q G_k}}( \il{\Zz_k}, y, \theta )[\cdot],\\
  & \il{\Gg_{k,y}}(\cdot) &=& \JAC{2}{\Prox_{\nu Q G_k}}( \il{\Zz_k}, y, \theta )[\cdot],\\
  & \il{\Ff_x}(\cdot) &=& \JAC{1}{\nabla_1 F} ( \il{x}, y, \theta )[\cdot]\\
  \qandq &
  \il{\Ff_y}(\cdot) &=& \JAC{2}{\nabla_1 F} ( \il{x}, y, \theta )[\cdot] ~.
  \end{array}
\end{align*}
\end{cor}

\begin{cor}\label{cor:gfb_diff_theta}
The GFB weak Jacobian $\il{\Jj}_x = \il{\Gamma_a}(\il{\Jj}_a)$, where $\ill{\Jj}_a = \il{\Psi_a}( \il{\Jj}_a ) + \il{\Psi_\theta}$,
is computed by evaluating iteratively
\begin{align*}
  \begin{array}{rlll}
  &
  \ill{\Jj_x} &=& \frac{1}{Q} \sum_{k=1}^Q \ill{\Jj_{z_k}}\\
  \qandq &
  \ill{\Jj_{z_k}} &=& \il{\Jj_{z_k}}  - \il{\Jj_x} + \il{\Gg_{k,x}}( \il{\Jj_{\Zz_k}} ) +  \il{\Gg_{k,\theta}}\\
  \qwithq &
  \il{\Jj_{\Zz_k}} &=& 2 \il{\Jj_x} - \il{\Jj_{z_k}} - \nu (\il{\Ff_x}(\il{\Jj_{x}}) + \il{\Ff_\theta}) ~,
  \end{array}
\end{align*}
where we have defined
\begin{align*}
  \begin{array}{rlll}
  & \il{\Gg_{k,x}}(\cdot) &=& \JAC{1}{\Prox_{\nu Q G_k}}( \il{\Zz_k}, y, \theta )[\cdot],\\
  & \il{\Gg_{k,\theta}} &=& \JAC{3}{\Prox_{\nu Q G_k}}( \il{\Zz_k}, y, \theta ),\\
  & \il{\Ff_x}(\cdot) &=& \JAC{1}{\nabla_1 F} ( \il{x}, y, \theta )[\cdot]\\
  \qandq &
  \il{\Ff_\theta} &=& \JAC{3}{\nabla_1 F} ( \il{x}, y, \theta )~.
  \end{array}
\end{align*}
\end{cor}

\subsection{Application to Primal-dual Splitting}


Proximal splitting schemes can be used to find an element of the set $x(y, \theta)$
defined as the solution of the large class of variational problems
\begin{align}\label{eq-min-cp}
	x(y, \theta) = \uArgmin{x \in \Xx} \left\{E(x,y, \theta) = H(x,y,\theta) + G(K(x),y,\theta)\right\} ~,
\end{align}
where both $x \mapsto H(x,y,\theta)$ and $u \mapsto G(u,y,\theta)$ are proper closed convex and simple functions, and $K : \Xx \to \Uu$ is a bounded linear operator.

The primal-dual\footnote{We invite the interested reader to consult \cite{komodakis2014playing} for a detailed review on primal-dual algorithms.} relaxed Arrow-Hurwicz algorithm as revitalized recently in \cite{chambolle2011first} (that we coin CP) to solve \eqref{eq-min-cp} implements \eqref{eq-iterations} with $\il{a} = (\il{\xi}, \il{\tilde x}, \il{u}) \in \Aa = \Xx^2 \times \Uu$,
$\il{x} = \gamma(\il{a}) = \il{\xi}$
and $\ill{a} = \psi(\il{a}, y, \theta)$ such that
\begin{align}
  \label{eq:primaldual}
  \begin{split}
    \ill{u} &= \Prox_{\tau G^*}( \il{U}, y, \theta )
    \qwhereq \il{U} = \il{u} + \tau K(\il{\tilde x}),  \\
    \ill{x} &= \Prox_{\xi H}(\il{X}, y, \theta)
    \qwhereq \il{X} = \il{x} - \xi K^*(\ill{u}) ,  \\
    \ill{\tilde x} &= \ill{x} + \zeta ( \ill{x} - \il{x} ).
  \end{split}
\end{align}
where the Legendre-Fenchel conjugate of $G$ is defined as $G^*(u,y,\tau) = \max_{z} \dotp{z}{u} - G(z,y,\tau)$, and its proximity operator is
given by Moreau's identity as
\begin{align*}
  \Prox_{\tau G^*}(u,y) = u - \tau \Prox_{G/\tau}(u/\tau,y)~.
\end{align*}
The parameters $\tau>0, \xi>0$ are chosen such that $\tau\xi \norm{K}^2<1$, and $\zeta \in [0,1]$ to ensure provable convergence of $\il{x}$ toward an element in the set $x(y,\theta)$ of \eqref{eq-min-cp}. $\zeta\!=\!0$ corresponds to the Arrow-Hurwitz algorithm, and for $\zeta\!=\!1$, a sublinear $O(1/\ell)$ convergence rate on the partial duality gap was established in \cite{chambolle2011first}.

\begin{cor}\label{cor:pds_diff_y}
For any vector $\delta \in \Xx$, the CP weak directional derivatives $\il{\Dd}_x = \il{\Gamma_a}(\il{\Dd}_a)$ and $\ill{\Dd}_a = \il{\Psi_a}( \il{\Dd}_a ) + \il{\Psi_y}( \de )$ are computed by evaluating iteratively
\begin{align*}
  \begin{array}{rlll}
    & \ill{\Dd_{u}} &=& \il{\Gg_u}( \il{\Dd_{U}} ) + \il{\Gg_y}(\de)
    \qwhereq
    \il{\Dd_{U}} = \il{\Dd_{u}} + \tau K(\il{\Dd_{\tilde x}}),\\
    & \ill{\Dd_x} &=& \il{\Hh_{x}}( \il{\Dd_{X}} ) + \il{\Hh_{y}} (\de)
    \qwhereq
    \il{\Dd_{X}} = \il{\Dd_{x}} - \xi K^*(\ill{\Dd_{u}}),\\
    \qandq &
    \ill{\Dd_{\tilde x}} &=& \ill{\Dd_x} + \zeta ( \ill{\Dd_x} - \il{\Dd_x} )
    \end{array}
\end{align*}
where we have defined the following linear mappings
\begin{align*}
  \begin{array}{rlll}
  & \il{\Hh_{x}}(\cdot) &=& \JAC{1}{\Prox_{\xi H}}( \il{X}, y, \theta )[\cdot],\\
  & \il{\Hh_{y}}(\cdot) &=& \JAC{2}{\Prox_{\xi H}}( \il{X}, y, \theta )[\cdot],\\
    & \il{\Gg_{u}}(\cdot) &=& \JAC{1}{\Prox_{\tau G^*}}( \il{U}, y, \theta )[\cdot]\\
  \qandq &
  \il{\Gg_{y}}(\cdot) &=& \JAC{2}{\Prox_{\tau G^*}}( \il{U}, y, \theta )[\cdot]~.
  \end{array}
\end{align*}
\end{cor}

\begin{cor}\label{cor:pds_diff_theta}
Similarly to Corollary \ref{cor:pds_diff_y},
the CP weak Jacobians $\il{\Jj}_x = \il{\Gamma_a}(\il{\Jj}_a)$ and $\ill{\Jj}_a = \il{\Psi_a}( \il{\Jj}_a ) + \il{\Psi_\theta}$ are
computed by evaluating iteratively
\begin{align*}
  \begin{array}{rlll}
    & \ill{\Jj_{u}} &=& \il{\Gg_u}( \il{\Jj_{U}} ) + \il{\Gg_\theta}
    \qwhereq
    \il{\Jj_{U}} = \il{\Jj_{u}} + \tau K(\il{\Jj_{\tilde x}}),\\
    & \ill{\Jj_x} &=& \il{\Hh_{x}}( \il{\Jj_{X}} ) + \il{\Hh_{\theta}}
    \qwhereq
    \il{\Jj_{X}} = \il{\Jj_{x}} - \xi K^*(\ill{\Jj_{u}}),\\
    \qandq &
    \ill{\Jj_{\tilde x}} &=& \ill{\Jj_x} + \zeta ( \ill{\Jj_x} - \il{\Jj_x} )
    \end{array}
\end{align*}
where we have defined
\begin{align*}
  \begin{array}{rlll}
  & \il{\Hh_{x}}(\cdot) &=& \JAC{1}{\Prox_{\xi H}}( \il{X}, y, \theta )[\cdot],\\
  & \il{\Hh_{\theta}} &=& \JAC{3}{\Prox_{\xi H}}( \il{X}, y, \theta ),\\
  & \il{\Gg_{u}}(\cdot) &=& \JAC{1}{\Prox_{\tau G^*}}( \il{U}, y, \theta )[\cdot]\\
  \qandq &
  \il{\Gg_{\theta}} &=& \JAC{3}{\Prox_{\tau G^*}}( \il{U}, y, \theta ).
  \end{array}
\end{align*}
\end{cor}

Note that the two proximal splitting schemes described here were chosen for their flexibility and the richness of the class of problems they can handle. Obviously, the methodology and discussion extend easily to the reader's favorite proximal splitting algorithm.

%% file: sections/numeric.tex
\section{Examples and Numerical Results}
\label{sec:numeric}

In this section, we exemplify the use of the
formal differentiation of iterative
proximal splitting algorithms for
three popular variational problems:
nuclear norm regularization,
total-variation regularization and
multi-scale wavelet $\ell_1$-analysis sparsity prior. 
For each of them,
the expressions of all quantities including
the proximal operators and their derivatives
are given in closed-form.
On each problem, we illustrate
the usefulness of our gradient risk estimators
for (multi) continuous parameter optimization.

\subsection{Implementation Details}

All experiments reported below are based on the algorithms
detailed in Figure~\ref{fig:pseudo-code-gsure} and
\ref{fig:pseudo-code-sugar} in conjunction with
proximal splitting algorithms presented in the
previous section. The step of the finite difference
is chosen as $\epsilon = 2 \sigma / P^{0.3}$. 
Iterative proximal splitting algorithms
will be used with $\Ll=100$ iterations.
For quasi-Newton optimization, we used the
Broyden-Fletcher-Goldfarb-Shanno (BFGS) method
with the implementation of \cite{lewis2009nonsmooth}.
The use of BFGS is just given as an example to make the most of the proposed gradient estimator of the risk
and seems to be good enough for the cases considered in all the following experiments examples.
Of course, other first-order optimization methods can be considered for just as well,
essentially if the risk presents several local minima.

An important issue in using quasi-Newton optimization
is the choice of the initialization, the initial step and the stopping
criteria. For a variation regularization problem expressed as
\begin{align}
  \uArgmin{x} \frac{1}{2}\norm{\Phi x - y}^2 + \sum_{k=1}^K \la^k \Rr^k(x) ~,
\end{align}
where $\lambda^k > 0, ~ \forall k \in \NN$, the initialization $\lambda^k_0$ is chosen empirically as
\begin{align}
  \lambda_0^k
  = \frac{P \sigma^2}{4 \sum_{k=1}^K \Rr^k(x_{\mathrm{LS}}(y))} ~,
\end{align}
where $x_{\mathrm{LS}}(y)$ is the least-square estimator.
At the first iteration, the approximate inverse Hessian $B_1$
should be chosen such that, for all $k > 0$,  $\lambda_1^k$ is of the same order as $\lambda_0^k$.
To this end, we suggest initializing $B_1$ as a diagonal matrix with diagonal entries
\begin{align}
  B_1^k = \left|\frac{\alpha \lambda_0^k}{\ASUGARAMC{x}(y , \lambda_0, \delta, \epsilon)_k}\right|
\end{align}
such that, for all $k$,
$\lambda_1^k = (1 \pm \alpha) \lambda_0^k$, where,
in practice, we have chosen $\alpha = 0.9$.
Finally, the BFGS method stops after the following criterion is reached
\begin{align}
  \frac{
    \norm{
    \ASUGARAMC{x}(y , \lambda_n, \delta, \epsilon)
    }_\infty
  }{
    \norm{
    \ASUGARAMC{x}(y , \lambda_0, \delta, \epsilon)
    }_\infty
  }
  \leq \tau
\end{align}
where we have chosen $\tau = 0.02$ meaning that
the algorithm stops if all (weak) partial derivatives are
at least $50$ times lower than the maximal one at initialization.

For the sake of reproducibility, the Matlab scripts implementing
the $\SUREn$ and $\SUGARn$ for the different problems details hereafter are available online
at \url{http://www.math.u-bordeaux1.fr/~cdeledal/sugar.php}.

\subsection{Nuclear Norm Regularization}


We consider the recovery of a low-rank matrix $x_0 \in \RR^{n_1 \times n_2}$
from an observation $y \in \RR^P$ of $Y = \Phi x_0 + W$, $W \sim \Nn(0, \sigma^2 \Id_P)$,
where we have identified the matrix space $\RR^{n_1 \times n_2}$
to the vector space $\RR^{N}$ with $N = n_1 n_2$.
To this end, we consider
the following spectral regularization problem
\begin{align}\label{eq-optim-nuclear}
  x^\star(y, \lambda) \in \uArgmin{x} \frac{1}{2}\norm{\Phi x - y}^2 + \la \norm{ x }_* ~,
\end{align}
where $\lambda > 0$ and $\norm{\cdot}_*$ is the nuclear norm (a.k.a., trace for the symmetric semi-definite positive case or Schatten $1$-norm). This is a spectral function defined as the $\ell_1$ norm of the singular values $\La_x \in \RR^{n = \min(n_1, n_2)}$, i.e.
\begin{align*}
  \norm{ x }_* = \norm{\La_x}_1~.
\end{align*}
The nuclear norm is a particular case of spectral regularization that
accounts for prior knowledge on the spectrum of $x$,
typically low-rank (see, e.g., \cite{Fazel02}).
It is the convex hull of the rank function restricted to the unit spectral ball~\cite{CandesCompletion08}.
The parameter $\lambda$ balances the sparsity of the spectrum of the recovered matrix, and the tolerated amount of noise.
However, except in the random measurements setting, there is no direct relation between $\lambda$ and the rank of $x(y, \lambda)$. The optimal value
of $\lambda$ depends indeed on $x_0$, $\Phi$ and $\sigma$ confirming the importance of automatic selection procedures.

Problem \eqref{eq-optim-nuclear} is a special instance of
\eqref{eq-min-gfb} with the parameter $\la = \theta \in \Theta = \RR^+$,
$Q=1$, and
\begin{align*}
  F(x, y, \la) &= \frac{1}{2}\norm{\Phi x - y}^2\\
  \qandq
  G_1(x,y,\la) &= \lambda \norm{x}_* ~.
\end{align*}
Hence the GFB algorithm\footnote{which corresponds in this case where $Q=1$ to the forward-backward algorithm.}
can be used to solve \eqref{eq-optim-nuclear} by setting
\begin{align*}
  \nabla_1 F(x,y,\la) &= \Phi^* (\Phi x - y)\\
  \qandq
  \Prox_{\tau G_1}(x, y, \la) &= V_x \diag( \ST(\La_x, \tau \la) ) U_x^*
\end{align*}
where $\diag: \RR^n \to \RR^{n_1 \times n_2}$ maps the entries of
a vector in $\RR^n$ to the main diagonal of a
rectangular matrix in $\RR^{n_1 \times n_2}$ filled with 0 elsewhere,
$(V_x, U_x, \La_x) \in \RR^{n_1 \times n_1} \times \RR^{n_2 \times n_2} \times \RR^{n}$ is the singular value decomposition (SVD) of
$x$ such that $x = V_x \diag(\La_x) U_x^*$ and
$\ST$ is the soft-thresholding operator \eqref{eq:st}.
Corollary~\ref{cor:gfb_diff_y} and \ref{cor:gfb_diff_theta} can then be applied using,
for any $\de_x \in \Xx$ and $\de_y \in \Yy$, the relations
\begin{align*}
  \begin{array}{rlll}
  & \JAC{1}{\nabla_1 F}(x,y,\la)[\de_x] &=& \Phi^* \Phi \de_x,\\
  & \JAC{2}{\nabla_1 F}(x,y,\la)[\de_y] &=& - \Phi^* \de_y,\\
  & \JAC{3}{\nabla_1 F}(x,y,\la) &=& 0
  \\\\[-0.3cm]
  \qandq
  & \JAC{1}{\Prox_{\tau G_1}}( x,y,\la )[\de_x] &=&
  V_x (\Hh(\La_x)[\bar \de_x]
  + \Gamma_S(\La_x)[\bar \de_x]
  + \Gamma_A(\La_x)[\bar \de_x]) U_x^*,\\
  & \JAC{2}{\Prox_{\tau G_1}}( x,y,\la )[\de_y] &=& 0,\\
  & \JAC{3}{\Prox_{\tau G_1}}( x,y,\la ) &=& V_x \diag( \JACs{2}{\ST}(\La_x, \tau \la) ) U_x^*
  \end{array}
\end{align*}
where $\bar \de_x = V_X^* \de_x U_X \in \RR^{n_1 \times n_2}$,
$\Hh(\La_x)$ is defined as
\begin{align*}
  \Hh(\La_x)[\bar \de_x] = \diag(\JACs{1}{\ST}(\La_x, \rho \la)[\diag(\bar \de_x)])
\end{align*}
and $\Ga_S(\La_x)$ and $\Ga_A(\La_x)$ are defined,
for all $1 \leq i \leq n_1$ and $1 \leq j \leq n_2$, as
\begin{align*}
  \Ga_S(\La_x)[\bar \de_x]_{i,j} &=
     \frac{(\bar \de_x)_{i,j} + (\bar \de_x)_{j,i}}{2}
     \times
     \choice{
       0 & \text{if} \quad i=j\\
       \frac{\ST(\La_x, \rho \la)_i-\ST(\La_x, \rho \la)_j}{(\La_x)_i-(\La_x)_j}
       & \text{if} \quad (\La_x)_i \neq (\La_x)_j \\
       \JACs{1}{\ST}(\La_x, \rho \la)_{i,i}
       & \text{otherwise},
     } \\
  \Ga_A(\La_x)[\bar \de_x]_{i,j} &=
  \frac{(\bar \de_x)_{i,j} - (\bar \de_x)_{j,i}}{2}
  \times
  \choice{
    0 & \text{if} \quad i=j\\
    \frac{\ST(\La_x, \rho \la)_i+\ST(\La_x, \rho \la)_j}{(\La_x)_i+(\La_x)_j}
    & \text{if} \quad (\La_x)_i > 0 \quad \text{or} \quad (\La_x)_j > 0\\
    \JACs{1}{\ST}(\La_x, \rho \la)_{i,i}
    & \text{otherwise},
  }
\end{align*}
where for $i > n$ we have extended $\La_x$ and $\ST(\La_x, \rho \la)$ as
$(\La_x)_i = 0$ and $\ST(\La_x, \rho \la)_i = 0$,
and for $j > n_1$ or $i > n_2$,
$\bar \de_x$ as $(\bar \de_x)_{j,i} = 0$.
Recall from \eqref{eq:diff_st}~that the weak derivatives of the soft-thresholding are defined, for
$t \in \RR^{N}$, $\rho > 0$, $\de_t \in \RR^{N}$,
$1 \leq i \leq N$, by
\begin{align}\label{eq:diff_soft_thresholding}
  \JACs{1}{\ST}(t, \rho)_{i,i} &=
  \choice{
    0 &\qifq |t_i| \leq \rho \\
    1 &\quad\text{otherwise},
  }\\
  \JACs{1}{\ST}(t, \rho)[\de_t]_i &=
  \JACs{1}{\ST}(t, \rho)_{i,i} \times (\de_t)_i
  \nonumber\\
  \qandq \quad
  \JACs{2}{\ST}(t, \rho)_i &=
  \choice{
    0 &\qifq |t_i| \leq \rho \\
    -\sign(t_i) &\quad\text{otherwise}.
  }
  \nonumber
\end{align}

The closed-form expression we derived for
$\JAC{1}{\Prox_{\tau G_1}}( x,y,\la )[\de_x]$ is far from trivial. It is
essentially due to \cite{lewis-twice-spectral,Sun02,edelman-handout}, see \cite{Candes12} for an expression similar to ours.
The generalization of this result to other matrix-valued spectral function has been studied in \cite{deledalle2012risk}.\\

\paragraph{Application to matrix completion}
We illustrate the nuclear norm regularization on a matrix completion problem encountered in recommendation systems such as the popular Netflix problem \cite{bennett2007netflix}.
We therefore consider $y \in \RR^P$ with the forward model $Y = \Phi x_0 + W$, $W \sim \Nn(0, \sigma^2 \Id_P)$,
where $x_0$ is a dense but low-rank (or approximately so) matrix and $\Phi$ is a binary masking operator.

We have taken $(n_1,n_2)=(1000,100)$ and $P=25000$ observed entries (i.e., $25\%$).
The underlying matrix $x_0 = V_{x_0} \diag{\La_{x_0}} U^*_{x_0}$ has been chosen with $V_{x_0}$ and $U_{x_0}$ two realizations of the uniform distribution of orthogonal matrices and $\La_{x_0} = (k^{-1})_{1\leq k \leq n}$ such that $x_0$ is approximately low-rank with a rapidly decaying spectrum.
The binary masking operator is such that for $i = 1, \ldots, P$,
$(\Phi x)_i = x_{\Sigma(i)_1,\Sigma(i)_2}$ where
$\Sigma : [1, \ldots, n_1 \times n_2] \rightarrow [1, \ldots, n_1] \times [1, \ldots, n_2]$
is the realization of a random permutation of the $n_1 \times n_2$ entries of $x$.
The standard deviation $\sigma$ has been set such that
the resulting minimum least-square estimate
$x_{\mathrm{LS}}(y) = \Phi^* y$
has a relative error $\norm{x_{\mathrm{LS}}(y)-x_0}_F/\norm{x_0}_F=0.9$.

\begin{figure}[t]
  \centering
  \subfigure[]{\includegraphics[width=0.48\linewidth]{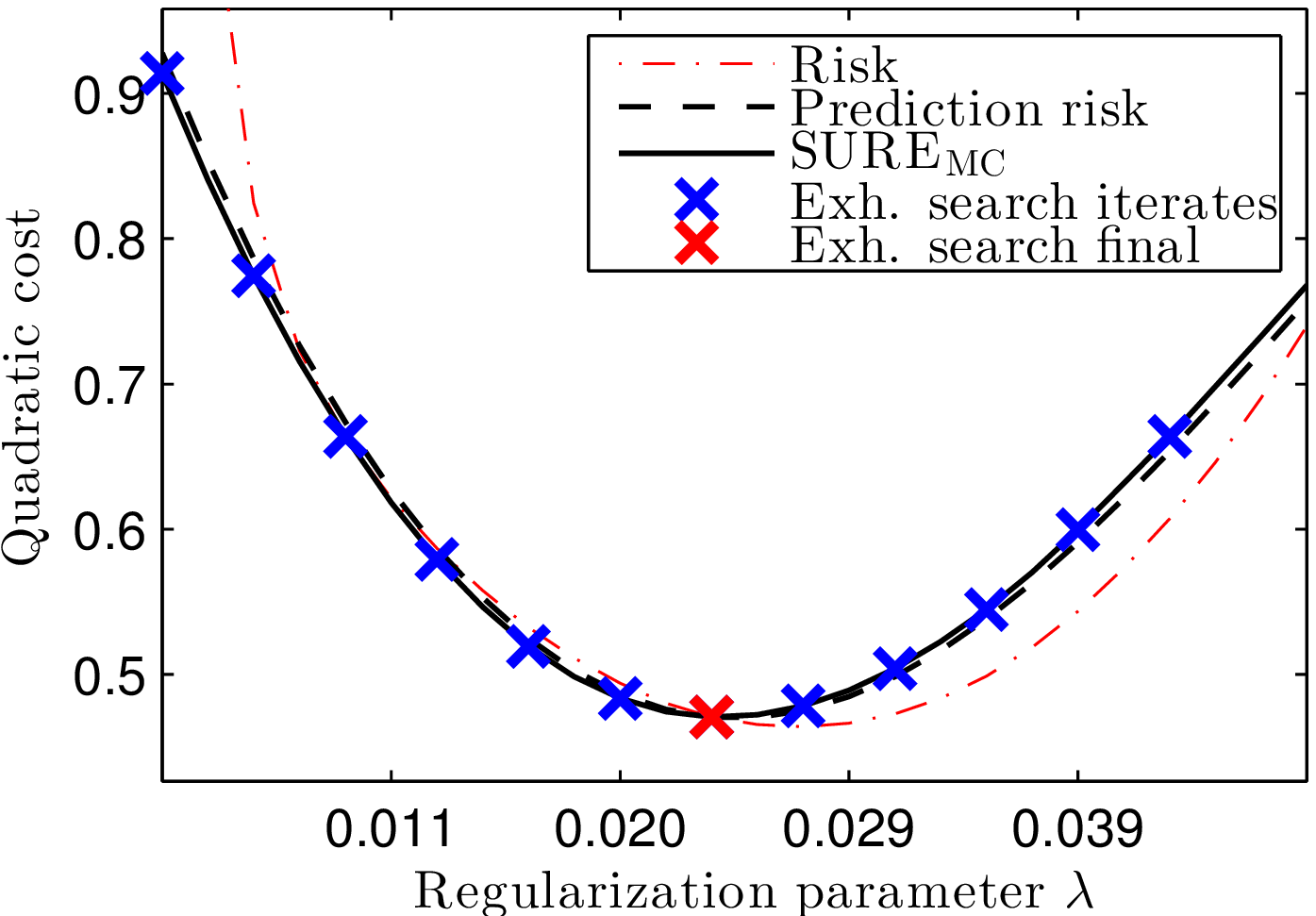}}
  \hfill
  \subfigure[]{\includegraphics[width=0.48\linewidth]{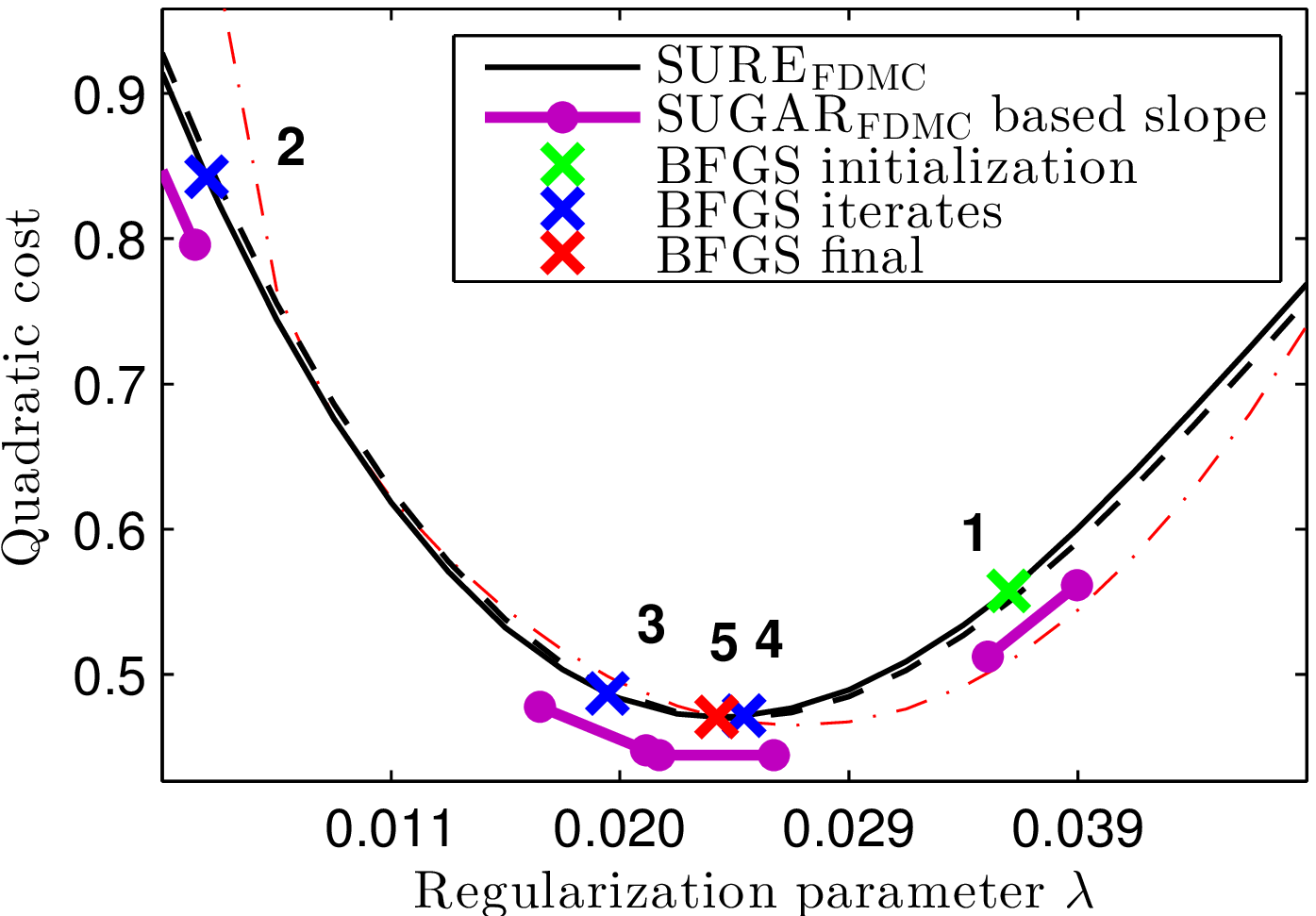}}\\
  \subfigure[]{\includegraphics[height=0.32\linewidth]{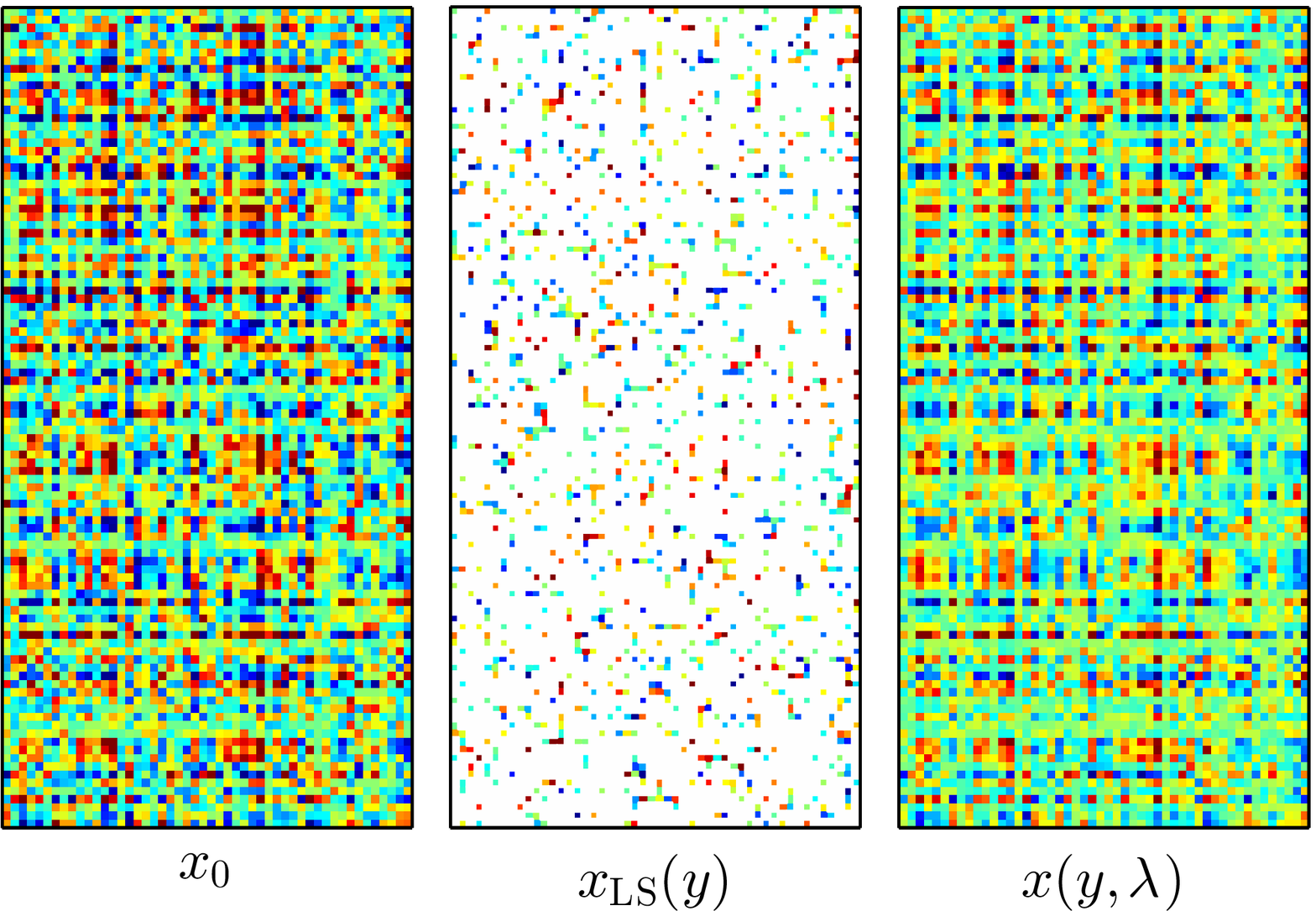}}
  \subfigure{\includegraphics[height=0.32\linewidth]{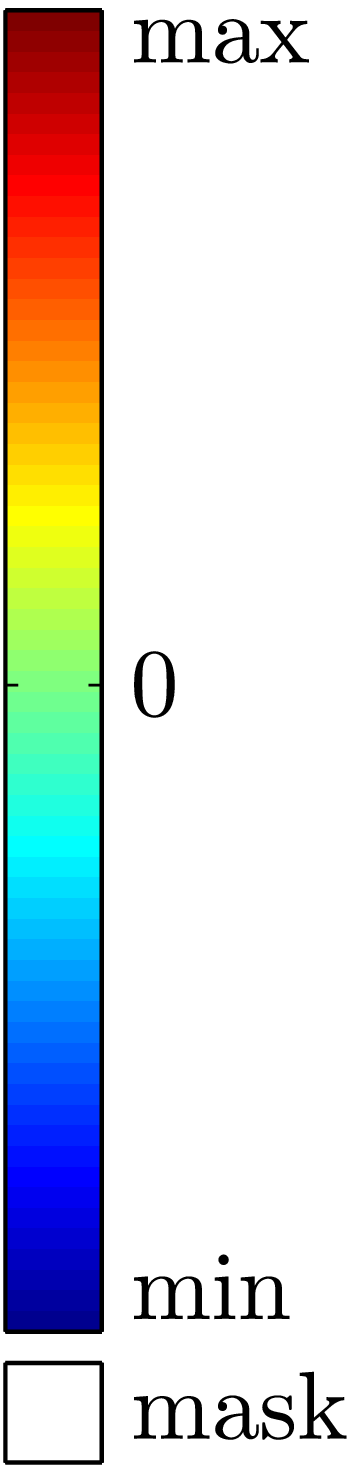}}
  \setcounter{subfigure}{3}%
  \hfill
  \subfigure[]{\includegraphics[height=0.32\linewidth]{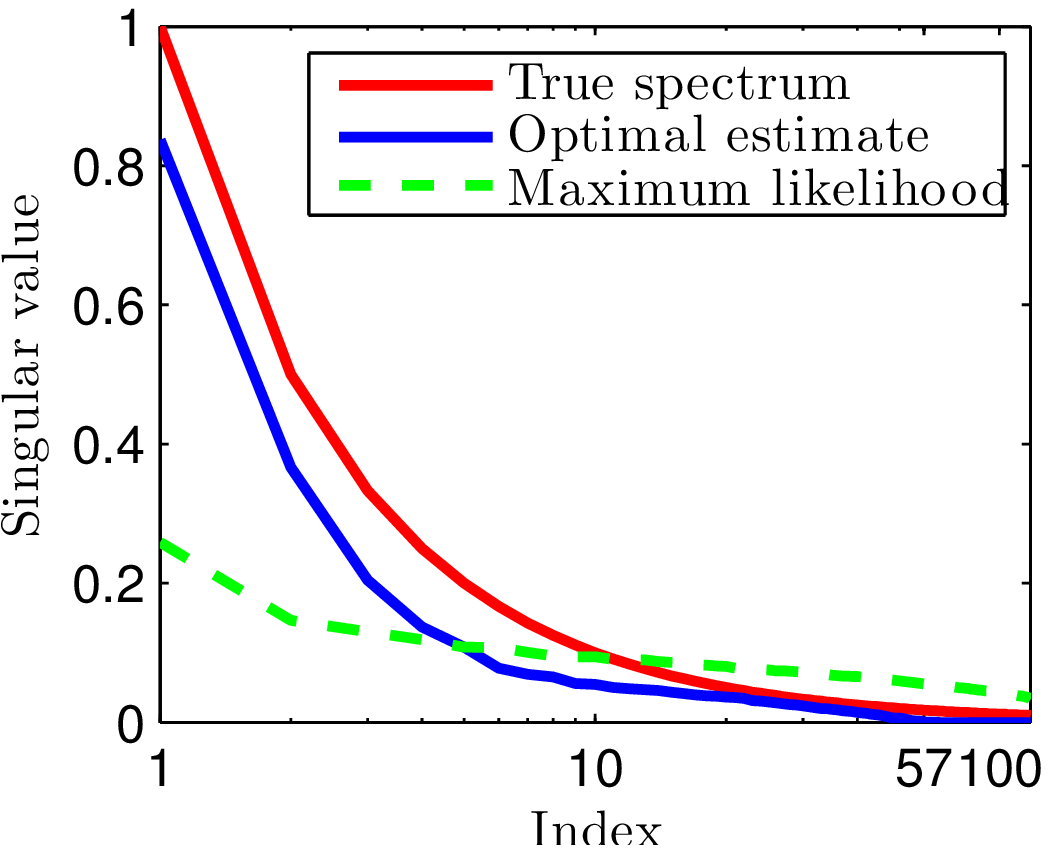}}\\
  \caption{(a-b) Risk, prediction risk and its $\SUREn$
    estimates$^{\ref{footnote:shift}}$ as a function
    of the regularization parameter $\lambda$.
    (a) The $12$ points where $\SUREMC{x}(y, \lambda, \delta)$ has been
    evaluated by exhaustive search. (b) The $5$ evaluation
    points of $\SUREAMC{x}(y, \lambda, \delta, \epsilon)$ and $\SUGARAMC{x}(y, \lambda, \delta, \epsilon)$
    required by BFGS to reach the optimal one.
    (c-d) Respectively, a close in and the spectrum
    of the underlying matrix $x_0$, the
    least-square estimate $x_{\mathrm{LS}}(y)$ and
    the solution $x(y, \lambda)$ at the optimal $\lambda$.
  }
  \label{fig:matrix_completion_sure}
\end{figure}

Figure \ref{fig:matrix_completion_sure}.(a) and (b) depict the risk,
the prediction risk and the $\SUREn=\ASUREn$ (with $A = \Id$)
estimates\footnote{Without impacting the optimal choice of $\la$, the curves have been rescaled for visualization purposes.\label{footnote:shift}} as a function of $\la$ obtained from a single realization of $y$ and $\de$.
In (a), $\SUREMC{x}(y, \lambda, \delta)$ has been
evaluated for $12$ values of $\la$ chosen in a suitable tested range
using the algorithm given in Figure~\ref{fig:pseudo-code-gsure}.
Figure (b), shows the benefit of computing $\SUREAMC{x}(y, \lambda, \delta, \epsilon)$
and $\SUGARAMC{x}(y, \lambda, \delta, \epsilon)$, as described in Figure~\ref{fig:pseudo-code-sugar}, to realize a quasi-Newton optimization.
The sequence of iterates $\lambda_{n}$
is represented as well as the sequence of the slopes
of $\SUREAMC{x}(y, \lambda_{n}, \delta, \epsilon)$ given by $\SUGARAMC{x}(y, \lambda_n, \delta, \epsilon)$.
The BFGS algorithm reaches
to the optimal value in $5$ iterations only.
One can also notice that $\SUREAMC{x}(y, \lambda, \delta, \epsilon)$ and $\SUREMC{x}(y, \lambda, \delta)$
are both good, and visually equivalent, estimators
of the prediction risk.
At the optimum value $\lambda^\star$ minimizing the $\SUREn$,
the true risk is not too far from its minimum showing that,
in this case, the prediction risk is indeed a good objective in order to
minimize the risk.
In Figure~\ref{fig:matrix_completion_sure}.(c)
a close in on the solution $x(y, \lambda^\star)$
is compared to $x_0$ and $x_{\mathrm{LS}}(y)$,
and their respective spectrum in  Figure~\ref{fig:matrix_completion_sure}.(d).
The solution $x(y, \lambda^\star)$ has a rank of $57$
with a relative error of $0.45$ (i.e., a gain of about a factor $2$ w.r.t.~the least-square estimator).

\subsection{Total-Variation Regularization}

We consider the recovery of a piece-wise constant two dimensional
image $x_0 \in \RR^{n_1 \times n_2}$
from an observation $y$ of $Y = \Phi x_0 + W \in \RR^P$, $W \sim \Nn(0, \sigma^2 \Id_P)$,
where we have identified the image space $\RR^{n_1 \times n_2}$
to the vector space $\RR^{N}$ with $N = n_1 n_2$.
To this end, we suggest using
(isotropic) total variation regularization of the form
\begin{align}\label{eq-optim-tv}
  x^\star(y, \lambda) \in \uArgmin{x} \frac{1}{2}\norm{\Phi x - y}^2 + \la \| \tilde{\nabla} x \|_{1,2} ~,
\end{align}
where $\lambda>0$ and $\tilde{\nabla} : \RR^N \to \RR^{N \times 2}$ is the two-dimensional discrete gradient operator.
The $\ell^1$-$\ell^2$ norm of a vector field $t = (t_i)_{i=1}^N \in \RR^{N \times 2}$,
with $t_i \in \RR^2$, is defined as $\norm{t}_{1,2} = \sum_i \norm{t_i}$.
Total-variation promotes the sparsity of the gradient field which
turns out to be a prior that enforces smoothing while
preserving edges.
The parameter $\lambda$ controls the regularity of the image.
A large value of $\lambda$ results to an image with large
homogeneous areas while a small value results to an image
with several small disconnected regions. The optimal value
of $\lambda$ is image and degradation dependent revealing
the importance of automatic selection procedures.



Problem \eqref{eq-optim-tv} is a special instance of
\eqref{eq-min-gfb} using $x=(f,u) \in \Xx = \RR^{N} \times \RR^{N \times 2}$,
the parameter $\la = \theta \in \Theta = \RR^+$,
 $Q=2$ simple functionals, and for $x=(f,u)$
\begin{align*}
  F(x, y, \la) &= \frac{1}{2}\norm{\Phi f - y}^2,\\
  G_1(x,y, \la) &= \la \norm{u}_{1,2}\\
  \qandq
  G_2(x,y, \la) &= \iota_{\Cc}(x)
  \qwhereq
  \Cc = \enscond{x=(f,u)}{u=\tilde{\nabla} f} ~.
\end{align*}
Hence the GFB algorithm can be used to solve \eqref{eq-optim-tv}
using
\begin{align*}
  \nabla_1 F(x,y,\la) &= (\Phi^* (\Phi f - y), 0),\\
  \Prox_{\tau G_1}(x, y, \la) &= (f, \ST_{1,2}(u, \tau\la))\\
  \qandq
  \Prox_{\tau G_2}(x, y, \la) &=
  ((\Id + \Delta)^{-1} (f + \diverg u), \tilde{\nabla} (\Id + \Delta)^{-1} (f + \diverg u))
\end{align*}
where $\Delta$ is the Laplacian operator and $\diverg$ is the discrete divergence operator such that $\diverg=-\tilde{\nabla}^*$.
The operator $\ST_{1,2}$ is the component-wise $\ell^1$-$\ell^2$ soft-thresholding defined, for any dimensions $N$ and $D$, $t \in \RR^{N\times D}$ and $\rho > 0$, by
\begin{align}
  \label{eq:cw_soft_thresholding}
  \ST_{1,2}(t, \rho)_i
  =
  \choice{
    0 & \qifq \norm{t_i} \leq \rho \\
    t_i-\rho \; t_i/\norm{t_i} & \quad\text{otherwise}
  },
  \quad
  \text{for all } 1 \leq i \leq N.
\end{align}
For $D = 1$, the component-wise $\ell^1$-$\ell^2$ soft-thresholding reduces to \eqref{eq:st}.
Corollary \ref{cor:gfb_diff_y} and \ref{cor:gfb_diff_theta} can then be applied,
for any $\de_x = (\de_f, \de_u) \in \Xx$ and $\de_y \in \Yy$, using the relations
\begin{align*}
  \begin{array}{llll}
  & \JAC{1}{\nabla_1 F}(x,y,\la)[\de_x]&=&(\Phi^* \Phi \de_f, 0),\\
  & \JAC{2}{\nabla_1 F}(x,y,\la)[\de_y] &=&(- \Phi^* \de_y, 0),\\
  & \JAC{3}{\nabla_1 F}(x,y,\la) &=&(0, 0),\\\\[-0.3cm]
  & \JAC{1}{\Prox_{\tau G_1}}( x,y,\la )[\de_x] &=&(\de_f, \JACs{1}{\ST_{1,2}}(u, \tau\la)[\de_u]),\\
  & \JAC{2}{\Prox_{\tau G_1}}( x,y,\la )[\de_y] &=&(0, 0),\\
  & \JAC{3}{\Prox_{\tau G_1}}( x,y,\la ) &=&(0, \JACs{2}{\ST_{1,2}}(u, \tau\la))\\\\[-0.3cm]
  \qandq
  & \JAC{1}{\Prox_{\tau G_2}}( x,y,\la )[\de_x] &=& ((\Id + \Delta)^{-1}(\de_f + \diverg \de_u),
  \tilde{\nabla} (\Id + \Delta)^{-1}(\de_f + \diverg \de_u)),\\
  & \JAC{2}{\Prox_{\tau G_2}}( x,y,\la )[\de_y] &=&(0, 0),\\
  & \JAC{3}{\Prox_{\tau G_2}}( x,y,\la ) &=&(0, 0)
  \end{array}
\end{align*}
where the weak derivatives of the component-wise $\ell^1$-$\ell^2$ soft-thresholding are defined, for any
dimensions $N$ and $D$, $t \in \RR^{N \times D}$, $\rho > 0$
and $\de_t \in \RR^{N \times D}$, by
\begin{align}\label{eq:cw_diff_soft_thresholding}
  \JACs{1}{\ST_{1,2}}(t, \rho)[\de_t]_i &=
  \choice{
    0 &\qifq \norm{t_i} \leq \rho \\
    \de_{t,i} - \frac{\rho}{\norm{t_i}}P_{t_i}(\de_{t,i}) &\quad\text{otherwise}
  }\\
  \qandq \quad
  \JACs{2}{\ST_{1,2}}(t, \rho)_i &=
  \choice{
    0 &\qifq \norm{t_i} \leq \rho \\
    -t_i/\norm{t_i} &\quad\text{otherwise}
  }\nonumber
\end{align}
where $P_{\al}$ is the orthogonal projector on $\al^\bot$ for $\al \in \RR^2$.
\\

\begin{figure}[t]
  \centering
  \subfigure[]{\includegraphics[width=0.48\linewidth]{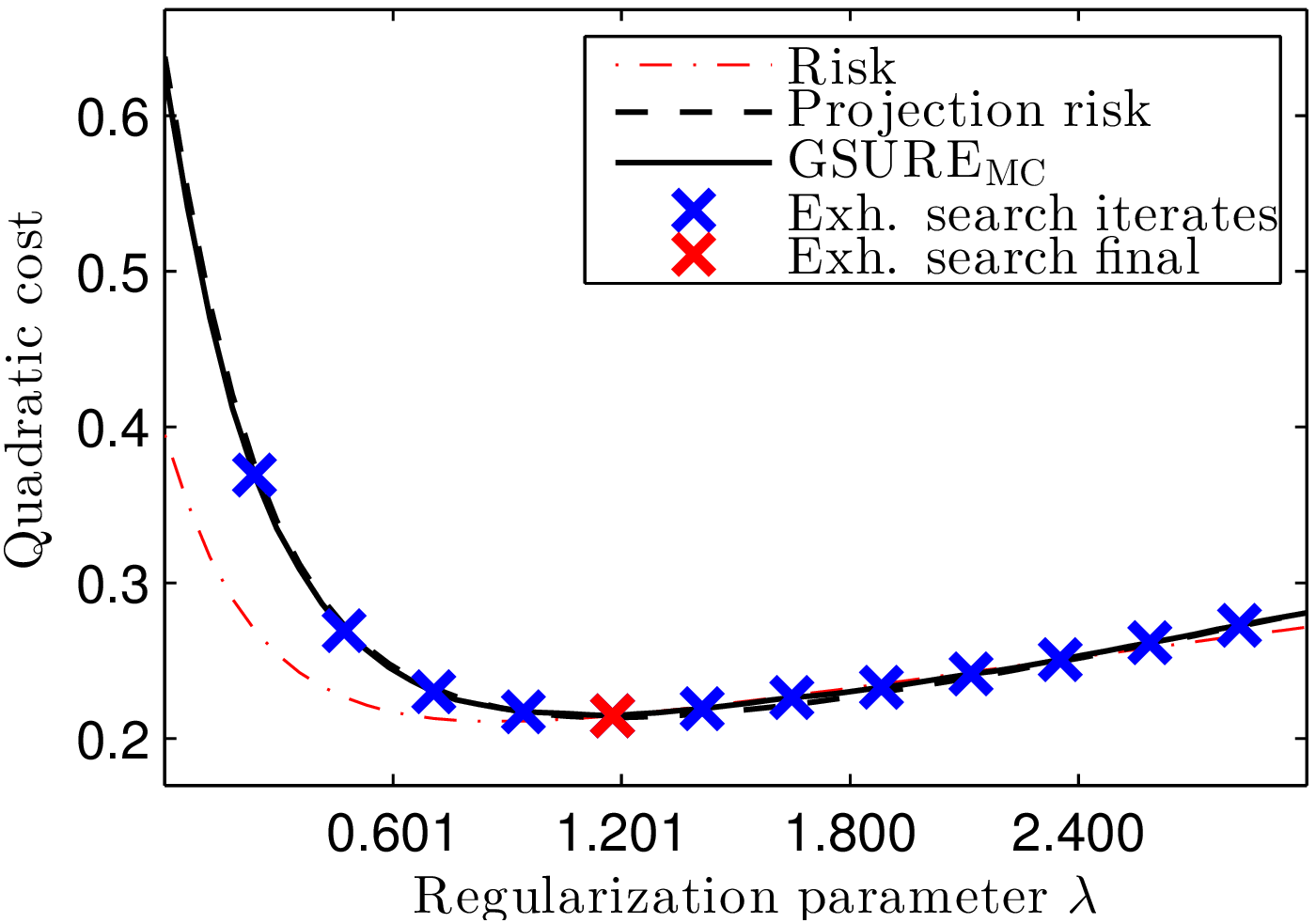}}
  \hfill
  \subfigure[]{\includegraphics[width=0.48\linewidth]{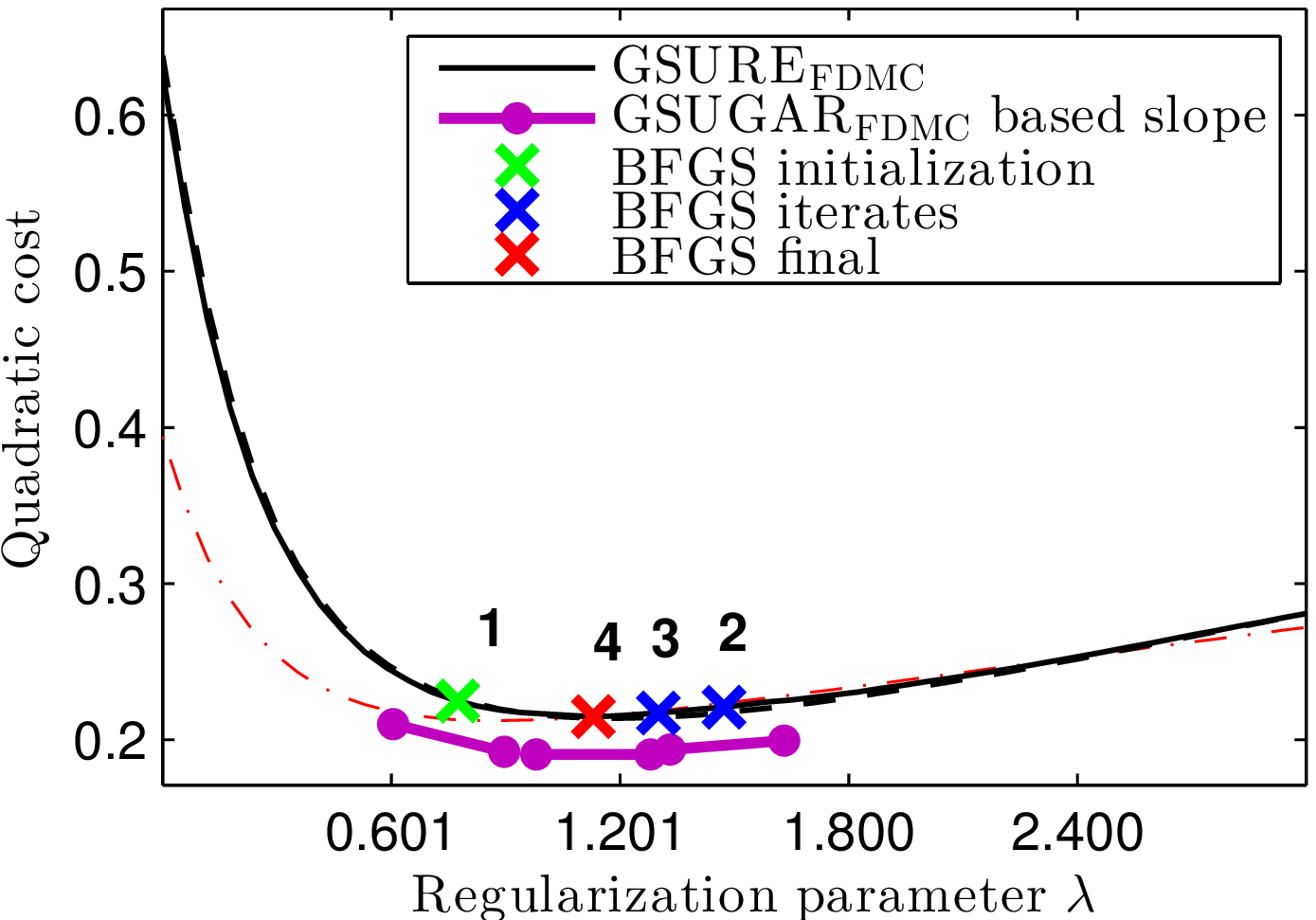}}\\
  \subfigure[]{\includegraphics[width=0.32\linewidth,viewport=200 320 392 512,clip]{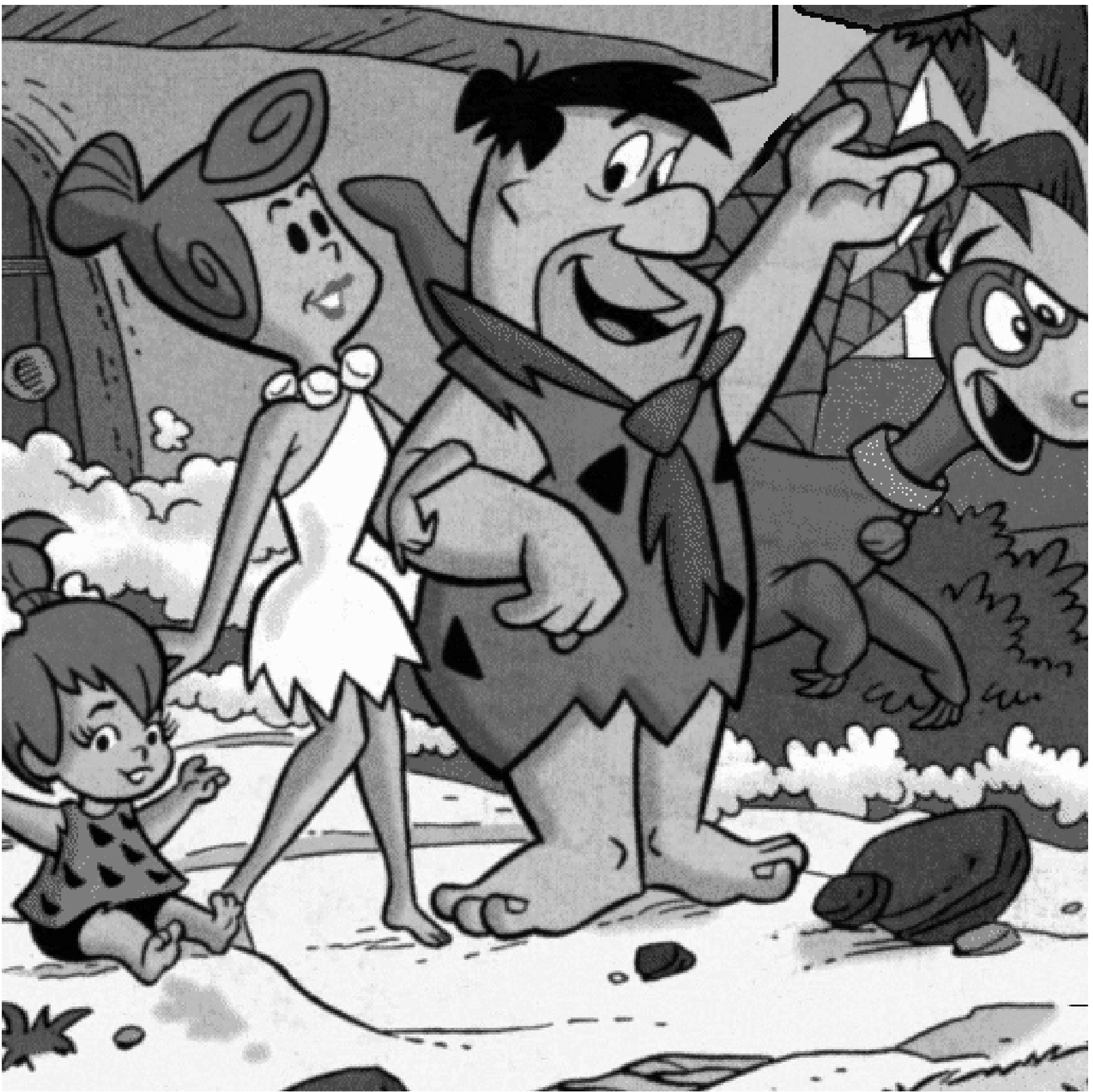}}
  \hfill
  \subfigure[]{\includegraphics[width=0.32\linewidth,viewport=200 320 392 512,clip]{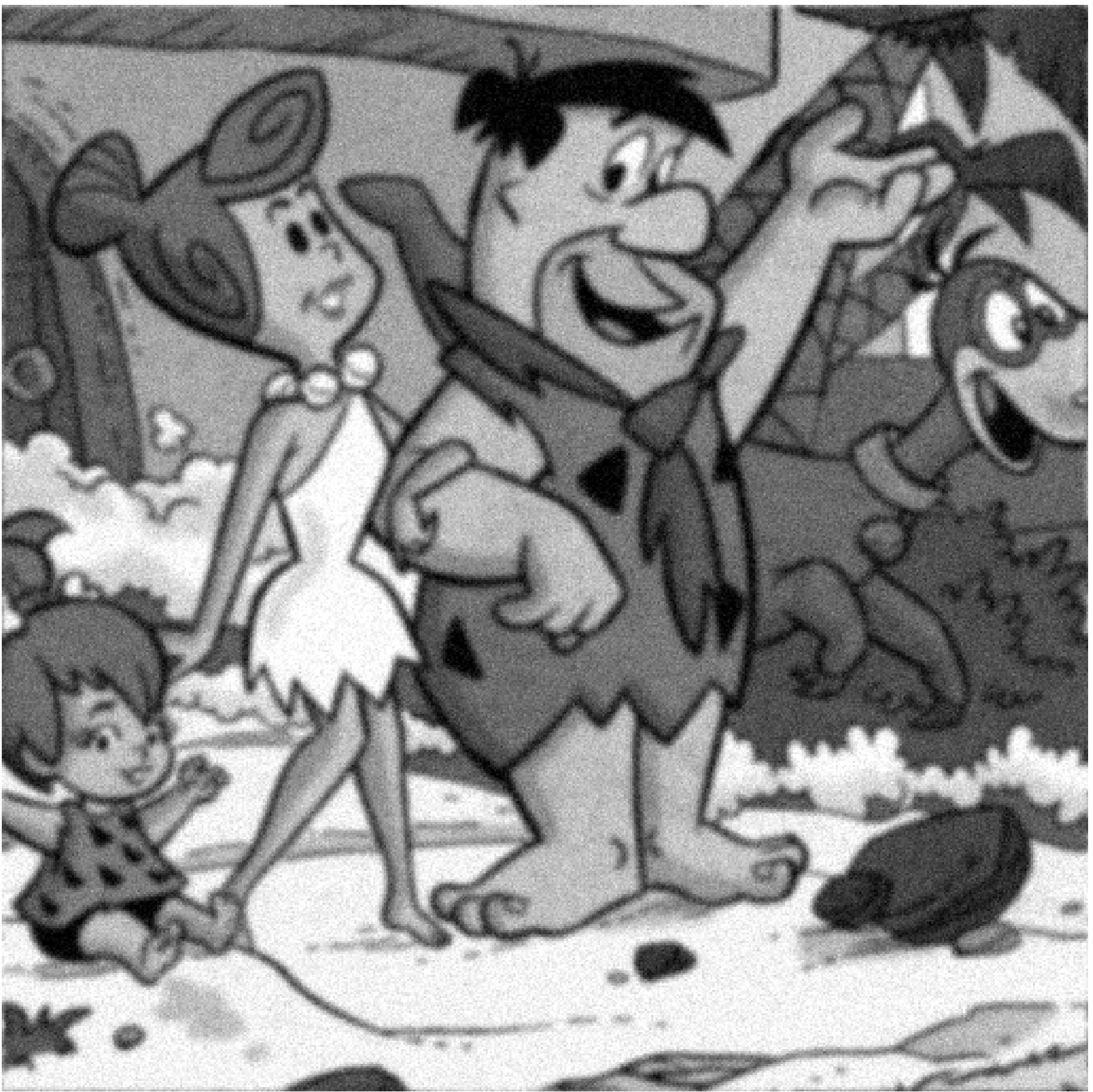}}
  \hfill
  \subfigure[]{\includegraphics[width=0.32\linewidth,viewport=200 320 392 512,clip]{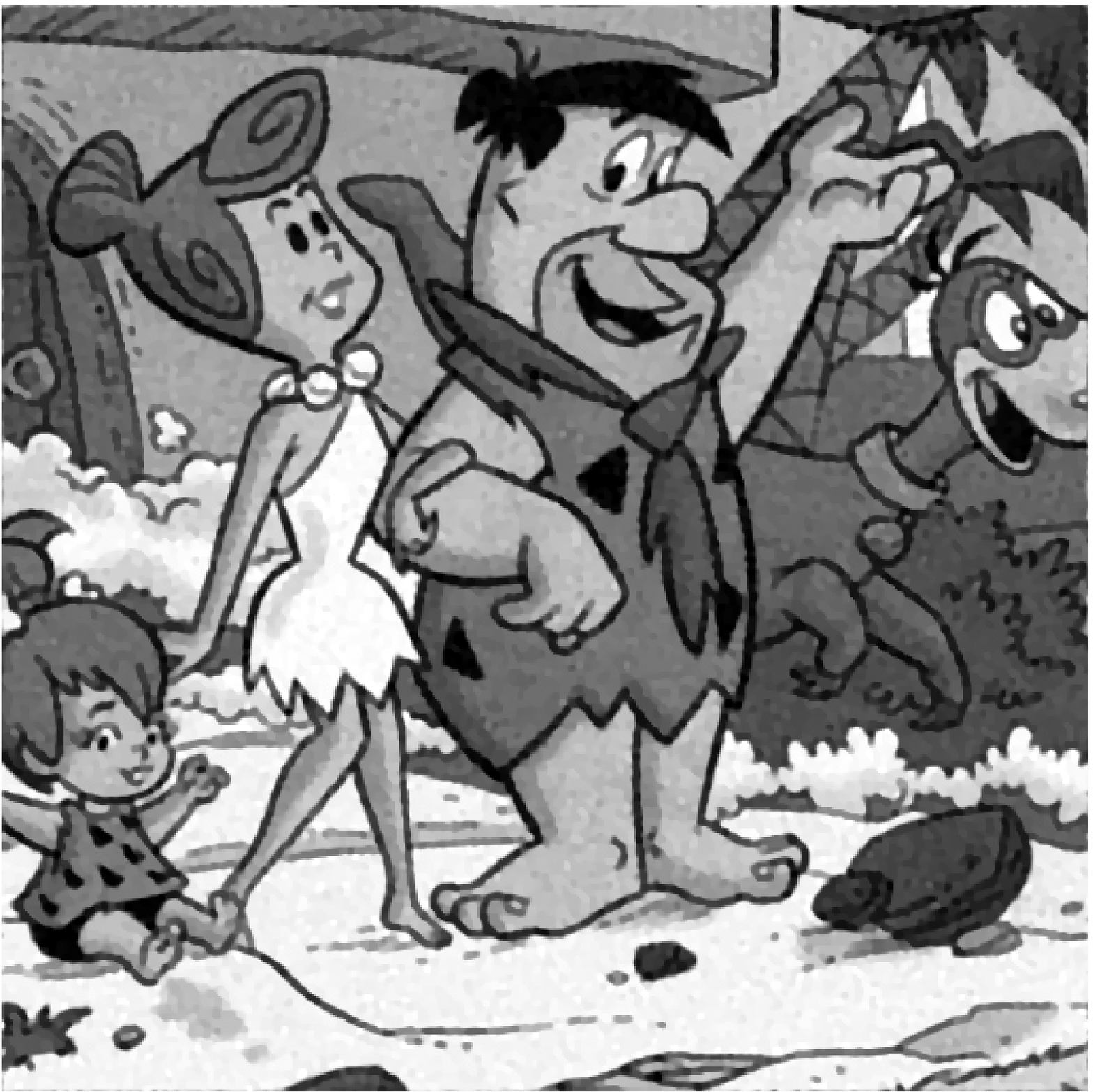}}\\
  \caption{(a-b) Risk, projection risk and its $\GSUREn$
    estimates$^{\ref{footnote:shift}}$ as a function
    of the regularization parameter $\lambda$.
    (a) The $12$ points where $\GSUREMC{x}(y, \lambda, \delta)$ has been
    evaluated by exhaustive search. (b) The $4$ evaluation
    points of $\GSUREAMC{x}(y, \lambda, \delta, \epsilon)$ and $\GSUGARAMC{x}(y, \lambda, \delta, \epsilon)$
    required by BFGS to reach the optimal one.
    (c-d) Respectively, a close in of the underlying image $x_0$, the
    observation $y$ and
    the solution $x(y, \lambda)$ at the optimal $\lambda$.
  }
  \label{fig:total_variation_sure}
\end{figure}

\paragraph{Application to image deblurring}
We illustrate the total-variation regularization on an image deblurring problem.
We therefore consider the forward model $Y = \Phi x_0 + W \in \RR^P$, $W \sim \Nn(0, \sigma^2 \Id_P)$,
where $x_0$ is a piece-wise constant (or approximately so) image
and $\Phi$ is a discrete convolution matrix.

We have taken a cartoon-like image of size $(n_1,n_2)=(512,512)$ and
$P = 512^2$ observations corresponding to noisy observations of
a convolution product with a discrete Gaussian kernel
of radius $2$ pixels. To ensure numerical stability of the pseudo-inverse (typically for the least-square estimate and the computation of the projection risk and its estimate), the kernel has been truncated in the Fourier domain such that
too small contributions have been set to $0$. The consequence is
that around $80$\% of (high) frequencies are masked.
The standard deviation of the noise has been set to $\sigma = 10$
(for an image $x_0$ with a range $[0,255]$)
such that
the resulting minimum least-square estimate
$x_{\mathrm{LS}}(y) = \Phi^+ y$ has a peak signal-to-noise ratio (PSNR)
equals to
$10 \log_{10}(255^2/\norm{x_{\mathrm{LS}}(y)-x_0}_F^2)=21.02$ dB.

Figure \ref{fig:total_variation_sure}.(a) and (b) display the risk,
the projection risk and the $\GSUREn=\ASUREn$ (with $A = \Pi$)
estimates as a function of $\la$
obtained from a single realization of $y$ and $\de$.
In (a),
$\GSUREMC{x}(y, \lambda, \delta)$ has been
evaluated for $12$ values of $\la$ chosen in a suitable tested range
using the algorithm given in Figure~\ref{fig:pseudo-code-gsure}.
Figure (b), shows the benefit of computing $\GSUREAMC{x}(y, \lambda, \delta, \epsilon)$
and $\GSUGARAMC{x}(y, \lambda, \delta, \epsilon)$, as described in Figure~\ref{fig:pseudo-code-sugar}, to realize a quasi-Newton optimization.
The sequence of iterates $\lambda_{n}$
is represented as well as the sequence of the slopes
of $\GSUREAMC{x}(y, \lambda_{n}, \delta, \epsilon)$ given by $\GSUGARAMC{x}(y, \lambda_n, \delta, \epsilon)$.
The BFGS algorithm reaches
to the optimal value in $4$ iterations.
The deviation of $\GSUREAMC{x}(y, \lambda, \delta, \epsilon)$ from the projection risk is of the same
order as the deviation of $\GSUREMC{x}(y, \lambda, \delta)$.
At the optimum value $\lambda^\star$ minimizing the $\GSUREn$,
the true risk is not too far from its minimum showing that relatively to
the range of variation of the risk,
in this case, the projection risk is indeed a good objective in order to
minimize the risk.
In Figure~\ref{fig:total_variation_sure}.(c-e)
the solution $x(y, \lambda^\star)$
is compared to $x_0$ and $y$.
The solution $x(y, \lambda^\star)$ has a PSNR
of $24.98$ dB (i.e., a gain of about $+3.94$ dB).
Remark that given such noise level and convolution operator,
masking $80$\% of (high) frequencies,
the solution selected by minimizing the $\GSUREn$ criterion
is still a bit blurred and exhibit
a staircasing effect. Using a larger regularization parameter $\lambda$
would result in a more ``cartoon'' result with less blur.
This would however entail a larger bias, which corresponds to a loss of contrast inherent to the convexity of the TV prior.
This larger bias subsequently degrades the mean square error, which explains why it is not selected by the $\GSUREn$ criterion.

\subsection{Weighted $\ell_1$-analysis Wavelet Regularization}


We focus on the recovery of a piece-wise regular image
$x_0 \in \RR^{n_1 \times n_2}$
from an observation $y$ of $Y = \Phi x_0 + W \in \RR^P$, $W \sim \Nn(0, \sigma^2 \Id_P)$,
using a $J$-scale undecimated wavelet analysis regularization of the form
\begin{align}\label{eq-optim-wavelet}
  x^\star(y, \la) \in \uArgmin{x} \frac{1}{2}\norm{\Phi x - y}^2 + \norm{ \Psi x }_{1, \la}
  \qwhereq
  \Psi = \left(
  \begin{array}{c}
    \Psi^h_{1}\\
    \Psi^v_{1}\\
    \vdots\\
    \Psi^h_{J}\\
    \Psi^v_{J}
  \end{array}
  \right)
\end{align}
and $\la \in {\RR^+}^J$ and
$\Psi \in \RR^{2JN \times N}$ is the analysis operator of a two-orientation wavelet transform, where, for all scales $1 \leq j \leq J$, $\Psi^h_{j}$, $\Psi^v_{j}$
are defined such that, for $x \in \RR^{N}$, $u^h_{j} = \Psi^h_{j} x$
and $u^v_{j} = \Psi^v_{j} x$
are respectively the vectors of undecimated
wavelet coefficients
of $x$ in
the horizontal and vertical directions at the decomposition level $j$.
The weighted $\ell^1$-norm $\norm{ . }_{1, \la}$ is
\begin{align*}
  \norm{ \Psi x }_{1, \la} = \sum_{j = 1}^J \la_j \left( \big\| \Psi^h_{j} x \big\|_1 + \big\| \Psi^v_{j} x \big\|_1 \right)~.
\end{align*}
Multi-scale wavelet analysis promotes piece-wise regular images by enforcing smoothness while preserving
sharp discontinuities at different scales and orientations.
Each parameter $\lambda_j$ controls the regularity at scale $j$.
A large value of $\lambda_j$ tends to over-smooth structures at
scale $j$, while a small value leads to under-smoothing.
As noted in several papers, see e.g.~\cite{pesquet1997new,chaux2008nonlinear},
the optimal values $\lambda_j$ are also
image and degradation dependent revealing again
the importance of automatic selection procedures.

\begin{table}[!t]
  \caption{Illustration of the minimization of
    $\SUREAMCn$ in multi-scale regularization obtained
    for the three images of Figure~\ref{fig:multi_scale}
    with different numbers of scales from $J=1$ to $J=3$ using either
    one global parameter or one parameter per scales.
    For each case, the obtained optimal parameters $\lambda^\star$ are
    given. The associated value of $\SUREn$ and the PSNR are compared
    to neighbors of $\lambda^\star$ located at $0.75\lambda^\star$ and
    $1.25\lambda^\star$.}
  \label{tab:multi_scale}
  \centering
  {\small
  \begin{tabular}{|lc|ccl|ccc|}
    \hline
    \multicolumn{2}{|c|}{Input} & \multicolumn{3}{c|}{Optimal parameters} & \multicolumn{3}{c|}{SURE/PSNR}\\
    \hline
    Image & PSNR & $J$ & $\dim \Lambda$ & $\lambda^\star$ & $0.75\lambda^\star$ & $\lambda^\star$ & $1.25\lambda^\star$\\
    \hline
    \hline
    \texttt{Mandrill} & 17.37 &
    1 & 1 & (7.58) & 7.53/24.84 & {\bf 7.39}/24.90 & 7.43/{\bf 24.94}\\
    & &
    2 & 1 & (5.63) & 7.60/24.85 & {\bf 7.45}/24.88 & 7.58/{\bf 24.89}\\
    & &
    3 & 1 & (4.54) & 7.87/24.04 & {\bf 7.71}/{\bf 24.10} & 7.83/{\bf 24.10}\\
    \cline{3-8}
    & &
    2 & 2 & (5.94, 4.24) & 7.49/25.02 & {\bf 7.30}/25.06 & 7.38/{\bf 25.07}\\
    & &
    3 & 3 & (7.51, 1.07, 0.99) & 7.37/25.12 & {\bf 7.22}/25.18 & 7.33/{\bf 25.20}\\
    \hline
    \hline
    \texttt{House} & 17.65 &
    1 & 1 & (18.38) & 3.69/{\bf 31.16} & {\bf 3.51}/31.15 & 3.68/30.55\\
    & &
    2 & 1 & (11.11) & 3.72/31.31 & {\bf 3.51}/{\bf 31.40} & 3.81/31.05\\
    & &
    3 & 1 & (8.73) & 4.30/30.18 & {\bf 4.08}/{\bf 30.31} & 4.43/30.13\\
    \cline{3-8}
    & &
    2 & 2 & (14.47, 5.20) & 3.53/31.51 & {\bf 3.34}/{\bf 31.57} & 3.55/31.05\\
    & &
    3 & 3 & (15.00, 2.50, 2.83) & 3.52/31.55 & {\bf 3.27}/{\bf 31.63} & 3.44/31.14\\
    \hline
    \hline
    \texttt{Cameraman} & 15.13 &
    1 & 1 & (13.50) & 5.29/28.61 & {\bf 5.09}/{\bf 28.73} & 5.35/28.64\\
    & &
    2 & 1 & (8.78) & 5.34/28.75 & {\bf 5.09}/{\bf 28.83} & 5.38/28.72\\
    & &
    3 & 1 & (7.14) & 5.84/28.03 & {\bf 5.60}/{\bf 28.06} & 5.88/27.99\\
    \cline{3-8}
    & &
    2 & 2 & (10.98, 3.74) & 5.16/28.91 & {\bf 4.90}/{\bf 29.04} & 5.09/28.96\\
    & &
    3 & 3 & (11.56, 3.31, 0.97) & 5.07/29.00 & {\bf 4.86}/{\bf 29.11} & 5.13/28.99\\
    \hline
  \end{tabular}
  }
\end{table}

Problem \eqref{eq-optim-wavelet} is a special instance of
\eqref{eq-min-cp} where the parameter $\la = \theta \in \Theta = {\RR^+}^J$, and
\begin{align*}
  \begin{array}{rlll}
  &H(x,y,\la) &=& \frac{1}{2} \norm{\Phi x - y}^2,\\
  &G(u,y,\la) &=& \norm{u}_{1,\la}\\
  \qandq &
  K(x) &=& \Psi x~.
  \end{array}
\end{align*}
Hence the primal-dual CP splitting can be used to solve \eqref{eq-optim-wavelet}
using 
\begin{align*}
  \begin{array}{rlll}
  & \Prox_{\xi H}(x, y, \la) &=& x + \xi \Phi^* y - \xi \Phi^* (\Id + \xi \Phi \Phi^*)^{-1} \Phi(x + \xi \Phi^* y),\\
  & \Prox_{\tau G^*}(u, y, \la) &=& u - \tau \ST(u/\tau, \lambda/\tau)\\
  \qandq &
  K^*(u, \la) &=& \sum_{j = 1}^J ({\Psi^h_j}^* u^h_j + {\Psi^v_j}^* u^v_j)
  \end{array}
\end{align*}
where $\ST$ denotes here the multi-scale extension of the soft-thresholding operator \eqref{eq:st}, such that, for $t \in \RR^{2JN}$ and $\rho \in \RR^J$, we have
\begin{align*}
    \ST(t, \rho)^o_j
    =
    \ST(t^o_j, \rho_j)
\end{align*}
for all scales $1 \leq j \leq J$ and orientations $o = { v, h }$.
Corollary \ref{cor:gfb_diff_y} and \ref{cor:gfb_diff_theta} can then be applied using
\begin{align*}
  \begin{array}{rlll}
  & \JAC{1}{\Prox_{\xi H}}( x,y,\la )[\de_x] &=&
  \de_x + \xi \Phi^* (\Id + \xi \Phi \Phi^*)^{-1} \Phi \de_x,\\
  & \JAC{2}{\Prox_{\xi H}}( x,y,\la )[\de_y] &=&
  \xi \Phi^* \de_y - \xi^2 \Phi^* (\Id + \xi \Phi \Phi^*)^{-1} \Phi \Phi^* \de_y,\\
  & \JAC{3}{\Prox_{\xi H}}( x,y,\la ) &=& 0\\\\[-0.3cm]
  \qandq
  & \JAC{1}{\Prox_{\tau G^*}}( u,y,\la )[\de_u] &=&
  \de_u - \JACs{1}{\ST}(u/\tau, \la/\tau)[\de_u], \\
  & \JAC{2}{\Prox_{\tau G^*}}( u,y,\la )[\de_y] &=&
  0,\\
  & \JAC{3}{\Prox_{\tau G^*}}( u,y,\la ) &=&
  -\JACs{2}{\ST}(u/\tau, \la/\tau) .
  \end{array}
\end{align*}
where the derivatives of
the multi-scale soft-thresholding are defined, for any
$t \in \RR^{2JN}$, $\rho \in \RR^J$
and $\de_t \in \RR^{2JN}$, by
\begin{align}
  \JACs{1}{\ST}(t, \rho)[\de_t]^o_j =
  \JACs{1}{\ST}(t^o_j, \rho_j)[{\de_t}^o_j]
  \qandq
  \JACs{2}{\ST}(t, \rho)^o_j =
  \JACs{2}{\ST}(t^o_j, \rho_j)
\end{align}
for all scales $1 \leq j \leq J$ and orientations $o = { v, h }$.\\

\begin{figure}[!t]
  \subfigure{\includegraphics[width=0.19\linewidth,viewport=50 128 178 256,clip]{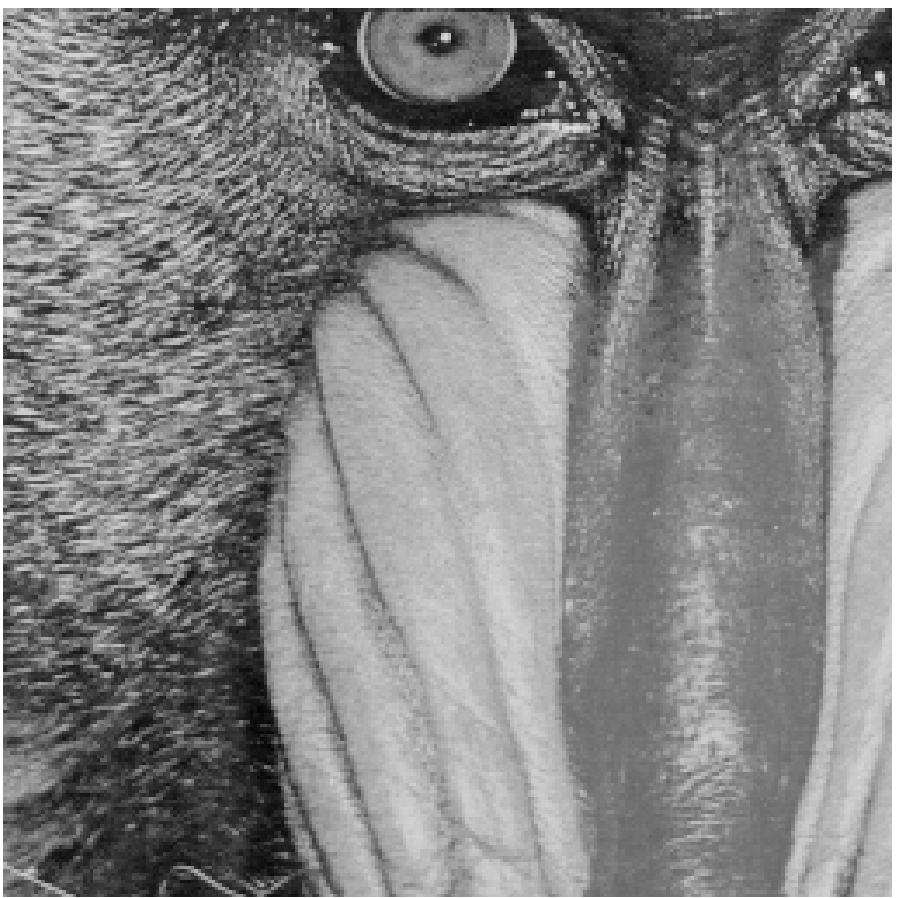}}
  \subfigure{\includegraphics[width=0.19\linewidth,viewport=50 128 178 256,clip]{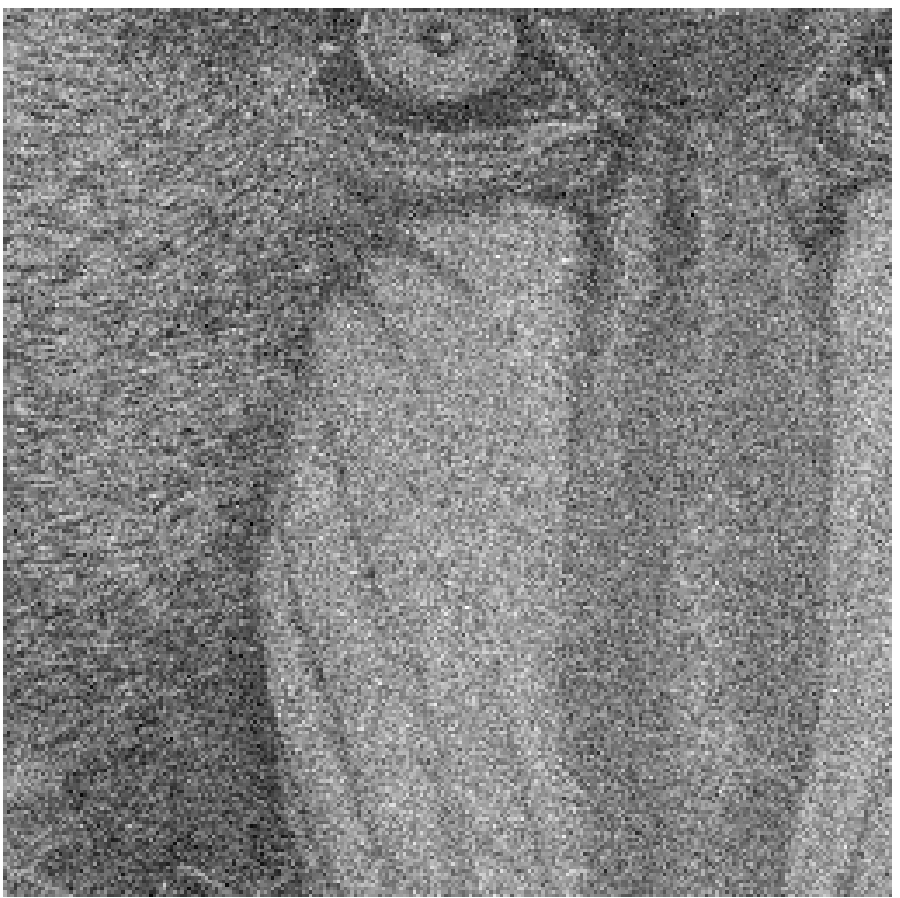}}
  \subfigure{\includegraphics[width=0.19\linewidth,viewport=50 128 178 256,clip]{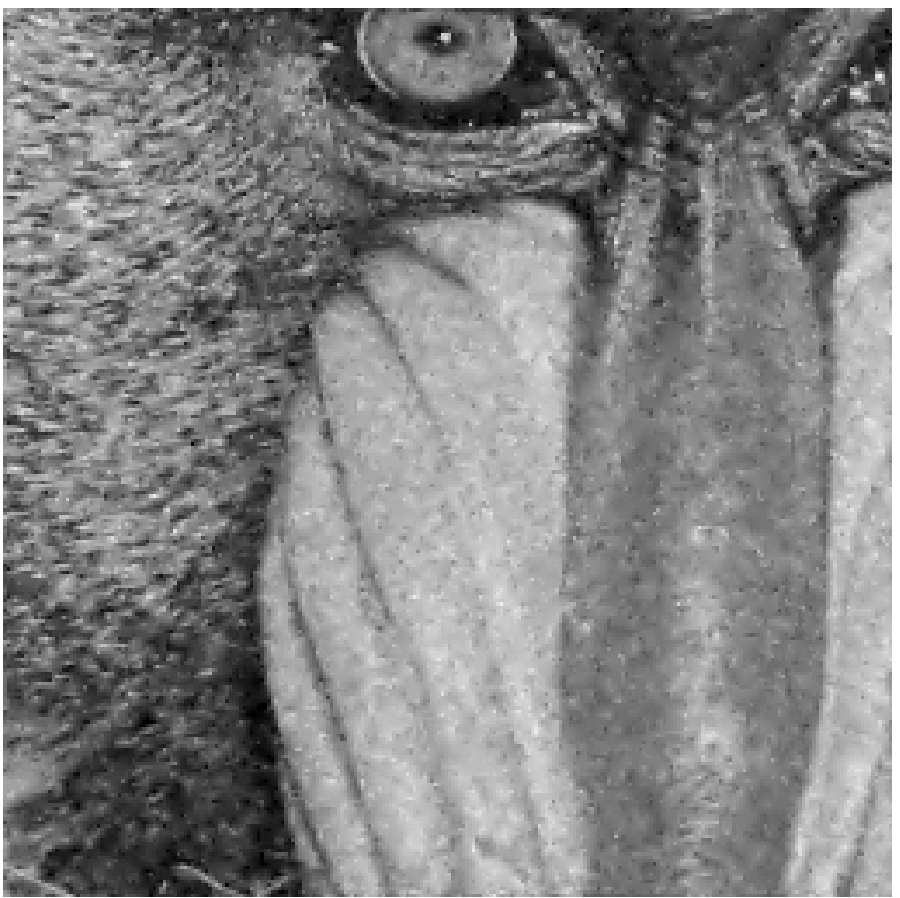}}
  \subfigure{\includegraphics[width=0.19\linewidth,viewport=50 128 178 256,clip]{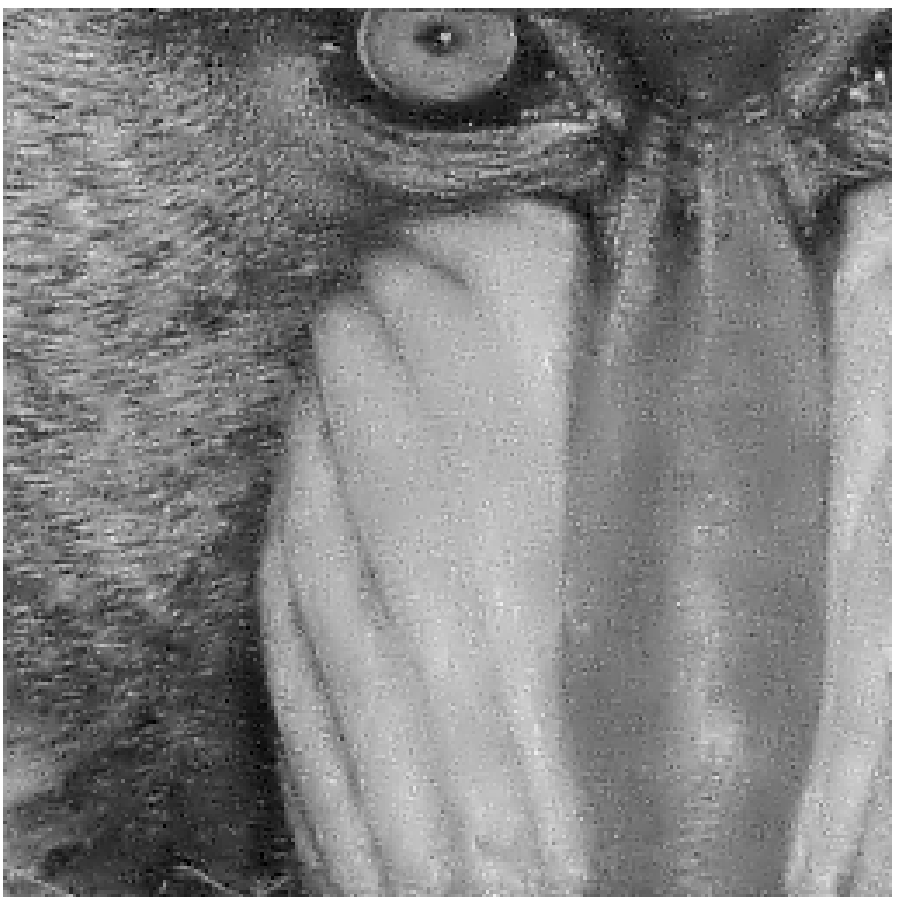}}
  \subfigure{\includegraphics[width=0.19\linewidth,viewport=50 128 178 256,clip]{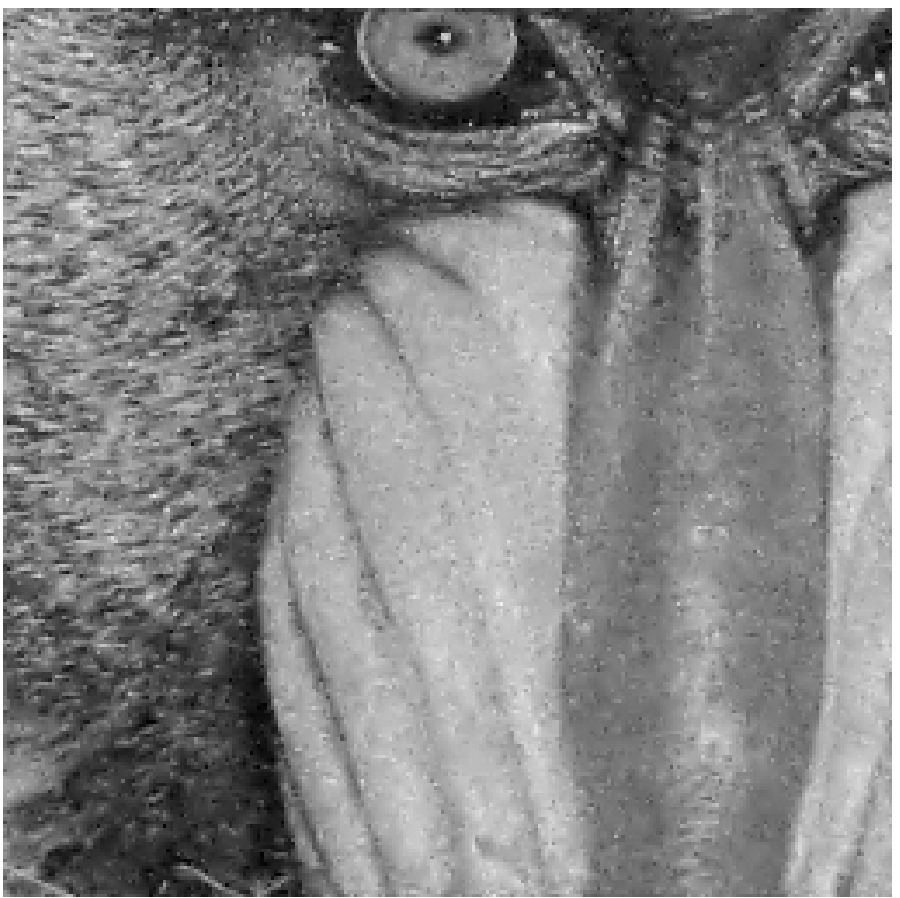}}\\[-0.1cm]
  \setcounter{subfigure}{0}%
  \subfigure{\includegraphics[width=0.19\linewidth,viewport=90 78 218 206,clip]{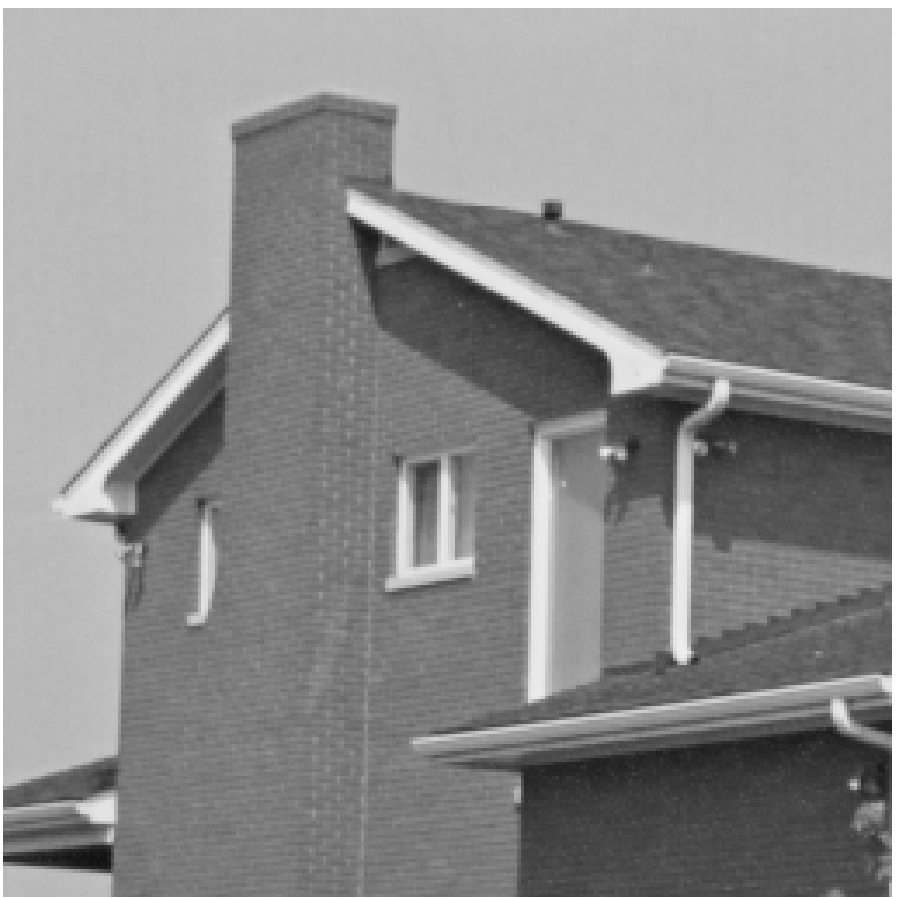}}
  \subfigure{\includegraphics[width=0.19\linewidth,viewport=90 78 218 206,clip]{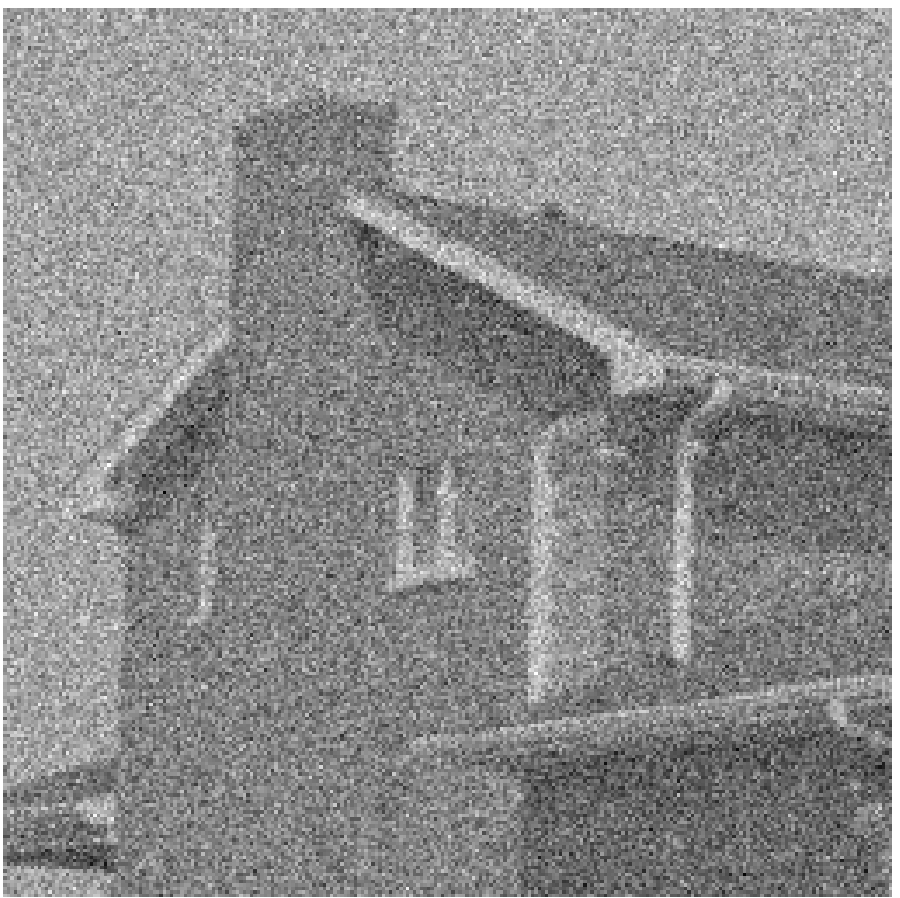}}
  \subfigure{\includegraphics[width=0.19\linewidth,viewport=90 78 218 206,clip]{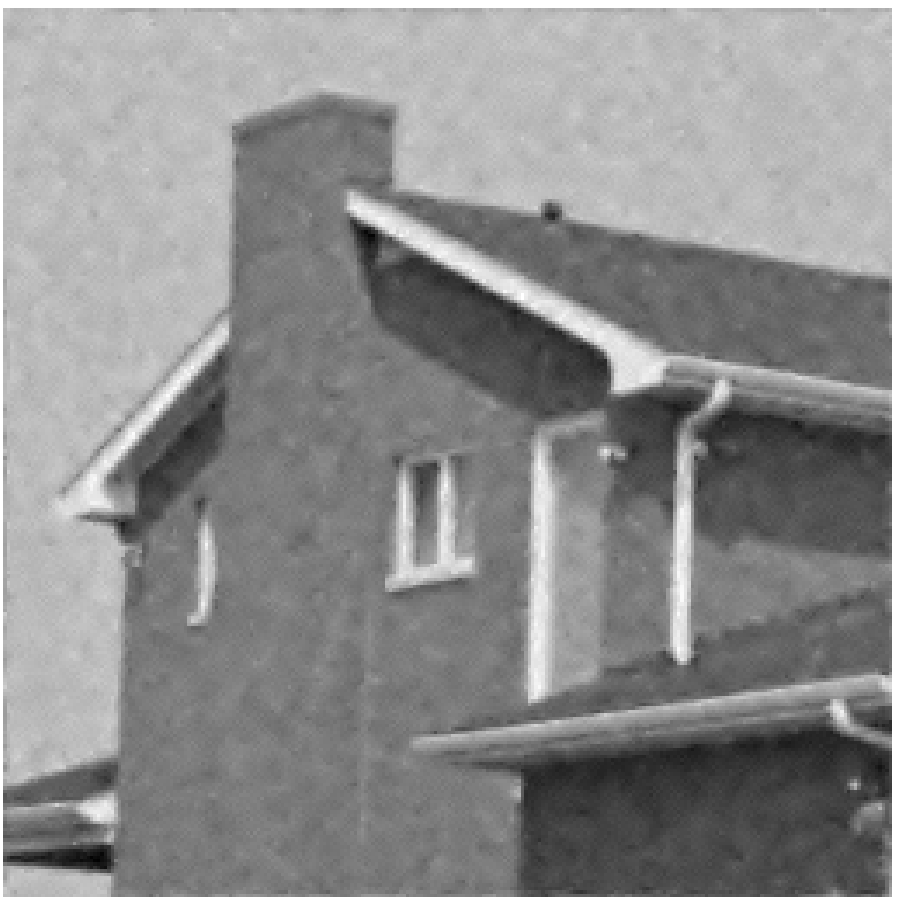}}
  \subfigure{\includegraphics[width=0.19\linewidth,viewport=90 78 218 206,clip]{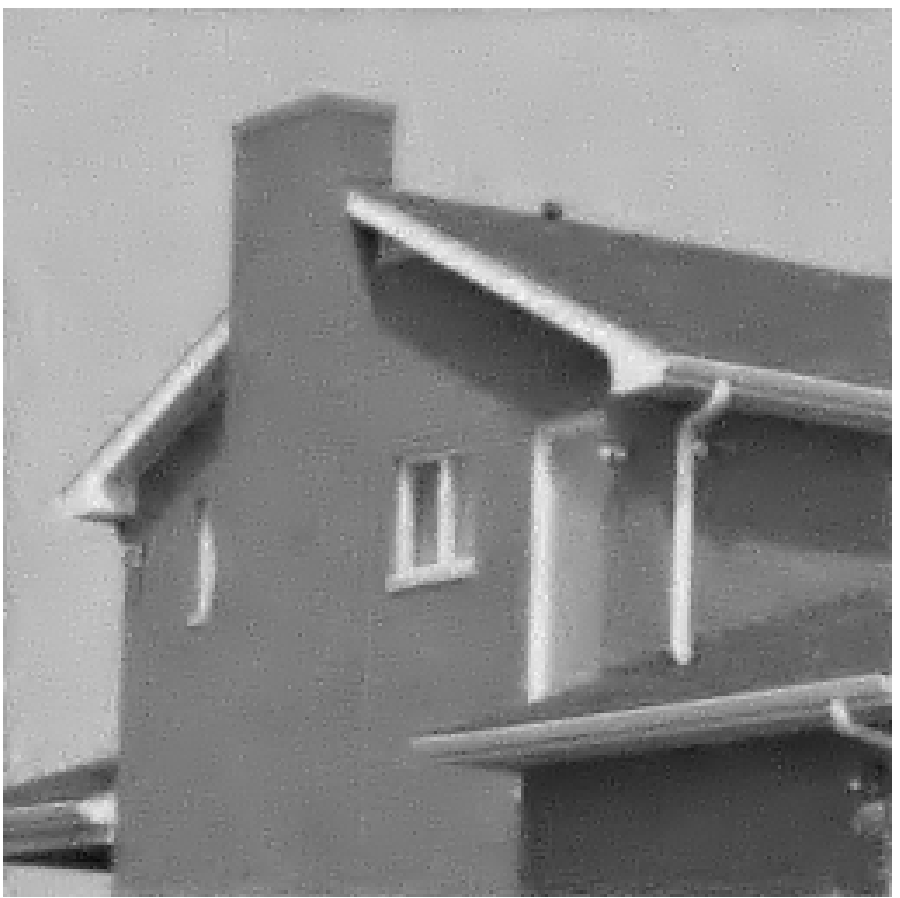}}
  \subfigure{\includegraphics[width=0.19\linewidth,viewport=90 78 218 206,clip]{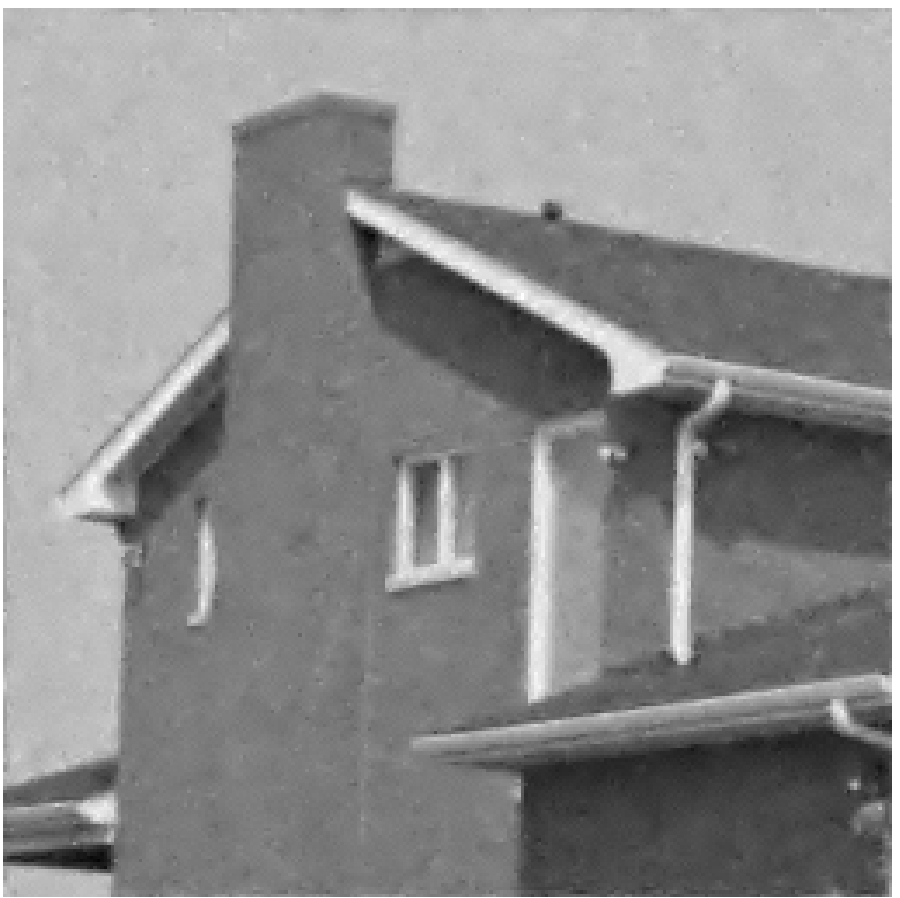}}\\[-0.1cm]
  \setcounter{subfigure}{0}%
  \subfigure[]{\includegraphics[width=0.19\linewidth,viewport=80 88 208 216,clip]{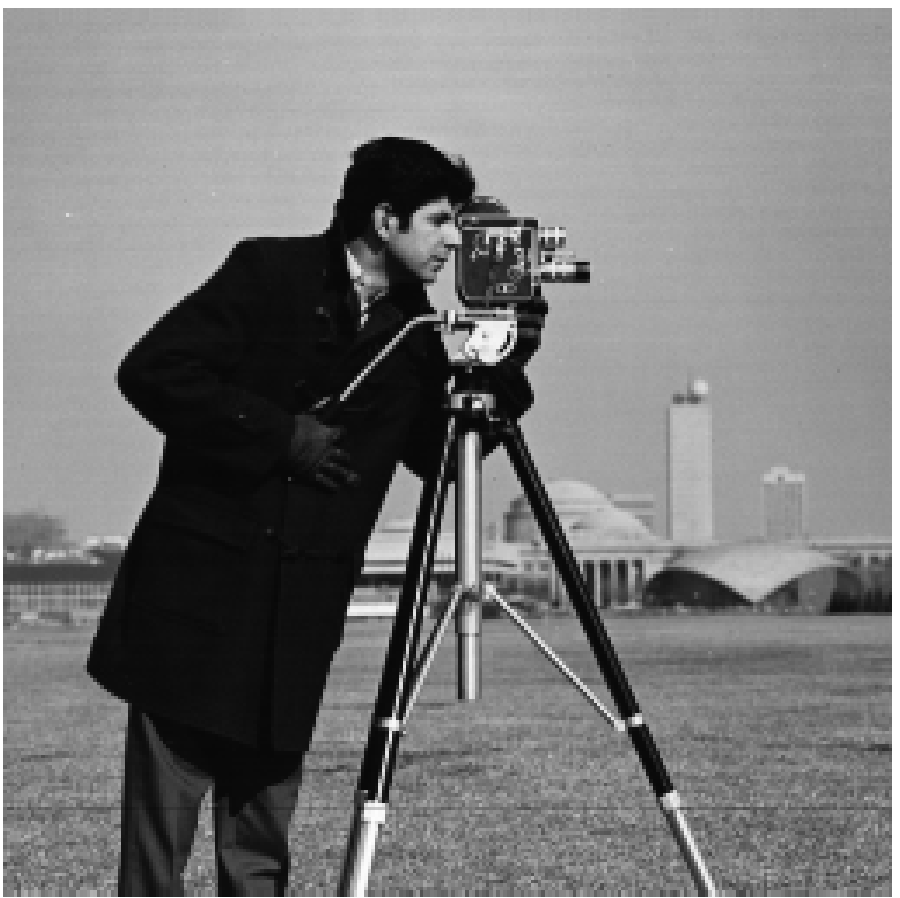}}
  \subfigure[]{\includegraphics[width=0.19\linewidth,viewport=80 88 208 216,clip]{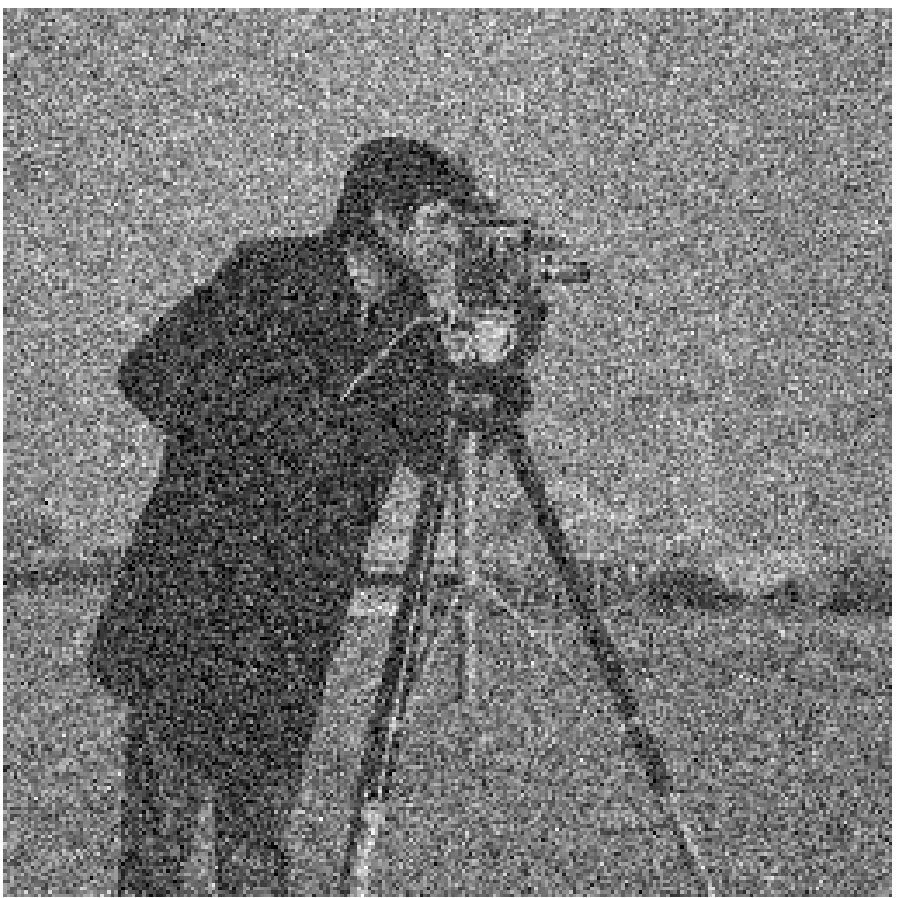}}
  \subfigure[]{\includegraphics[width=0.19\linewidth,viewport=80 88 208 216,clip]{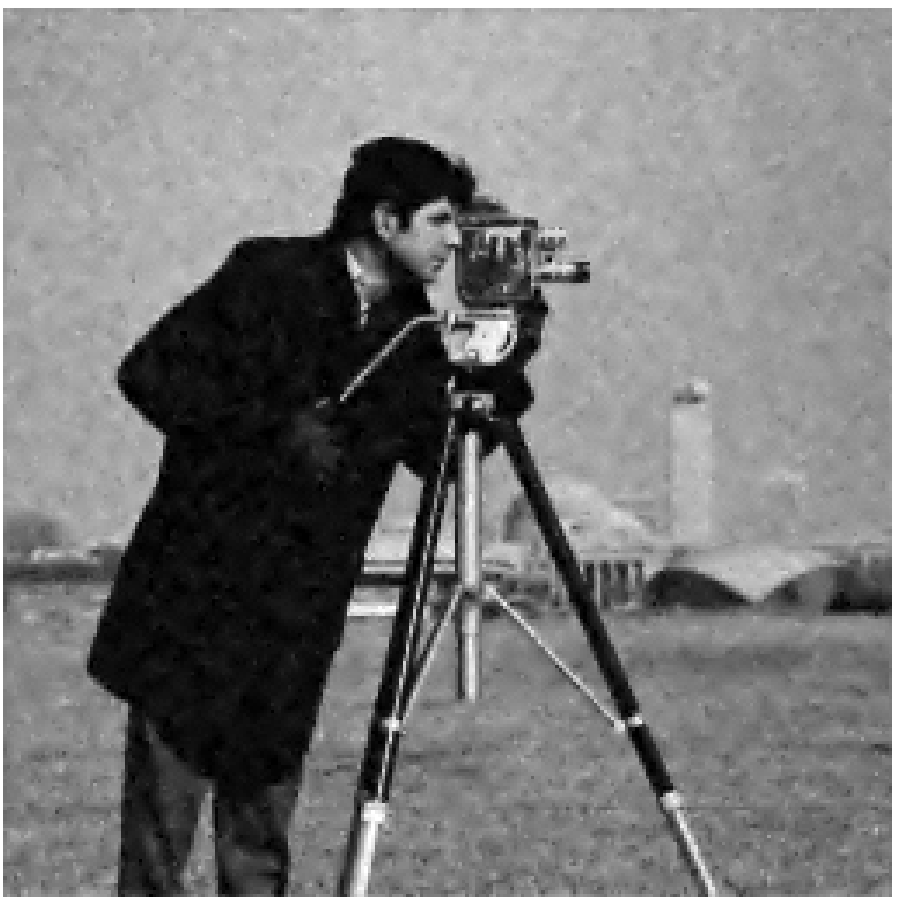}}
  \subfigure[]{\includegraphics[width=0.19\linewidth,viewport=80 88 208 216,clip]{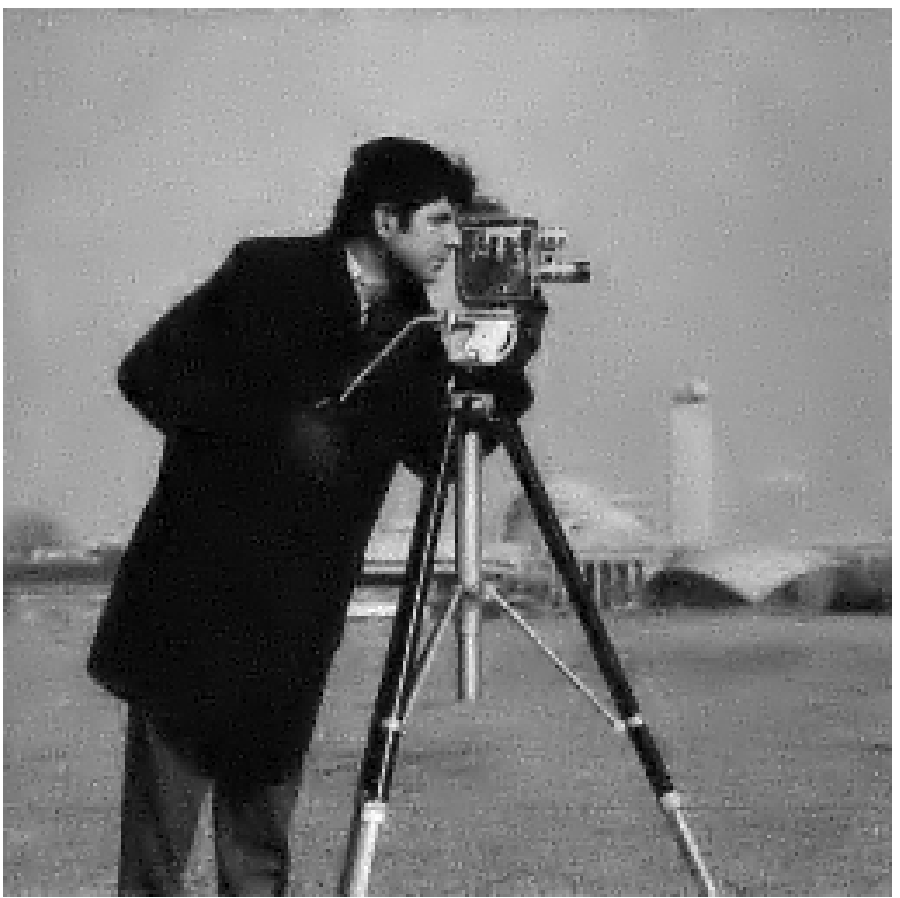}}
  \subfigure[]{\includegraphics[width=0.19\linewidth,viewport=80 88 208 216,clip]{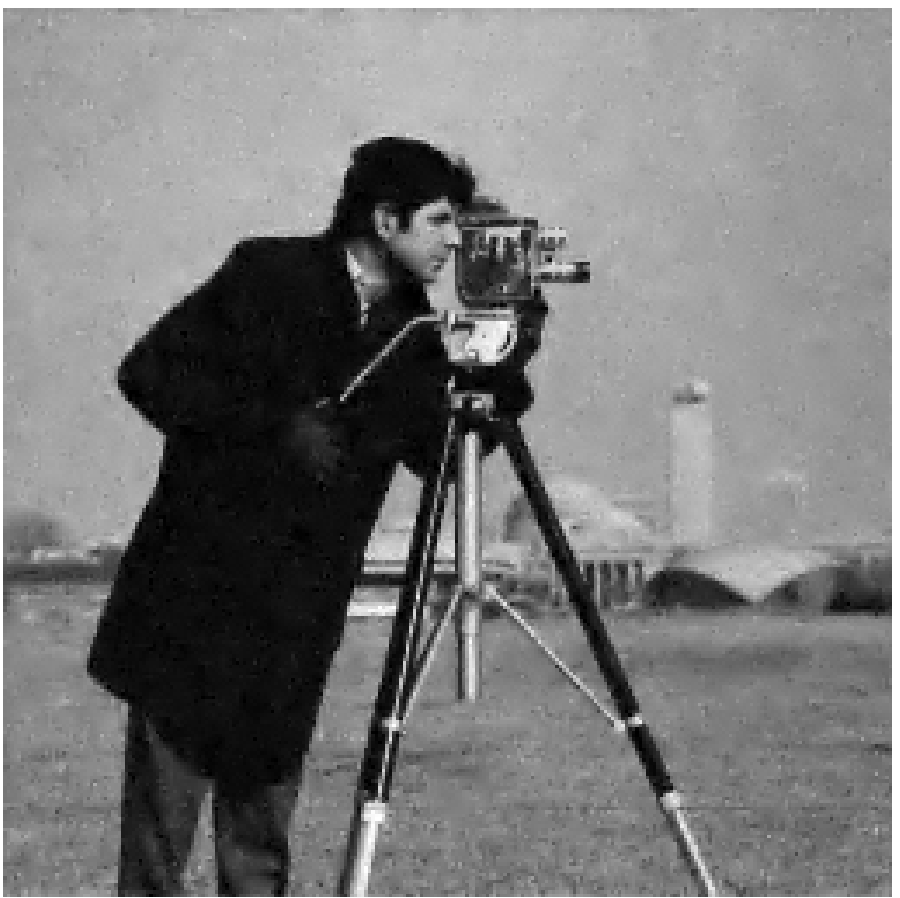}}
  \caption{From top to bottom, a close up on
    \texttt{Mandrill}, \texttt{House} and
    \texttt{Cameraman}. (a) Underlying image $x_0$. (b) Least-square estimate $x_{\mathrm{LS}}(y)$.
    (c) Result with $\lambda^\star$ for one level of decomposition $J=1$,
    (d) for three levels of decomposition $J=3$ using one global parameter, and
    (e) for three levels of decomposition $J=3$ using one parameter per scales.}
  \label{fig:multi_scale}
\end{figure}

\paragraph{Application to compressed sensing}

We illustrate the multi-scale wavelet $\ell_1$-analysis regularization
on a compressed sensing problem.
We therefore consider the forward model
$Y = \Phi x_0 + W \in \RR^P$, $W \sim \Nn(0, \sigma^2 \Id_P)$,
where $x_0$ is a piece-wise multi-scale regular (or approximately so) image
and $\Phi$ is a random matrix.
Here the multi-scale transform $W$ is constructed
from undecimated Daubechies 4 wavelets \cite{daubechies1988orthonormal}.

We have taken a uniformly randomized sub-sampling of a uniform random convolution,
where ($P/N = 0.5$).
The standard deviation has been set to $\sigma = 10$
(for an image $x_0$ with a range $[0,255]$)
such that
the resulting minimum least-square estimate
$x_{\mathrm{LS}}(y) = \Phi^+ y$ has a PSNR given by
$10 \log_{10}(255^2/\norm{x_{\mathrm{LS}}(y)-x_0}_F^2) \approx 16$ dB.

Table \ref{tab:multi_scale} and Figure \ref{fig:multi_scale}
illustrates the multi-scale regularization obtained by minimizing
the $\SUREn=\ASUREn$ (with $A = \Id$) for
three different images $x_0$, known as \texttt{Mandrill},
\texttt{House} and \texttt{Cameraman}, 
and a single realization of $y$ and $\de$. Three
levels of decomposition from $J = 1$ to $3$ are considered.
We consider also to use either one global regularization parameter
or one parameter per scales.
Table \ref{tab:multi_scale} gives the selected optimal vector
of parameters $\lambda^\star$ for each level of decomposition
and their associated performance in terms of $\SUREn$ and PSNR.
We first observe that compared to the global approach, optimizing
one parameter per scale indeed adapts better to the regularity of the image.
For instance, the image
\texttt{Mandrill} contains fine scales with more energy
than \texttt{House}, and then the obtained penalization of the first scale
is smaller for \texttt{Mandrill} than for \texttt{House}.
Visual inspection of these results on Figure~\ref{fig:multi_scale}
illustrates this automatic adaptation.
In the same vein,
with three levels of decomposition, the penalization is less severe
for \texttt{Mandrill} than for \texttt{House}
and \texttt{Cameraman}.
We next observe that increasing the level of decomposition
improves the PSNR when using one parameter per scale, while
this is not the case when a global parameter is used. The gap is more important between
$J=1$ and $J=2$. To assess the minimization of $\SUREAMCn$,
we have compared the $\SUREn$ and the PSNR values at
$0.75\lambda^\star$ and $1.25\lambda^\star$. At the optimal
$\lambda^\star$, the SURE is as expected minimal. Furthermore,
at $\lambda^\star$, the PSNR is either maximal
or not too far from its maximal value, showing that, in this case,
the prediction risk is indeed a good objective in order to
maximize the PSNR.

%% file: sections/proofs/proofs_risk_optimization.tex
\section{Proofs of Section~\ref{sec:risk_optimization}}\label{sec:proofs_risk_optimization}

\begin{proof}[Proposition~\ref{prop:sugar}]
  This is a consequence of the chain rule and linearity of the weak derivative.
  Indeed, $\AEDOFA{\mu}(y, \theta, \epsilon)$ is just the sum of $P$ weakly differentiable functions, and hence is weakly differentiable with the weak derivative with respect to $\theta$ as given.
  Moreover, $\norm{A(\mu(y, \theta) - y)}^2 = \sum_{i=1}^P \pa{\pa{A(\mu(y, \theta) - y)}_i}^2$. Each term $i$ is the composition of a weakly differentiable function $\pa{A(\mu(y, \cdot) - y)}_i$ and $(\cdot)^2$, where the latter is obviously continuously differentiable with bounded derivative, and takes $0$ at the origin. It then follows from the chain rule~\cite[Theorem~4(ii), Section~4.2.2]{EvansGariepy92} that $\pa{A(\mu(y, \cdot) - y)}_i$ is weakly differentiable, and the weak derivative of $\norm{A(\mu(y, \cdot) - y)}^2$ with respect to $\theta$ is indeed
  \[
  2 \JACs{2}{\mu}(y, \theta)^* A^*\!A (\mu(y, \theta) - y) ~.
  \]
\end{proof}

\begin{proof}[Theorem~\ref{thm:sugar}]
The proof strategy consists in commuting in an appropriate order the different signs (limit, integration and derivation) while checking that our assumptions provide sufficient conditions for this to hold.

Let $V$ a compact subset of $\Theta$, and choose $\varphi \in C_c^1(\Theta)$ with support in $V$. We have
\begin{eqnarray*}
\int_{\Theta} \ARISK{\mu}(\mu_0, \theta) \frac{\partial \varphi(\theta)}{\partial \theta_i} d\theta
	&=& \int_{V} \ARISK{\mu}(\mu_0, \theta) \frac{\partial \varphi(\theta)}{\partial \theta_i} d\theta \\	&=& \int_{V} \EE_W \left[\norm{A (\mu(Y, \theta) - y)}^2\right] \frac{\partial \varphi(\theta)}{\partial \theta_i} d\theta \\
\text{\scriptsize{[Stein Lemma]}}
	&\equal{(S.1)}& \int_{V} \EE_W \left[\ASURE{\mu}(Y, \theta)\right] \frac{\partial \varphi(\theta)}{\partial \theta_i} d\theta \\
\text{\scriptsize{[\eqref{eq:limsurefd}]}}
	&\equal{(S.2)}& \int_{V} \EE_W\left[\lim_{\epsilon \to 0} \ASUREA{\mu}(Y, \theta, \epsilon)\right] \frac{\partial \varphi(\theta)}{\partial \theta_i} d\theta \\
\text{\scriptsize{[Dominated convergence]}}
	&\equal{(S.3)}& \lim_{\epsilon \to 0}\int_{V}\EE_W\left[\ASUREA{\mu}(Y, \theta, \epsilon)\right] \frac{\partial \varphi(\theta)}{\partial \theta_i} d\theta \\
\text{\scriptsize{[Fubini]}}
	&\equal{(S.4)}& \lim_{\epsilon \to 0}\EE_W\left[\int_{V}\ASUREA{\mu}(Y, \theta, \epsilon) \frac{\partial \varphi(\theta)}{\partial \theta_i} d\theta \right] \\
\text{\scriptsize{[Weak differentiability, Proposition~\ref{prop:sugar}]}}
	&\equal{(S.5)}& - \lim_{\epsilon \to 0}\EE_W\left[\int_{V}\frac{\partial}{\partial \theta_i}\ASUREA{\mu}(Y, \theta, \epsilon) \varphi(\theta) d\theta\right] \\
\text{\scriptsize{[Proposition~\ref{prop:sugar}]}}
	&\equal{(S.6)}& - \lim_{\epsilon \to 0}\EE_W\left[\int_{V}\pa{\ASUGARA{\mu}(Y ,\theta, \epsilon)}_i \varphi(\theta) d\theta\right]  \\
\text{\scriptsize{[Fubini]}}
	&\equal{(S.7)}& - \lim_{\epsilon \to 0}\int_{V}\EE_W\left[\pa{\ASUGARA{\mu}(Y ,\theta, \epsilon)}_i\right] \varphi(\theta) d\theta  \\
\text{\scriptsize{[Dominated convergence]}}
	&\equal{(S.8)}& - \int_{V}\pa{\lim_{\epsilon \to 0} \EE_W\left[\pa{\ASUGARA{\mu}(Y ,\theta, \epsilon)}_i\right]} \varphi(\theta) d\theta  \\
	&=& - \int_{\Theta}\pa{\lim_{\epsilon \to 0} \EE_W\left[\pa{\ASUGARA{\mu}(Y ,\theta, \epsilon)}_i\right]} \varphi(\theta) d\theta ~. \\
\end{eqnarray*}
From the definition of weak derivative, we get the claimed result on the asymptotic unbiasedness of
$\ASUGARAn$. The asymptotic unbiasedness of the gradient of the finite difference DOF naturally follows
with the same proof strategy by ignoring the two first terms in the decomposition
$\ASUREA{\mu}(\mu_0, \theta, \epsilon) =
\norm{A(\mu(y, \theta) - y)}^2 - \sigma^2 \tr( A^*\!A )
+
2 \sigma^2
\AEDOFA{\mu}(\mu_0, \theta, \epsilon)$.

We now justify each of the steps (S.1)-(S.8). We denote $g_{1,\sigma}$ the Gaussian probability density function of zero-mean and variance $\sigma^2$, and $g_{\sigma}$ its $P$-dimensional version, i.e.~$g_{\sigma}=(g_{1,\sigma})^P$.

\begin{enumerate}[label=(S.\arabic*)]
  \item This is Stein lemma which applies owing to Assumption~\ref{assump:L1}. Indeed, $\mu(\cdot, \theta)$ is Lipschitz, hence weakly differentiable and its derivative equals its weak derivative Lebesgue a.e.~\cite[Theorem~1-2, Section~6.2]{EvansGariepy92}. Moreover, we have for any $\theta$
\begin{equation}\label{eq:lipmu}
\norm{\mu(y, \theta) - \mu(y', \theta)} \leq L_1 \norm{y - y'} \Rightarrow \abs{\mu_i(y, \theta) - \mu_i(y', \theta)} \leq L_1 \norm{y - y'} ~,
\end{equation}
and thus, whenever the derivatives of $\mu_i(\cdot, \theta)$ exist, they are bounded by $L_1$. Consequently,
\[
\EE_W\left[ \abs{\frac{\partial \mu_i(Y)}{\partial y_i}} \right] \leq L_1 ~,
\]
i.e.~the weak partial derivatives are essentially bounded.
  \item This is just \eqref{eq:limsurefd} and the arguments justifying hold owing to Assumption~\ref{assump:L1}.
    \item Let $f_\epsilon(y,\theta) = \ASUREA{\mu}(y, \theta, \epsilon)$. From \eqref{eq:limsurefd}, $\lim_{\epsilon \to 0} f_\epsilon(y,\theta)=\ASURE{\mu}(y, \theta)$ exists Lebesgue a.e.~Assumptions~\ref{assump:L1}-\ref{assump:mu} give
    \begin{equation*}\label{eq:lipmuId}
    \norm{\mu(y, \theta) - y} \leq \norm{y} + \norm{\mu(y,\theta) - \mu(0,\theta)} \leq (1+L_1)\norm{y} ~.
    \end{equation*}
    Combining this with \eqref{eq:lipmu} leads to
    \begin{align*}
      |f_\epsilon(y,\theta)| &=
      \left| \norm{A(\mu(y, \theta) - y)}^2 - \sigma^2 \tr( A^*\!A )
    	+
    	2 \sigma^2 \frac{1}{\epsilon}
    	\sum_{i=1}^P (A^*\!A (\mu(y + \epsilon e_i, \theta) - \mu(y, \theta)))_i \right|\\
      &\leq
      \norm{A}^2P\sigma^2\pa{(1+L_1)^2\frac{\norm{y}^2}{P\sigma^2} + 1 + 2 L_1} ~.
    \end{align*}
    Note that the bound is independent of $\theta$. Thus
    \begin{multline*}
    \EE_W\left[\norm{A}^2P\sigma^2\pa{(1+L_1)^2\frac{\norm{y}^2}{P\sigma^2} + 1 + 2 L_1}\right] = \\ \norm{A}^2P\sigma^2\pa{(1+L_1)^2\pa{\frac{\norm{\mu_0}^2}{P\sigma^2}+1} + 1 + 2 L_1} < \infty ~.
    \end{multline*}
    This bound together with the fact that $\varphi$ is continuously differentiable with compact support in $V$ means that $f_\epsilon\frac{\partial\varphi}{\partial\theta_i}$ is dominated by an integrable function on $\Yy \times V$. The dominated convergence then applies which yields the claim.
    \item Fubini theorem surely applies in view of the integrability just shown at the end of (S.3).
    \item This is a consequence of Proposition~\ref{prop:sugar} and definition of weak differentiability since $\mu(y,\cdot)$ is Lipschitz continuous independently of $y$.
    \item By definition of $\ASUGARA{\mu}$ in Proposition~\ref{prop:sugar}.
    \item Let $f_\epsilon(y,\theta) = (\ASUGARA{\mu}(y, \theta, \epsilon))_i$ and $h(y,\theta) = \pa{2 \JACs{2}{\mu}(y, \theta)^* A^*\!A (\mu(y, \theta) - y)}_i$. By the translation invariance of the convolution product, we have
    \[
    \EE_W[f_\epsilon(Y,\theta)] = 2 (g_\sigma \ast h(\cdot,\theta))(\mu_0) + 2 \sigma^2\sum_{j=1}^P\frac{g_\sigma(\cdot + \epsilon e_i) - g_\sigma}{\epsilon} \ast \pa{\JACs{2}{\mu}(\cdot, \theta)^* \!A^*\!A e_j}_i(\mu_0) ~.
    \]
    Thus
    \begin{align*}
    \abs{(g_\sigma \ast h(\cdot,\theta))(\mu_0)}
		&\leq 2\int_{\Yy} g_\sigma(y-\mu_0) \abs{\pa{\JACs{2}{\mu}(y, \theta)^* A^*\!A (\mu(y, \theta) - y)}_i} dy \\
\text{\scriptsize{[Assumption~\ref{assump:L2}]}}
		&\leq 2L_2\norm{A}^2\int_{\Yy} g_\sigma(y-\mu_0) \norm{\mu(y, \theta) - y}dy \\
\text{\scriptsize{[Assumptions~\ref{assump:L1}-\ref{assump:mu}]}}
		&\leq 2(1+L_1)L_2\norm{A}^2\int_{\Yy} g_\sigma(y-\mu_0) \norm{y} dy \\
		&\leq 2(1+L_1)L_2\norm{A}^2 \EE_W[\norm{y}] dy \\
\text{\scriptsize{[Jensen inequality]}}
		&\leq 2(1+L_1)L_2\norm{A}^2 \EE_W[\norm{y}^2]^{1/2} dy \\
		&\leq 2(1+L_1)L_2\norm{A}^2 \pa{\norm{\mu_0}^2 + P\sigma^2}^{1/2} < \infty ~. \\
    \end{align*}
    For the second term, we have
    \begin{align*}
    &\abs{\frac{g_\sigma(\cdot + \epsilon e_i) - g_\sigma}{\epsilon} \ast \pa{\JACs{2}{\mu}(\cdot, \theta)^* \!A^*\!A e_j}_i(\mu_0)} \\
    		&\leq \int_{\Yy}\abs{\frac{g_\sigma(y-\mu_0 + \epsilon e_i) - g_\sigma(y-\mu_0)}{\epsilon}} \abs{\pa{\JACs{2}{\mu}(y, \theta)^* \!A^*\!A e_j}_i} dy\\
\text{\scriptsize{[Assumption~\ref{assump:L2}]}}
		&\leq L_2\norm{A}^2\int_{\Yy}\abs{\frac{g_\sigma(y-\mu_0 + \epsilon e_i) - g_\sigma(y-\mu_0)}{\epsilon}} dy \\
		&\leq L_2\norm{A}^2\int_{\RR}\abs{\frac{g_{1,\sigma}(t-(\mu_0)_i + \epsilon) - g_{1,\sigma}(t-(\mu_0)_i)}{\epsilon}} dt \\
\text{\scriptsize{[Taylor]}}
		&\leq L_2\norm{A}^2\int_{\RR}\int_{0}^1\abs{g'_{1,\sigma}(t-(\mu_0)_i+\tau)} dtd\tau \\
\text{\scriptsize{[Fubini]}}
		&\leq L_2\norm{A}^2\int_{\RR}\abs{g'_{1,\sigma}(t)} dt < \infty ~. \\
    \end{align*}
    In view of these bounds, and since $\varphi$ is compactly supported in $V$, integrability of $f_\epsilon\varphi$ on $\Yy \times V$ is ensured, whence the claimed result follows.
    \item Let $f_\epsilon$ defined as in (S.7). We have just shown that the integrand in $\theta$, i.e.~$\EE_W[f_\epsilon(Y,\cdot))_i]\varphi$, is dominated by a function that is integrable on $V$. It remains to check that its limit exists Lebesgue a.e. But this is yet again an application of the dominated convergence theorem to the sequence $f_\epsilon$ as an integrand with respect to the Gaussian measure $g_\sigma(y) dy$, which allows to deduce that $\lim_{\epsilon \to 0} \EE_W[f_\epsilon(Y,\theta)\varphi(\theta)] = \EE_W[\lim_{\epsilon \to 0} f_\epsilon(Y,\theta)\varphi(\theta)]$.
    \end{enumerate}
This completes the proof.
\end{proof}

\begin{proof}[Proposition~\ref{prop:st_stats}]
  For a fixed $\lambda$, it can be shown similarly to \cite[Theorem~4(iii), Section~4.2.2]{EvansGariepy92}, that $\ST(\cdot,\lambda)$ is weakly differentiable and that its weak Jacobian $h(y) = \JACs{2}{\ST}(y, \lambda)$ is diagonal, with diagonal elements, for
  $1 \leq i \leq P$,
  \begin{align}\label{eq:diff_st}
    h(y)_{i} &= \choice{
      +1 & \text{if}\quad y_i \leq -\lambda\\
      0 & \text{if}\quad -\lambda < y_i < \lambda\\
      -1 & \text{otherwise}
    }~.
  \end{align}
  We next define, for a fixed $\lambda$, the quantity
  $h'(y, \epsilon) = \GRAD{2}{\EDOFA{\ST}}(y, \lambda, \epsilon)$.
  Using Proposition \ref{prop:sugar} and the fact that $\epsilon < 2\lambda$ give
  \begin{align*}
    h'(y, \epsilon)
    =
    \sum_{i=1}^P \frac{h(y + \epsilon e_i)_i - h(y)_i}{\epsilon} =
    \sum_{i=1}^P
    \choice{
      0 & \text{if}\quad y_i < -\lambda - \epsilon\\
      -1/\epsilon & \text{if}\quad -\lambda - \epsilon < y_i < -\lambda\\
      0 & \text{if}\quad -\lambda < y_i < \lambda - \epsilon\\
      -1/\epsilon & \text{if}\quad \lambda - \epsilon < y_i < \lambda\\
      0 & \text{if}\quad \lambda < y_i\\
    }~.
  \end{align*}
  Computing the expectation and the variance of $h'(Y, \epsilon)$ in closed-form
  with truncated Gaussian statistics,
  and using the fact that $h$ is separable in its arguments,
  give the proposed formula.
\end{proof}

\begin{proof}[Theorem~\ref{thm:cons_sugar}]
  For $P$ big enough, $\hat{\epsilon}(P) < 2\lambda$ since $\lim_{P \to \infty} \hat{\epsilon}(P) = 0$.
  Using the notations in the proof of Lemma \ref{prop:st_stats} leads to
  \begin{align*}
    & \SUGARA{\ST}(y, \lambda, \epsilon) =
    2 h(y)^* (\ST(y, \lambda) - y) + 2 \sigma^2 h'(y, \hat{\epsilon}(P))~.
  \end{align*}
  The Cauchy-Schwartz inequality implies that
  \begin{align*}
  \VV_W \left[ \frac{1}{P} \SUGARA{\ST}(Y, \lambda, \epsilon) \right]^{1/2}
  \leq \quad &
  2 \VV_W \left[ \frac{1}{P} h(y)^* (\ST(Y, \lambda) - Y) \right]^{1/2}\\
  \hspace{2cm}
  &
  \hspace{2cm}
  +
  2 \sigma^2 \VV_W \left[ \frac{1}{P} h'(Y, \hat{\epsilon}(P)) \right]^{1/2}~.
  \end{align*}
  Since $x \mapsto \sqrt{\pi}\erf\pa{x/a}$ is Lipschitz continuous
  with a constant of $2/a$, Lemma \ref{prop:st_stats} yields
  \begin{align*}
    \VV_W \!\!\left[ \frac{1}{P}  h'(Y, \hat{\epsilon}(P)) \right]
    &\leq
    \frac{\sqrt{2}}{\sqrt{\pi}\sigma P\hat{\epsilon}(P)}
    ~.
  \end{align*}
By assumption, we have $\lim_{P \to \infty} P^{-1}\hat{\epsilon}(P)^{-1} = 0$
and then the variance of $\frac{1}{P} h'(Y, \hat{\epsilon}(P))$ vanishes to zero.
Next, remark that
\begin{align*}
  h(y)^* (\ST(y, \lambda) - y)
  =
  \lambda \#\{|y| > \lambda\}
\end{align*}
where $\#\{|y| > \lambda\}$ denotes
the number of entries of $|y|$ greater than $\lambda$.
We have $\#\{|Y_i| > \lambda\} \sim_{\mathrm{iid}} \mathrm{Bernoulli}(p_i)$
whose variance is $p_i (1 - p_i)$,
where
$p_i = \frac{1}{2}\left(
  \erf\pa{\frac{(\mu_0)_i + \lambda}{\sqrt{2}\sigma}}
  -
  \erf\pa{\frac{(\mu_0)_i - \lambda}{\sqrt{2}\sigma}}
\right)$.
It follows that
$\VV_W [ \#\{|Y| > \lambda\} ] = \sum_{i=1}^P p_i (1 - p_i) \leq P$ and hence
\begin{align*}
  \lim_{P \to \infty} \VV_W \left[ \frac{1}{P} h(Y)^* (\ST(Y, \lambda) - Y) \right]
  =
  \lim_{P \to \infty} \VV_W \left[ \frac{1}{P} \lambda \#\{|Y| > \lambda\} \right]
  =
  0~.
\end{align*}
Consistency (i.e.~convergence in probability)
follows from traditional arguments by invoking Chebyshev inequality and
using asymptotic unbiasedness (Theorem \ref{thm:sugar}) and vanishing variance.
\end{proof}

\begin{proof}[Proposition~\ref{prop:st_mse}]
  Developing the mean squared error in terms of bias and variance give the
  first part of the proposition. Lemma \ref{prop:st_stats} and the fact that
  $\lim_{\epsilon \to 0} \EE_W \GRAD{2}{\EDOF{\ST}}(Y, \lambda, \epsilon) =
  \GRAD{2}{\DOF{\ST}}(\mu_0, \lambda)$ concludes the second part.
\end{proof}

%% file: sections/proofs/proofs_extra.tex
\section{Regularity of the proximal operator of a gauge}\label{sec:reggauge}
We first provide a glimpse of gauges.

\begin{defn}[Gauge]\label{defn:convex-gauge}
  Let $\Cc$ be a non-empty closed convex set containing the origin.
  The \emph{gauge} of $\Cc$ is the function $\gamma_{\Cc}$ defined by
  \begin{equation*}
    \gamma_{\Cc}(y) =
    \inf \enscond{\omega > 0}{y \in \omega \Cc} .
  \end{equation*}
  As usual, $\gamma_{\Cc}(y) = + \infty$ if the infimum is not attained.
\end{defn}

\begin{defn}[Polar set]\label{defn:convex-polarset}
  Let $\Cc$ be a non-empty convex set. The set $\Cc^\circ$ given by
  \begin{equation*}
    \Cc^\circ = \enscond{z \in \RR^N}{\dotp{z}{x} \leq 1 \text{ for all } x \in \Cc}
  \end{equation*}
  is called the \emph{polar} of $\Cc$. $\Cc^\circ$ is a non-empty closed convex set containing the origin, and if $\Cc$ is closed and contains the origin as well, $\Cc^{\circ\circ}=\Cc$.
\end{defn}
  
We now summarize some key properties that will be needed in the main proof.
\begin{lem}\label{lem:gaugeprop}
Let $\Cc$ be a non-empty closed convex set containing the origin. The following assertions hold.
\begin{enumerate}[label={(\roman{*})}, ref=]
\item $\gamma_{\Cc}$ is a non-negative, closed, convex and positively homogenenous function.
\item $\Cc$ is the unique closed convex set containing the origin such that
\[
\Cc = \enscond{y \in \Yy}{\gamma_{\Cc}(y) \leq 1} .
\]
\item $\gamma_{\Cc}$ is bounded and coercive if, and only if, $\Cc$ is compact and contains the origin as an interior point.
\item The gauge of $\Cc$ and the support function $\sigma_{\Cc^\circ}(y)=\max_{z \in \Cc^\circ} \dotp{y}{z}$ coincide, i.e.
\begin{equation*}
  \gamma_{\Cc} = \sigma_{\Cc^\circ} ~.
\end{equation*}
\end{enumerate}
\end{lem}

\begin{proof}
(i)-(ii) follow from \cite[Theorem~V.1.2.5]{hiriart1996convex}. (iii) is a consequence of \cite[Theorem~V.1.2.5(ii) and Corollary~V.1.2.6]{hiriart1996convex}. (iv) \cite[Proposition~V.3.2.4]{hiriart1996convex}.
\end{proof}

We are now equiped to prove our regularity result.
\begin{prop}\label{prop:lipproxgauge}
Let $\Cc$ be a compact convex set containing the origin as an interior point, i.e. a convex body, and $G=\gamma_\Cc$ is its gauge. For any $\theta > 0$, $\theta' > 0$ and any $y \in \Yy$, the following holds
\[
\norm{\Prox_{\theta G}(y) - \Prox_{\theta' G}(y)} \leq L_2 |\theta - \theta'|
\]
for some constant $L_2 > 0$ independent of $y$, i.e. for any $y$, $\theta \mapsto \Prox_{\theta G}(y)$ is Lipschitz continuous on $]0,+\infty[$.
\end{prop}

\begin{proof}
From \cite[Proposition 2.3(ii)]{AttouchNewton11}, we have that for any $y$, the function $\theta \mapsto \Prox_{\theta G}(y)$ is such that
\begin{equation}
\label{eq:lipprox}
\norm{\Prox_{\theta G}(y) - \Prox_{\theta' G}(y)} \leq |\theta - \theta'|\norm{y - \Prox_{\theta G}(y)}/\theta ~.
\end{equation}
Now, we have 
\begin{equation}
\label{eq:homogauge}
\theta G(y) = \gamma_{\Cc/\theta}(y) = \sigma_{\theta \Cc^\circ}(y) ~,
\end{equation}
where the first equality follows from positive homogeneity (Lemma~\ref{lem:gaugeprop}(i)) and Definition~\ref{defn:convex-gauge}, and the second equality is a consequence of Lemma~\ref{lem:gaugeprop}(iv) and polarity. 

Applying Moreau identity, we get that
\[
y - \Prox_{\theta G}(y) = y - \Prox_{\sigma_{\theta \Cc^\circ}}(y) = \Proj_{\theta \Cc^\circ}(y) ~.
\]
By virtue of Lemma~\ref{lem:gaugeprop}(iii), there exists a constant $L_2 > 0$, independent of $y$, such that\footnote{The constant $L_2$ can be given explicitly by bounding from below the support function of the inscribed ellipsoid of maximal volume, the so-called John ellipsoid. For symmetric convex bodies, $L_2$ can be made tightest possible. For simplicity, we avoid delving into these technicalities here.}
\begin{equation*}
\label{eq:equivnormgauge}
\norm{y - \Prox_{\theta G}(y)} = \norm{\Proj_{\theta \Cc^\circ}(y)} \leq L_2 \gamma_{\Cc^\circ}\bpa{\Proj_{\theta \Cc^\circ}(y)} ~.
\end{equation*}
Applying \eqref{eq:homogauge} to $\gamma_{\Cc^\circ}$, we get
\begin{equation}
\label{eq:normproxgaugebnd}
\norm{y - \Prox_{\theta G}(y)} \leq L_2 \theta\gamma_{\theta\Cc^\circ}\bpa{\Proj_{\theta \Cc^\circ}(y)} \leq L_2\theta ~,
\end{equation}
where the last inequality follows from Lemma~\ref{lem:gaugeprop}(ii) since obviously $\Proj_{\theta \Cc^\circ}(y) \in \theta\Cc^\circ$. Combining \eqref{eq:lipprox} and \eqref{eq:normproxgaugebnd}, we get the desired result.
\end{proof}

\begin{cor}\label{cor:lipcompproxgauge}
Let $\Cc_i$, $i=1,\ldots,m$, be compact convex sets containing the origin as an interior point, i.e. convex bodies, and $G_i=\gamma_{\Cc_i}$ the associated gauges. For any $\theta$, $\theta' \in ]0,+\infty[^m$, and any $y \in \Yy$, the following holds
\[
\norm{\Prox_{\theta_1 G_1} \circ \cdots \circ\Prox_{\theta_m G_m}(y) - \Prox_{\theta'_1 G_1} \circ \cdots \circ\Prox_{\theta'_m G_m}(y)} \leq \sqrt{m}\max_i L_{2,i} \norm{\theta - \theta'}
\]
where $L_{2,i} > 0$ is the same Lipschitz constant associated to $\Cc_i$ given in Proposition~\ref{prop:lipproxgauge}.
\end{cor}

\begin{proof}
Using repeatedly the triangle inequality, Proposition~\ref{prop:lipproxgauge} and the fact that the mapping $y \mapsto \Prox_{\theta_i G_i}(y)$ is 1-Lipschitz \cite{rockafellar1976monotone}, we obtain
\begin{align*}
&\norm{\Prox_{\theta_1 G_1} \circ \cdots \circ \Prox_{\theta_m G_m}(y) - 
\Prox_{\theta'_1 G_1} \circ \cdots \circ\Prox_{\theta'_m G_m}(y)} = \\ 
&\Big\|\pa{\Prox_{\theta_1 G_1} \circ \Prox_{\theta_2 G_2} \circ \cdots \circ\Prox_{\theta_m G_m}(y) - \Prox_{\theta'_1 G_1} \circ \Prox_{\theta_2 G_2} \circ \cdots \circ \Prox_{\theta_m G_m}(y)} + \\
&\pa{\Prox_{\theta'_1 G_1} \circ \Prox_{\theta_2 G_2} \circ \cdots \circ\Prox_{\theta_m G_m}(y) - \Prox_{\theta'_1 G_1} \circ \Prox_{\theta'_2 G_2} \circ \cdots \circ \Prox_{\theta_m G_m}(y)}\Big\| \\
&\leq L_{2,1} |\theta_1 - \theta_1'| + \norm{\Prox_{\theta_2 G_2} \circ \cdots \circ \Prox_{\theta_m G_m}(y) - 
\Prox_{\theta'_2 G_2} \circ \cdots \circ\Prox_{\theta'_m G_m}(y)} \\
&\leq \sum_i L_{2,i} |\theta_i - \theta_i'| \leq \max_{i} L_{2,i} \norm{\theta-\theta'}_1 \leq \sqrt{m}\max_{i} L_{2,i} \norm{\theta-\theta'}
\end{align*}
as claimed.
\end{proof}

%% file: sections/proofs/proofs_differentiation.tex
\section{Proofs of Section~\ref{sec:formal_diff}}\label{sec:proofs_formal_diff}

\begin{proof}[Proposition~\ref{prop:gen_diff_y}]
  Since \eqref{eq-iterations} is the composition of Lipschitz continuous mappings of $y$ by assumption, applying the chain rule~\cite[Theorem~4 and Remark, Section~4.2.2]{EvansGariepy92} gives the formula.
\end{proof}

\begin{proof}[Proposition~\ref{prop:gen_diff_theta}]
   The argument is exactly the same as that for Proposition~\ref{prop:gen_diff_y} replacing $y$ by $\theta$ where the required Lipschitz continuity assumptions w.r.t. $\theta$ hold true.
\end{proof}

\begin{proof}[Corollary~\ref{cor:gfb_diff_y}]
We first notice that
$\il{\Dd}_a = (\il{\Dd_{\xi}}, \il{\Dd_{z_1}}, \ldots, \il{\Dd_{z_Q}})$
where
$\il{\Dd_{\xi}} = \JACs{1}{\il{\xi}}(y, \theta)[\de]$ and
$\il{\Dd_{z_k}} = \JACs{1}{\il{z_k}}(y, \theta)[\de]$.
Hence, applying again the chain rule \cite[Theorem~4 and Remark, Section~4.2.2]{EvansGariepy92} to the the sequence of iterates and
using the fact that all involved mappings are Lipschitz, and
$\il{\Dd_x} = \il{\Gamma_a}(\il{\Dd_a}) = \il{\Dd_{\xi}}$
concludes the proof.
\end{proof}

\begin{proof}[Corollary~\ref{cor:gfb_diff_theta}]
Observe that
$\il{\Jj}_a = (\il{\Jj_{\xi}}, \il{\Jj_{z_1}}, \ldots, \il{\Jj_{z_Q}})$
where  $\il{\Jj_\xi} = \JACs{2}{\il{\xi}}(y, \theta)$ and
$\il{\Jj_{z_k}} = \JACs{2}{\il{z_k}}(y, \theta)$.
Arguing as in the proof of Corollary~\ref{cor:gfb_diff_y}, using now that
$\il{\Jj_{x}} = \il{\Gamma_a}(\il{\Jj}_a) = \il{\Jj_{\xi}}$ yields the formula.
\end{proof}

\begin{proof}[Corollary~\ref{cor:pds_diff_y}]
As before, but now with
$\il{\Dd}_a = (\il{\Dd_{\xi}}, \il{\Dd_{\tilde x}}, \ldots, \il{\Dd_{u}})$
where
$\il{\Dd_{\xi}} = \JACs{1}{\il{\xi}}(y, \theta)[\de]$,
$\il{\Dd_{\tilde x}} = \JACs{1}{\il{\tilde x}}(y, \theta)[\de]$,
$\il{\Dd_{u}} = \JACs{1}{\il{u}}(y, \theta)[\de]$,
and
$\il{\Dd_x} = \il{\Gamma_a}(\il{\Dd_a}) = \il{\Dd_{\xi}}$.
The chain rule completes the proof.
\end{proof}

\begin{proof}[Corollary~\ref{cor:pds_diff_theta}]
As before, but now with
$\il{\Jj}_a = (\il{\Jj_{\xi}}, \il{\Jj_{\tilde x}}, \ldots, \il{\Jj_{u}})$
where
$\il{\Jj_{\xi}} = \JACs{1}{\il{\xi}}(y, \theta)$,
$\il{\Jj_{\tilde x}} = \JACs{1}{\il{\tilde x}}(y, \theta)$,
$\il{\Jj_{u}} = \JACs{1}{\il{u}}(y, \theta)$
and
$\il{\Jj_x} = \il{\Gamma_a}(\il{\Jj_a}) = \il{\Jj_{\xi}}$.
The chain rule completes the proof.
\end{proof}